\documentclass[leqno,english]{amsart}
\usepackage{amsfonts,amssymb,amsmath,amsgen,amsthm}
\usepackage{bbm}
\usepackage{mathrsfs}
\usepackage{hyperref}
\usepackage{xcolor}
\newcommand{\msc}[2][2000]{%
  \let\@oldtitle\@title%
  \gdef\@title{\@oldtitle\footnotetext{#1 \emph{Mathematics subject
        classification.} #2}}% 
}

\theoremstyle{plain}
\newtheorem{theorem}{Theorem} [section]
\newtheorem{definition}[theorem]{Definition}

\newtheorem{lemma}[theorem]{Lemma}
\newtheorem{corollary}[theorem]{Corollary}
\newtheorem{proposition}[theorem]{Proposition}

\theoremstyle{remark}
\newtheorem{remark}[theorem]{Remark}

\def\R{{\mathbb R}}% real numbers
\def\N{{\mathbb N}}% nonnegative integers
\def\Sch{{\mathcal S}}% Schwartz space
\def\O{\mathcal O}
\def\F{\mathcal F}

\def\En{\mathcal E}

\def\u{\mathbf u}

\def\cont{\mathcal C}

\def\1{\mathbbm{1}}

\def\({\left(}
\def\){\right)}
\def\<{\left\langle}
\def\>{\right\rangle}
\def\le{\leqslant}
\def\ge{\geqslant}

\def\Eq#1#2{\mathop{\sim}\limits_{#1\rightarrow#2}}
\def\Tend#1#2{\mathop{\longrightarrow}\limits_{#1\rightarrow#2}}

\def\d{{\partial}}
\def\eps{\varepsilon}

\def\si{{\sigma}}

\DeclareMathOperator{\RE}{Re}
\DeclareMathOperator{\IM}{Im}
\DeclareMathOperator{\diver}{div}

\newcommand{\enstq}[2]{\left\{#1~\middle|~#2\right\}}
\newcommand{\dd}{\mathrm{d}}

\numberwithin{equation}{section}

\begin{document}
\title[Dependence of NLS on the power of the nonlinearity]{On the
  dependence of the nonlinear Schr\"odinger flow  
upon the power of the nonlinearity} 
\author[R. Carles]{R\'emi Carles}
\address{CNRS\\ IRMAR - UMR
  6625\\ F-35000 Rennes, France}
\email{Remi.Carles@math.cnrs.fr}

\author[Q. Chauleur]{Quentin Chauleur}
\address{Univ. Lille, CNRS, Inria, UMR 8524 - Laboratoire Paul Painlevé, F-59000 Lille, France}
\email{quentin.chauleur@inria.fr}

\author[G. Ferriere]{Guillaume Ferriere}
\address{Univ. Lille, CNRS, Inria, UMR 8524 - Laboratoire Paul Painlevé, F-59000 Lille, France}
\email{guillaume.ferriere@inria.fr}
\begin{abstract}
We prove continuity properties for the flow map associated to the
defocusing energy-subcritical power-like nonlinear Schr\"odinger
equation, when the power varies. We show local in time continuity in
the energy space for any power, and global in time continuity for
sufficiently large powers. When the linear dispersive rate is
counterbalanced by a time-dependent rescaling, we show a uniform in
time continuity of the squared 
modulus of this rescaled function, in Kantorovich distance, for any
power, including long range cases in terms of scattering.  The most
difficult result addresses the convergence of suitably renormalized
solutions to the solution of the logarithmic Schr\"odinger equation,
when the power goes to zero, uniformly in time, in Kantorovich
distance. The proof relies on estimates for perturbed porous
medium equations, involving the harmonic Fokker-Planck
operator. 
\end{abstract}
\thanks{A CC-BY public
copyright license has been applied by the authors to the present
document and will be applied to all subsequent versions up to the
Author Accepted Manuscript arising from this submission. }  
\maketitle

\section{Introduction}
\label{sec:intro}

\subsection{Setting}
\label{sec:setting}

We consider the Cauchy problem associated to the defocusing nonlinear Schr\"odinger equation with power-like nonlinearity,
\begin{equation}
  \label{eq:NLS}
  i\d_t u +\frac{1}{2}\Delta u= |u|^{2\si}u\quad ;\quad u_{\mid t=0}=\phi,
\end{equation}
for $x\in \R^d$, $d \ge 1$, in the  energy-subcritical case, 
$0<\si<\frac{2}{(d-2)_+}$. For clarity, we denote by $u_\sigma$ the
above solution, as the main goal is to understand the dependence of
$u_\sigma$ with respect to $\sigma$, locally in time and, especially,
globally in time.

\subsubsection{Power nonlinearity}

Several papers have considered the continuity of the flow map with
respect to the initial data, see
e.g. \cite{CW90,CFH11,Kato95,Kato95corr} for positive results, and 
\cite{KPV01,CCT,CDS12,CaGa25} for lack of continuity (according to the
function space considered). In the present paper, we address the
dependence of the flow map $\phi\mapsto u$ with respect to the nonlinearity, that is
with respect to $\si$. The initial value $\phi$ may depend on $\si$ as well. We
consider a defocusing nonlinearity in order to ensure that the
solution is defined globally in time, and enjoys a dispersive
behavior. In the focusing case, dynamical properties may be different.
\smallbreak

Since we consider energy-subcritical nonlinearities, it is reasonable
to expect that at least locally in time, the flow map is continuous in
$H^1$. This is indeed what we prove in Theorem~\ref{theo:local} by
using Strichartz estimates. The question of the interplay with the
large time behavior is less obvious.
\smallbreak

When $\si>1/d$, the solution to \eqref{eq:NLS} is asymptotically linear, in
the sense that there exists $u_+\in L^2(\R^d)$ such that
\begin{equation}\label{eq:asym-lin}
  \|u(t)-e^{i\frac{t}{2}\Delta} u_+\|_{L^2(\R^d)}\Tend t \infty 0, 
\end{equation}
see e.g. \cite{BGTV23} and references therein. We show that when $\si$
is sufficiently large ($\si>\si_0(d)>1/d$), the flow map is continuous
with respect to $\si$, 
uniformly in time (Theorem~\ref{theo:global}), which implies that the
nonlinear scattering map is continuous with respect to $\si$
(Corollary~\ref{cor:scattering}). 
\smallbreak

It is well known that as soon as $\si\le 1/d$, \eqref{eq:asym-lin} can
hold only in the trivial case of the zero solution, $\phi=u=u_+=0$, as
proven in \cite{Barab}. In the critical case $\si=1/d$, as established
initially in \cite{Ozawa91,GO93,HN98}, the large time behavior of $u$
is described at leading order by a nonlinear phase modification of the
free dynamics. If $d=1$ and $\phi$ is sufficiently small, smooth and
localized, there exists $u_+\in L^2(\R)$ such that 
\begin{equation*}
  \left\|u(t)-e^{i\Phi_+(t)}e^{i\frac{t}{2}\Delta} u_+\right\|_{L^2(\R^d)}\Tend t
  \infty 0, \quad \text{where }\Phi_+(t,x) = -\left\lvert \widehat
    u_+\(\frac{x}{t}\)\right\rvert^2\ln t,
\end{equation*}
and the Fourier transform (on $\R^d$) is normalized as 
\begin{equation*}
  \widehat f(\xi)=\frac{1}{(2\pi)^{d/2}}\int_{\R^d}f(x)e^{-ix\cdot \xi}\dd
    x. 
\end{equation*}
This implies for instance that the flow map cannot be
continuous uniformly in time at $\si=1/d$ (as long range effects are
absent for $\si>1/d$, see \cite{BGTV23}). The general consensus is
that even for $\si<1/d$, modified scattering should be characterized
at leading order by a (nonlinear) phase modification, even
though no general proof of this fact seems to be available so far. Our
approach supports this train of thought: instead of considering $u$,
we first introduce a time dependent rescaling which counterbalances
the linear dispersive behavior, and we consider the squared modulus of this new
function, that is
\begin{equation*}
  \<t\>^d |u_\si(t,x\<t\>)|^2,\quad\text{where}\quad
  \<t\>=\sqrt{1+t^2}.  
\end{equation*}
Dividing this quantity by $\|\phi\|_{L^2}^2$ yields a time dependent
family of probability densities $(\varrho_\si(t))_{t\in \R}$. We show that this
family is continuous with respect to $\si$, uniformly in time $t\in
\R$, at any $\si\in (0,\frac{2}{(d-2)_+})$, in Wasserstein distance
$W_1$ (Theorem~\ref{theo:W1}), which implies uniform continuity in some
negative order Sobolev 
spaces  (Corollary~\ref{cor:interpolation}). In
particular, this property is  
insensitive to long range effects (when $\si\le 1/d$).

\subsubsection{Logarithmic nonlinearity}

In the limit $\si\to 0$, at points where $u\not =0$,
\begin{equation*}
  |u|^{2\si}=\exp \(\si \ln|u|^2\) = 1 + \si \ln|u|^2+\O(\si^2). 
\end{equation*}
Up to a gauge transform ($u\to u e^{it}$), we may replace the
nonlinearity in \eqref{eq:NLS} with $\(|u|^{2\si}-1\)u$. We may then
proceed like in \cite{GMS-p} where the stationary, focusing case is
considered, and let $\si\to 0$ in  
\begin{equation}
  \label{eq:rescaledNLS}
    i\d_t \u_\si +\frac{1}{2}\Delta \u_\si = \frac{1}{\si}\(|\u_\si|^{2\si}-1\)\u_\si\quad
    ;\quad \u_{\si\mid t=0}=\phi_\si.
  \end{equation}
Formally, if $\phi_\si\to \phi$ as $\si\to 0$,
$\u_\si$ is expected to converge to the solution of the 
logarithmic Schr\"odinger equation
\begin{equation}
  \label{eq:logNLS}
  i\d_t \u +\frac{1}{2}\Delta \u =\u\ln(|\u|^2)\quad ;\quad \u_{\mid t=0}=\phi.
\end{equation}
We denote by $\u_{\rm log}$ or $\u_0$ this solution. This equation was
introduced in \cite{BiMy76}, and the mathematical study of the Cauchy
problem \eqref{eq:logNLS} started in \cite{CaHa80} (with a different
sign in front of the nonlinearity though).
\begin{remark}
  One may adopt an alternative point of view. If $u$ solves \eqref{eq:NLS}, then
  \begin{equation*}
    \tilde u_\si(t,x) = \si^{1/(2\si)} u(t,x) e^{it/\si}
   \end{equation*}
   solves the PDE in \eqref{eq:rescaledNLS}, but with initial value
   $\si^{1/(2\si)} \phi$. The fact that the factor $ \si^{1/(2\si)} $
   is unbounded as $\si\to 0$
   is an alternative evidence that nonlinear effects must be enhanced
   in the nonlinear Schr\"odinger equation in order to get some
   convergence toward the logarithmic Schr\"odinger equation. 
 \end{remark}
 \begin{remark}
   If the continuity of the flow map is considered at $\si>0$
   (instead of $\si=0$), then it is essentially equivalent to consider
   \eqref{eq:NLS} or \eqref{eq:rescaledNLS}, since $\si\mapsto 1/\si$
   is continuous on $(0,\infty)$.  
 \end{remark}
Several aspects indicate
that the limit $\si\to 0$ in \eqref{eq:rescaledNLS} is more involved
than in the previous case. Indeed, it was proven in \cite{CaGa18} that
solutions to \eqref{eq:logNLS} obey an enhanced dispersive rate:
roughly speaking, instead of a decay rate of order $t^{-d/2}$ in
\eqref{eq:NLS}, solutions to \eqref{eq:logNLS} decay like
$\tau_0(t)^{-d/2}\approx t^{-d/2}\(\ln t\)^{-d/4}$. Moreover, rescaling the solution to
\eqref{eq:logNLS}  
to compensate this new dispersive rate, the large time behavior
exhibits a universal profile,
\begin{equation*}
  \tau_0(t)^d|\u_0(t,x\tau_0(t))|^2\Tend t \infty
  \|\phi\|_{L^2}^2\gamma^2 \text{ in Wasserstein distance }W_2,
\end{equation*}
where $\gamma(x)=e^{-|x|^2/2}$ is independent of $\phi$. 
Finally, the $H^1$-norm of any nonzero solution to
\eqref{eq:logNLS} is unbounded in the large time limit. These three
properties are in sharp contrast with the dynamical properties
associated to \eqref{eq:NLS}.

\subsection{Main results}
\label{sec:results}
Denote 
\begin{equation*}
  \Sigma= \left\{ f\in H^1(\R^d),\ \int_{\R^d}|x|^2
    |f(x)|^2\dd x<\infty\right\}, 
\end{equation*}
equipped with the norm
\begin{equation*}
  \|f\|_{\Sigma} = \|f\|_{L^2(\R^d)}+ \|\nabla
  f\|_{L^2(\R^d)}+\|xf\|_{L^2(\R^d)}. 
\end{equation*}

\subsubsection{Power nonlinearity}
Our first  result addresses the continuity of the flow map on
bounded time intervals:
\begin{theorem}[Local in time continuity]\label{theo:local}
  Let $0<\si<\frac{2}{(d-2)_+}$ and $(\phi_\nu)_{|\nu-\si|\le \eps}$
  bounded in
  $H^1(\R^d)$ for some $\eps>0$ such that $0<\si-\eps$ and $\si+\eps<
  \frac{2}{(d-2)_+}$. Denote by $u_\nu$ the 
  corresponding 
  solutions to the nonlinear Schr\"odinger equation
  \eqref{eq:NLS}, $u_\nu\in \cont(\R;H^1(\R^d))$. Let $T>0$:
  \begin{itemize}
  \item There exists $C$ such that, as $\nu\to \si$,
    \begin{align*}
      &\sup_{t\in [-T,T]}\|u_\nu(t)- u_\si(t)\|_{L^2(\R^d)}\le
        C\|\phi_\nu-\phi_\si\|_{L^2(\R^d)}+ C|\nu-\si|,\\
      & \sup_{t\in [-T,T]}\|u_\nu(t)- u_\si(t)\|_{H^1(\R^d)}\le
        C\|\phi_\nu-\phi_\si\|_{H^1(\R^d)}+C|\nu-\si|^\theta,
    \end{align*}
    for some $\theta>0$ (given by \eqref{eq:exp_theta_1}-\eqref{eq:exp_theta_2}), which can be taken equal to one if $2\si>1$.
  \item If in addition $(\phi_\nu)_{|\nu-\si|\le \eps}$ is
  bounded in $\Sigma$, then
    \begin{equation*}
      \sup_{t\in [-T,T]}\|u_\nu(t)- u_\si(t)\|_{\Sigma}\le
     C\|\phi_\nu-\phi_\si\|_{\Sigma}+   C|\nu-\si|^\theta,
      \end{equation*}
      for the same $\theta$ as above. 
  \end{itemize}
\end{theorem}

In the case where the power $\si$ is sufficiently large, the previous
continuity can be made uniform in time:
\begin{theorem}[Global continuity in the (very) short range
  case]\label{theo:global} 
  Let $\si_0(d)<\si<\frac{2}{(d-2)_+}$ and $(\phi_\nu)_{|\nu-\si|\le
    \eps}$  bounded in $\Sigma$ with $\eps>0$ such that
  $\si-\eps>\si_0(d)$ and $\si+\eps<\frac{2}{(d-2)_+}$, where 
  \begin{equation*}
   \si_0(d)= \frac{2-d+\sqrt{d^2+12d+4}}{4d}\in
    \(\frac{1}{d},\frac{2}{d}\). 
  \end{equation*}
  Then we have the global in time estimate:
  \begin{equation*}
    \sup_{t\in \R}\left\|e^{-i\frac{t}{2}\Delta}\(u_\nu(t)-
      u_\si(t)\)\right\|_{\Sigma}\le
     C\|\phi_\nu-\phi_\si\|_{\Sigma}+   C|\nu-\si|^\theta,
   \end{equation*}
    for the same $\theta$ as in  Theorem~\ref{theo:local}.  
\end{theorem}
We briefly recall the notion of nonlinear scattering operator. For a
given asymptotic state $u_-$, consider the equation \eqref{eq:NLS}
where the initial value is replaced by a final value (at infinite time),
\begin{equation*}
  e^{-i\frac{t}{2}\Delta}u(t)\Big|_{t=-\infty}=u_-. 
\end{equation*}
If this new problem has a solution defined up to $t=0$, then the map
$W_-:u_-\mapsto u_{\mid t=0}$ is called wave operator. Conversely, if the
solution to the Cauchy problem \eqref{eq:NLS} behaves asymptotically
linearly,
\begin{equation*}
  u(t)\Eq t {+\infty} e^{i\frac{t}{2}\Delta}u_+,
\end{equation*}
for some asymptotic state $u_+$, then $u_+=W_+^{-1}u_{\mid t=0}$ and
the scattering operator $S=W_+^{-1}\circ W_-$ is the map $u_-\mapsto
u_+$. This map is well defined on $\Sigma$ for $\si>\si_0(d)$, as
established initially in \cite{GV79Scatt} (see also \cite{CW92,NakanishiOzawa}). In view of the present
context, we emphasize its dependence upon $\si$ by using the notation
$S_\si$. 
\begin{corollary}[Continuity of the scattering
  operator]\label{cor:scattering} 
  Let $\si_0(d)<\si<\frac{2}{(d-2)_+}$ and
  $\(u_{-,_\nu}\)_{|\nu-\si|\le \eps}$ be a family in $\Sigma$ for
  some $\eps>0$. The
  scattering operator is continuous at $\si$ in the sense 
  that if
  \begin{equation*}
    \|u_{-,\si}-u_{-,_\nu}\|_\Sigma\Tend \nu \si 0,
  \end{equation*}
  then
  \begin{equation*}
    \|S_\si u_{-,\si}-S_\nu u_{-,_\nu}\|_\Sigma=\
    \|u_{+,\si}-u_{+,_\nu}\|_\Sigma \Tend \nu \si 0, 
  \end{equation*}
\end{corollary}

For $\mu_1$ and $\mu_2$ probability measures on $\R^d$,  the
Wasserstein (or Kantorovich–Rubinstein) distance $W_1$ can be
characterized by
\begin{equation}\label{eq:W1}
  W_1(\mu_1,\mu_2):= \sup \enstq{\int_{\R^d} \psi \dd (\mu_1 -
    \mu_2)}{ \psi \in \cont(\R^d)  \text{ and }  \|\psi\|_{\rm
      Lip}\le 1} , 
\end{equation}
where the Lipschitz semi-norm is defined by
\begin{equation*}
 \|\psi\|_{\rm Lip} =\sup_{x\not = y}\frac{|\psi(x)-\psi(y)|}{|x-y|}.  
\end{equation*}
See for instance \cite{Vi03}. As announced in the previous subsection,
we set
\begin{equation}\label{eq:rho}
  \varrho_\si(t,y) = \<t\>^d |u_\si(t,y\<t\>)|^2 \|\phi_\si\|_{L^2}^{-2}. 
\end{equation}
For each $t\in \R$, $\varrho_\si(t,\cdot)$ is a probability density, and
the change of space variable compensates the dispersion of the linear
Schr\"odinger equation.

\begin{theorem}[Global continuity for the rescaled
  modulus]\label{theo:W1} 
   Let $0<\si<\frac{2}{(d-2)_+}$ and $(\phi_\nu)_{|\nu-\si|\le \eps}$
  bounded in $\Sigma$ for some $\eps>0$ such that $0<\si-\eps$ and
  $\si+\eps <  \frac{2}{(d-2)_+}$. If 
   \begin{equation*}
     \phi_\nu\Tend \nu \si \phi_\si\quad\text{in }L^2(\R^d),
   \end{equation*}
   then, with $\varrho_\si$ given by \eqref{eq:rho},
   \begin{equation*}
     \sup_{t\in \R}W_1\(\varrho_\nu(t),\varrho_\si(t)\)\Tend \nu \si 0. 
   \end{equation*}
\end{theorem}

\begin{corollary}\label{cor:interpolation}
  Under the assumptions of Theorem~\ref{theo:W1}, for any
  $s>\frac{1+d}{2}$,
  \begin{equation*}
     \sup_{t\in \R}\|\varrho_\nu(t)-\varrho_\si(t)\|_{H^{-s}(\R^d)}\Tend \nu \si 0. 
   \end{equation*}
   We infer,
   \begin{itemize}
   \item If $\si>2/d$,
     \begin{itemize}
     \item When $d=1$, for any $s\in [0,1)$ and $p\in (1,\infty]$,
      \begin{equation*}
        \sup_{t\in \R}\|\varrho_\nu(t)-\varrho_\si(t)\|_{H^{s}(\R)}+
        \sup_{t\in \R}\|\varrho_\nu(t)-\varrho_\si(t)\|_{L^p(\R)}\Tend \nu \si 0. 
   \end{equation*}  
 \item When $d=2$, for any $p\in (1,\infty)$,
   \begin{equation*}
        \sup_{t\in \R}\|\varrho_\nu(t)-\varrho_\si(t)\|_{L^p(\R^2)}\Tend \nu \si 0. 
   \end{equation*}  
     \item When $d\ge 3$, there exists $p(d)\in (1,2)$ such that for
       all $p\in (1,p(d)]$,
    \begin{equation*}
        \sup_{t\in \R}\|\varrho_\nu(t)-\varrho_\si(t)\|_{L^p(\R^d)}\Tend \nu \si 0. 
   \end{equation*}  
      \end{itemize}
   \item If $\si\le 2/d$, for all $0<\eps\ll 1$, there exists
     $q(\eps)\in (1,1+\si)$ such that
   \begin{equation*}
        \sup_{t\in
          \R}\|\varrho_\nu(t)-\varrho_\si(t)\|_{W^{-\eps,q}(\R^d)}\Tend \nu
        \si 0,\quad \forall q\in (1,q(\eps)]. 
   \end{equation*}     
   \end{itemize}
\end{corollary}
The values of the above Lebesgue indices $p(d)$ and $q(\eps)$ can
actually be computed, as it should be clear from the proof in
Section~\ref{sec:interpolation}. Similarly, convergence holds in other
Sobolev spaces, but we have chosen to make the statements as simple as
possible in Corollary~\ref{cor:interpolation}. 

\subsubsection{Logarithmic nonlinearity}
Like for the case of power nonlinearity, we first start with strong
convergence locally in time, before considering global in time
convergence, in Wasserstein distance. Recall that for the convergence
toward the logarithmic Schr\"odinger equation, we consider
\eqref{eq:rescaledNLS} (as opposed to \eqref{eq:NLS}). 
\begin{theorem}\label{theo:log-loc-temps}
  Let $\eps>0$ and $(\phi_\si)_{0\le \si\le \eps}$ bounded in $\Sigma$. Let $\u_\si$ denote the solution to
  \eqref{eq:rescaledNLS}, and $\u_0$ the solution to
  \eqref{eq:logNLS}, with initial data $\phi_\si$ and $\phi_0$
  respectively. There exists $\eps_0>0$ such that for all $\si\in 
  (0,\eps_0)$, the following holds. There exist $C_0,C_1>0$ such that
  for any $T>0$,  
  \begin{equation*}
  \sup_{t\in [-T,T]}  \|\u_\si(t)-\u_0(t)\|_{L^2(\R^d)}\le \(C_1 \si +
  \|\phi_\si-\phi_0\|_{L^2(\R^d)}\) e^{C_0T},
\end{equation*}
and, if $\phi_\si\to \phi_0$ in $L^2(\R^d)$, for any $\theta\in (0,1)$,
\begin{equation*}
  \sup_{|t|\le \frac{\theta}{C_0}\ln\frac{1}{\si}}\|\u_\si(t)-\u_0(t)\|_{L^2(\R^d)}\Tend \si
  0 0. 
\end{equation*}
\end{theorem}
The last convergence stated in Theorem~\ref{theo:log-loc-temps}
can be understood as a convergence up to some Ehrenfest time, by
analogy with the definition from the semiclassical propagation of
coherent states in (linear) Schr\"odinger equations, see
e.g. \cite{CoRo21}. 

\begin{theorem}\label{theo:log-temps-long}
  For $\si>0$, let $\tau_\si,\tau_0\in \cont^2(\R)$ solve the ordinary
  differential equations
  \begin{equation*}
    \ddot \tau_\si= \frac{1}{2\tau_\si^{d\si+1}},\quad \ddot \tau_0=
    \frac{1}{2\tau_0},\quad \tau_\si(0)=\tau_0(0)=1,\quad
    \dot\tau_\si(0)=\dot\tau_0(0)=0. 
  \end{equation*}
  Their large time behavior is given by
  \begin{equation*}
    \tau_\si(t)\Eq t \infty \frac{t}{\sqrt \si},\quad \tau_0(t)\Eq t
  \infty t\sqrt{\ln t}, 
\end{equation*}
and $\tau_\si$ converge to $\tau_0$: there exists $C>0$ such that
\begin{equation*}
  |\tau_0(t) - \tau_\si(t)|\le C\si t\(\ln
  (t+2)\)^{3/2},\quad\text{hence}\quad \tau_0(t) -
  \tau_\si(t)=o\(\tau_0(t)\) \text{ for }\si\ln t\to 0.  
\end{equation*}
Let $\eps>0$, $(\phi_\si)_{0\le \si\le \eps}$ bounded in $\Sigma$,
with $\phi_\si\to \phi_0$ in $L^2(\R^d)$ as $\si\to 0$. Define
\begin{equation*}
  \varrho_0(t,y) = \tau_0(t)^d\left\lvert
    \u_0\(t,y\tau_0(t)\)\right\rvert^2 \|\phi_0\|_{L^2}^{-2},\quad
   \varrho_\si(t,y) = \tau_\si(t)^d\left\lvert
    \u_\si\(t,y\tau_\si(t)\)\right\rvert^2 \|\phi_\si\|_{L^2}^{-2}.
\end{equation*}
We have the uniform in time convergence in Wasserstein distance:
\begin{equation} \label{eq:convergence_thm}
  \sup_{t\ge 0}W_1\(\varrho_\si,\varrho_0 \) \Tend \si 0 0 . 
\end{equation}
Moreover, if $\| \phi_{\sigma} - \phi_0 \|_{L^2(\R^d)}=\mathcal{O}(\sigma)$, we have the convergence rate
\begin{equation} \label{eq:convergence_rate_thm}
  \sup_{t\ge 0}W_1\(\varrho_\si,\varrho_0 \) \lesssim  \frac{1}{ \sqrt{\ln \ln \frac{1}{\sigma}}}.
\end{equation}
\end{theorem}
We note that the convergence of $\varrho_\si$ toward $\varrho_0$ holds
uniformly in time, and so supersedes the range  on which $\tau_\si$
converges to $\tau_0$. Indeed, the asymptotic behaviors of $\tau_0$
and $\tau_\si$ become different for $\si\approx \frac{1}{\ln t}$ (and
then the control of $\tau_0-\tau_\si$ becomes of the same order as
$\tau_0$, indicating some sharpness in this error estimate).

\subsection{Content of the article}
\label{sec:content}

In Section~\ref{sec:prelim}, we introduce standard tools and results
related to \eqref{eq:NLS}. Based on those, Theorem~\ref{theo:local} is
proven in Section~\ref{sec:bounded}. The case of short range
nonlinearity $\si>\si_0(d)$ is addressed in
Section~\ref{sec:global-scatt}, containing the proof of
Theorem~\ref{theo:global} and Corollary~\ref{cor:scattering}. The
uniform in time convergence of rescaled densities, stated in
Theorem~\ref{theo:W1}, is given in Section~\ref{sec:global-gen}. The
rest of the paper deals with the convergence to the logarithmic
Schr\"odinger equation, from \eqref{eq:rescaledNLS} to
\eqref{eq:logNLS}. In Section~\ref{sec:ODE}, we focus our attention on
the ordinary differential equations for $\tau_0$ and $\tau_\si$,
proving the first conclusions of
Theorem~\ref{theo:log-temps-long}. In Section~\ref{sec:apriori}, we
prove a priori estimates for the solution $\u_\si$, with emphasis on
large time. The convergence toward the logarithmic Schr\"odinger
equation up to Ehrenfest time, Theorem~\ref{theo:log-loc-temps}, is proven
in Section~\ref{sec:ehrenfest}, where we also establish convergence in
Wasserstein distance $W_1$ on such time intervals, serving as a
preparatory step for the final Section~\ref{sec:log}, where the proof
of Theorem~\ref{theo:log-temps-long} is completed.

\section{Preliminary estimates}
\label{sec:prelim}

\subsection{Some technical tools}
\label{sec:tools}

We recall Gagliardo-Nirenberg inequalities, and standard inequalities
which may be viewed as a dual version of Gagliardo-Nirenberg
inequalities: for all
$v\in \Sch(\R^d)$, 
\begin{align}\label{eq:GN}
    &\|v\|_{L^p}\lesssim \|v\|_{L^2}^{1-\delta(p)}\|\nabla
    v\|_{L^2}^{\delta(p)},\quad \text{for}\quad 2\le p<
      \frac{2d}{(d-2)_+},\\
& \left\|v\right\|_{L^{p}} \lesssim\left\| v\right\|_{L^2}^{1-\delta(p')}
  \left\| y v\right\|_{L^2}^{\delta(p')},\quad \text{for}\quad
  \max\(1,\frac{2d}{d+2}\)< p\le 2, \label{eq:GNdual}
\end{align}
where
\begin{equation*}
  \delta(p):=d\(\frac{1}{2}-\frac{1}{p}\),\quad
  \frac{1}{p}+\frac{1}{p'}=1. 
\end{equation*}
We now recall classical Strichartz estimates:
\begin{definition}\label{def:adm}
A pair $(q,r)$ is {\bf admissible} if $2\le
 r\le\infty$ if $d=1$, $2\le r< 
  \infty$ if $d=2$, $2\le r\le \frac{2d}{d-2}$ if $d\ge 3$,
  and 
  \begin{equation*}
    \frac{2}{q}=\delta(r):= d\left( \frac{1}{2}-\frac{1}{r}\right).
  \end{equation*}
\end{definition}

\begin{lemma}[Strichartz
  estimates]\label{lem:strichartz}  
Denote $U_0(t)=e^{i\frac{t}{2} \Delta}$.\\
$(1)$ \emph{Homogeneous Strichartz estimate.} For any admissible pair
\index{admissible pair} $(q,r)$,
there exists $C_{q}$ such that
\begin{equation*}
  \|U_0(t) \varphi\|_{L^{q}(\R;L^{r})} \le C_q 
\|\varphi \|_{L^2},\quad \forall \varphi\in L^2(\R^d).
\end{equation*}
$(2)$ \emph{Inhomogeneous Strichartz estimate.}
For a time interval $I$, denote
\begin{equation*}
  D_I(F)(t,x) = \int_{I\cap\{\tau\le
      t\}} U_0(t-\tau)F(\tau,x)\dd \tau. 
\end{equation*}
For all admissible pairs \index{admissible pair}$(q_1,r_1)$ and~$
    (q_2,r_2)$, and any 
    interval $I$, there exists $C=C_{r_1,r_2}$ independent of 
    $I$ such that 
\begin{equation}\label{eq:strichnl}
      \left\lVert D_I(F)
      \right\rVert_{L^{q_1}(I;L^{r_1})}\le C \left\lVert
      F\right\rVert_{L^{q'_2}\(I;L^{r'_2}\)},
\end{equation}
for all $F\in L^{q'_2}(I;L^{r'_2})$.
\end{lemma}
For a time interval $I$, we introduce the norms
\begin{align*}
  &\|u\|_{X(I)} = \sup_{(q,r)\text{ admissible}}\|u\|_{L^q(I;L^r)},\\
  &\|u\|_{Y(I)} = \max_{A\in \{{\rm Id},\nabla\}}
  \sup_{(q,r)\text{ admissible}}\|A u\|_{L^q(I;L^r)},\\
 &\|u\|_{Z(I)} = \max_{A\in \{{\rm Id},\nabla,x+it\nabla\}}
  \sup_{(q,r)\text{ admissible}}\|Au\|_{L^q(I;L^r)}.
\end{align*}

\subsection{A priori estimates}

\label{sec:est-gen}
Suppose that $\phi_\si\in H^1(\R^d)$. 
Both for the power case \eqref{eq:NLS} and the logarithmic case
\eqref{eq:logNLS}, the mass is conserved: 
\begin{equation*}
M_\si  =\|u_\si(t)\|_{L^2(\R^d)}^2=\|\u_\si(t)\|_{L^2(\R^d)}^2=\|\phi_\si\|^2_{L^2(\R^d)},
\quad  \forall t\in \R,\ \forall \si\ge 0. 
\end{equation*}
Both equations are Hamiltonian, and the conserved energies read
\begin{align*}
  E_\si
  & = \frac{1}{2}\|\nabla u_\si(t)\|_{L^2(\R^d)}^2
    +\frac{1}{\si+1}\|u_\si(t)\|_{L^{2\si+2}(\R^d)}^{2\si+2},   \\
E_0 &=   \frac{1}{2}\|\nabla \u_0(t)\|_{L^2(\R^d)}^2+\int_{\R^d}
      |\u_0(t,x)|^2\ln |\u_0(t,x)|^2 \dd x. 
\end{align*}
In the case of \eqref{eq:rescaledNLS}, the second factor in $E_\si$
must be divided by $\si$ (recalling that the $L^2$-norm is
conserved); we study this case in more details in
Section~\ref{sec:apriori}.  
Suppose in addition that $\phi\in \Sigma$ (we leave out the possible
dependence upon $\si$ at this stage). 
For the power case \eqref{eq:NLS}, the pseudoconformal evolution law,
discovered in \cite{GV79Scatt}, reads:
\begin{equation}\label{eq:pseudo-conf}
  \frac{\dd}{\dd t}\(\frac{1}{2}\|(x+it\nabla) u\|_{L^2}^2
+ \frac{ t^2}{\si +1}\|u\|_{L^{2\si+2}}^{2\si+2}
\)=\frac{t}{\si+1}(2-d\si)\|u\|_{L^{2\si+2}}^{2\si+2} .
\end{equation}
For future references, we recall some important properties associated
with the operator involved in the pseudoconformal evolution law:
\begin{proposition}\label{prop:J}
  The operator
\[J(t)=x+it\nabla\]
satisfies the following properties:
  \begin{itemize}
  \item $J(t) =U_0(t)xU_0(-t)$, where $U_0$ denotes the Schr\"odinger
    group $U_0(t) =e^{i\frac{t}{2}\Delta}$. Therefore, $J$ commutes with the
 linear part of \eqref{eq:NLS},
\begin{equation*}\label{eq:Jcommute}
\left[ J(t),i\d_t +\frac{1}{2}\Delta\right]=0\, .
\end{equation*}
\item It can be factorized as
 \begin{equation*}%\label{eq:Jfactor}
J(t)= i t \, e^{i\frac{|x|^2}{2t}}\nabla\Big( e^{-i\frac{|x|^2}{2t}}\,
\cdot\Big)\, .
\end{equation*}
As a consequence, $J$ yields weighted Gagliardo-Nirenberg
inequalities. For
$2\le r <\frac{2d}{(d-2)_+}$ ($2\le r\le \infty$ if $d=1$), there exists
$C(d, r)$ depending only on $d$ and $r$ such that
\begin{equation*}\label{eq:GNlibre}
\left\| f \right\|_{L^r}\le \frac{C(d, r)}{|t|^{\delta(r)}} \left\|
f \right\|_{L^2}^{1-\delta(r)} \left\|
J(t) f \right\|_{L^2}^{\delta(r)},\quad \delta(r):=d\(
\frac{1}{2}-\frac{1}{r}\)\, .
\end{equation*}
Also, if $F(z)=G(|z|^2)z$ is $C^1$, then $J(t)$
acts like a derivative on $F(w)$:
\begin{equation*}\label{eq:Jder}
J(t)\(F(w)\) = \d_z F(w)J(t)w -\d_{\overline z} F(w)\overline{ J(t)w
}\, .
\end{equation*}
\end{itemize}
\end{proposition}
\smallbreak

If we assume $\si\ge 2/d$, the pseudoconformal evolution law
\eqref{eq:pseudo-conf} 
yields the a priori bound
\begin{equation*}
  \|(x+it\nabla) u_\sigma(t)\|_{L^2}\le C\( \|\phi\|_{\Sigma}\),
\end{equation*}
hence, by Proposition~\ref{prop:J}, 
\begin{equation*}
  \|u_\sigma(t)\|_{L^r}\le\frac{C}{\<t\>^{\delta(r)}}, \quad 2\le
    r<\frac{2}{(d-2)_+},\quad \delta (r) =
  d\(\frac{1}{2}-\frac{1}{r}\), 
\end{equation*}
for some constant $C$ depending only on $d$ and $\|\phi\|_{\Sigma}$. 
In particular,
\begin{equation}\label{eq:pot-decay}
  \|u_\si(t)\|_{L^{2\si+2}}\lesssim \frac{1}{\<t\>^{\frac{d\si}{2\si+2}}}.
\end{equation}
Integrating \eqref{eq:pseudo-conf} in time, and discarding the
(positive) term $\|J(t)u_\si\|_{L^2}^2$, 
Gr\"onwall lemma shows that for
\begin{equation*}
  \si >\frac{2-d+\sqrt{d^2+12d+4}}{4d}=\si_0(d),
\end{equation*}
 \eqref{eq:pot-decay} stills holds. The exponent $\si_0(d)$ is
 standard in the scattering theory in $\Sigma$ for
\eqref{eq:NLS} (see e.e.g \cite{CazCourant}). We note that $\si\ge
2/d>\si_0(d)$. 
\smallbreak

In the case $\si<2/d$, \eqref{eq:pseudo-conf} and Gr\"onwall lemma
yield
\begin{equation*}
  \|J(t)u\|_{L^2}^2 +t^2\|u(t)\|_{L^{2\si+2}}^{2\si+2}\lesssim \<t\>^{2-d\si}. 
\end{equation*}

\subsection{Existence results}
\label{sec:exist}

The following result is classical (see
e.g. \cite{CazCourant}):
\begin{proposition}\label{prop:cauchyNLS}
  Let $d\ge 1$, $0<\si<\frac{2}{(d-2)_+}$ and $\phi\in
  H^1(\R^d)$. Then \eqref{eq:NLS} has a unique solution
  \begin{equation*}
    u_\si \in \cont(\R;H^1(\R^d))\cap \cont^1(\R;H^{-1}(\R^d)).
  \end{equation*}
  For all bounded interval $I$, there exists
  $C=C(I,\|\phi\|_{H^1})$ such that 
  \begin{equation*}
    \|u_\si\|_{Y(I)} \le C,
  \end{equation*}
  and the mass $M$ and the energy $E_\si$ are independent of $t\in \R$. \\
  If in addition $\phi\in \Sigma$, then $u_\si \in \cont(\R;\Sigma)$ and
  for all bounded interval $I$, there exists
  $C=C(I,\|\phi\|_{\Sigma})$ such that 
  \begin{equation*}
    \|u_\si\|_{Z(I)} \le C,
  \end{equation*}
  and \eqref{eq:pseudo-conf} holds. 
 \end{proposition}

The next result follows from \cite{CaGa18} (case $\Sigma$)
and \cite{HO25} (case $H^1$) for \eqref{eq:logNLS}:
\begin{proposition}\label{prop:cauchylogNLS}
  Let $d\ge 1$. If $\phi\in H^1(\R^d)$, then \eqref{eq:logNLS} has a
  unique solution $\u_0 \in \cont(\R;H^1(\R^d))$. If in addition $\phi\in
  \Sigma$, then $\u_0\in \cont(\R;\Sigma)$, and the mass $M_0$ and the energy
  $E_0$ are independent of time. 
\end{proposition}
For future reference, we gather a priori estimates
stemming from \eqref{eq:pseudo-conf} and the conservation of the
energy in the following lemma.
\begin{lemma}\label{lem:apriori-pseudoconf}
  Let $\phi\in \Sigma$ and $u\in \cont(\R;\Sigma)$ solve
  \eqref{eq:NLS}, with 
  $0<\si<\frac{2}{(d-2)_+}$.
  \begin{itemize}
  \item If $\si\ge 2/d$, then for all $t\in \R$,
    \begin{align*}
      & \|J(t)u\|_{L^2(\R^d)}\le \|x\phi\|_{L^2(\R^d)},\\
      & \|u(t)\|_{L^r(\R^d)}\le
        \frac{C}{|t|^{\delta(r)}}\|x\phi\|_{L^2(\R^d)},\quad 2\le
        r<\frac{2d}{(d-2)_+}, 
    \end{align*}
    where $\delta(r)$ is defined in Proposition~\ref{prop:J}.
  \item If $\si<2/d$, then for all $t\in \R$,
    \begin{align*}
      & \|J(t)u\|_{L^2(\R^d)}\lesssim \<t\>^{1-\frac{d\si}{2}} ,\\
      & \|u(t)\|_{L^r(\R^d)}\lesssim
        \<t\>^{-\frac{d\si}{2}\delta(r)},\quad 2\le
        r<\frac{2d}{(d-2)_+}.
    \end{align*}
  \end{itemize}
\end{lemma}

\section{Bounded time intervals}
\label{sec:bounded}

In this section, we prove Theorem~\ref{theo:local}.
Since the equation \eqref{eq:NLS} is reversible, we consider positive
time only.
\subsection{$L^2$-estimate}
We fix $\si>0$, and consider $\nu\in \(0,\frac{2}{(d-2)_+}\)$ in a
neighborhood of $\si$. The difference $u_\nu-u_\si$ solves
\begin{equation}\label{eq:err-cont}
  i\d_t\(u_{\nu}- u_{\sigma}\) +\frac{1}{2}\Delta
  \(u_{\nu}- u_{\sigma}\)=
 |u_{\nu}|^{2\nu}u_{\nu} -
  |u_{\sigma}|^{2\nu}u_{\sigma}+
  |u_{\sigma}|^{2\nu}u_{\sigma}
  - |u_{\sigma}|^{2\sigma}u_{\sigma}.
\end{equation}
The pair
\begin{equation*}
  (q_\si,r_\si):=\(\frac{4\si+4}{d\si},2\si+2\)
\end{equation*}
is admissible, and (with the same implicit time interval in each term)
\begin{align*}
  \left\| |u_{\sigma}|^{2\nu}u_{\sigma}
  - |u_{\sigma}|^{2\sigma}u_{\sigma}\right\|_{L^{q_\si'}L^{r_\si'}}\le
  \left\| |u_{\sigma}|^{2\nu}
  - |u_{\sigma}|^{2\sigma}\right\|_{L^{\theta}L^{\frac{r_\si}{2\si}}}
   \left\| u_{\sigma}\right\|_{L^{q_\si}L^{r_\si}},
\end{align*}
where
\begin{equation*}
\frac{1}{r_\si'}=\frac{2\si+1}{r_\si}\quad ;\quad   \frac{1}{q_\si'} = \frac{1}{q_\si}+\frac{1}{\theta}\Longleftrightarrow
  \theta = \frac{2\si+2}{2-(d-2)\sigma}.
\end{equation*}
We have $1<\theta<\infty$ since we consider energy-subcritical
nonlinearities. Write $\nu=\si(1+h)$, with $|h|\ll 1$. Taylor formula
yields for $y\ge 0$ (a placeholder for $|u_\si|^{2\si}$),
\begin{equation*}
  y-y^{1+h}= -hy\ln y \int_0^1 y^{s h} \dd s,
\end{equation*}
and
\begin{equation*}
  \int_0^1 y^{s h} \dd s \le 1 + y^{-|h|}\1_{0<y\le
    1}+y^{|h|}\1_{y>1},
\end{equation*}
hence 
\begin{equation*}
  \left|y-y^{1+h}\right|\lesssim |h| y |\ln y| \(
  1+y^{-|h|}+y^{|h|}\)\lesssim |h| \(y^{1-\eta} + y^{1+\eta}\),
\end{equation*}
as soon as $|h|<\eta$. Denoting $\nu_\pm = \si(1\pm \eta)$, we have,
for  $0<\eta\ll 1$,
\begin{equation*}
  |u_\si|^{2\nu_\pm}\in L^\theta L^{\frac{r_\si}{2\si}},
\end{equation*}
since for $0<\eta\ll 1$, $r_\si \frac{\nu_\pm}{\si}=(1\pm \eta)r_\si
\in (2,2^*)$, where
\begin{equation*}
  2^* =\frac{2d}{(d-2)_+}. 
\end{equation*}
Indeed, $r_\si \in (2,2^*)$, and since this condition is open, it
remains true for nearby indices. Therefore,
\begin{equation*}
  \left\| |u_{\sigma}|^{2\nu}u_{\sigma}
  -
  |u_{\sigma}|^{2\sigma}u_{\sigma}\right\|_{L^{q_\si'}L^{r_\si'}}
\lesssim |\si-\nu| \|u_\si\|_{L^\infty H^1}^{2\si+1}. 
\end{equation*}
For the term $ |u_{\nu}|^{2\nu}u_{\nu} -
|u_{\sigma}|^{2\nu}u_{\sigma}$, we use the pair $(q_\nu,r_\nu)$, and
write
\begin{equation*}
  \left\| |u_{\nu}|^{2\nu}u_{\nu} -
|u_{\sigma}|^{2\nu}u_{\sigma}\right\|_{L^{q_\nu'}L^{r_\nu'}}\le \(
  \|u_\nu\|^{2\nu}_{L^kL^{r_\nu}} +  \|u_\si\|^{2\nu}_{L^kL^{r_\nu}}
  \) \|u_\nu-u_\sigma\|_{L^{q_\nu}L^{r_\nu}},
\end{equation*}
where
\begin{equation*}
  \frac{1}{q_\nu'} = \frac{1}{q_\nu}+\frac{2\nu}{k}\Longleftrightarrow
  k = \frac{2\nu(2\nu+2)}{2-(d-2)\nu}.
\end{equation*}
For a bounded time interval, the conservation of the energy implies
\begin{equation*}
  \| u_\sigma\|_{L^\infty H^1} + \| u_\nu\|_{L^\infty H^1} \le C\(
  \|\phi\|_{H^1}\), 
\end{equation*}
and so, thanks to Sobolev embedding,
\begin{equation*}
   \|u_\nu\|^{2\nu}_{L^k(I;L^{r_\nu})} +
   \|u_\si\|^{2\nu}_{L^k(I;L^{r_\nu})}\le |I|^{1/k} C \(
  \|\phi\|_{H^1}\).
\end{equation*}
Let $I_n=[t_n,t_{n+1}]$:
\begin{equation*}
L^{q_\nu}(I_n;L^{r_\nu})\cap L^{q_\si}(I_n;L^{r_\si})\cap
L^\infty(I_n;L^2)\subset X(I_n).
\end{equation*}
Strichartz estimates yield
\begin{align*}
  \|u_\nu-u_\sigma\|_{X(I_n)}
  & \lesssim \|u_\nu(t_n)-u_\sigma(t_n)\|_{L^2}
    + |I_n|^{1/k} \|u_\nu-u_\sigma\|_{X(I_n)}\\
&\quad+  \left\| |u_{\sigma}|^{2\nu}
  -
  |u_{\sigma}|^{2\sigma}\right\|_{L^{\theta}(I_n;L^{\frac{r_\si}{2\si}})}
\left\| u_{\sigma}\right\|_{L^{q_\si}L^{r_\si}}.
\end{align*}
We decompose $[0,T]$ into finitely many $I_n$ so we infer, up to
doubling the implicit constant,
\begin{equation*}
  \|u_\nu-u_\sigma\|_{X(I_n)}\lesssim \|u_\nu(t_n)-u_\sigma(t_n)\|_{L^2}
  + \left\| |u_{\sigma}|^{2\nu}
  -
  |u_{\sigma}|^{2\sigma}\right\|_{L^{\theta}(I_n;L^{\frac{r_\si}{2\si}})}
\left\| u_{\sigma}\right\|_{L^{q_\si}L^{r_\si}}.
\end{equation*}
We conclude that, as $\nu\to \si$, 
\begin{equation*}
  \|u_\nu-u_\sigma\|_{X([0,T])}=\O_T\(\|\phi_\nu-\phi_\si\|_{L^2}+|\nu-\si|\).
\end{equation*}

\subsection{Estimate in $\Sigma$}

The convergence in $H^1$ is interesting, as well as in $\F(H^1)$: the
latter essentially needs the former, as either we apply the
multiplication operator $x$ to \eqref{eq:err-cont}, and a source term
$\nabla (u_\nu-u_\si)$ appears due to the commutator $[x,\Delta]$, or
we apply $x+it\nabla$, which commutes with the left hand side of
\eqref{eq:err-cont}, but acts on the nonlinearity like the gradient. 

Let $A\in \{i\nabla,x+it\nabla\}$, and apply $A$ to
\eqref{eq:err-cont}:
\begin{equation*}
   i\d_tA\(u_{\nu}- u_{\sigma}\) +\frac{1}{2}\Delta
  A\(u_{\nu}- u_{\sigma}\)=  L_A(\nu,\si)+S_A(\nu,\si),
 \end{equation*}
where the term which is always locally Lipschitz when $A=\mathrm{Id}$ is 
\begin{equation*}
  L_A(\nu,\si) = A\( |u_{\nu}|^{2\nu}u_{\nu} -
  |u_{\sigma}|^{2\nu}u_{\sigma}\),
\end{equation*}
and the source term is 
\begin{equation*}
 S_A(\nu,\si)= A\( |u_{\sigma}|^{2\nu}u_{\sigma}
  - |u_{\sigma}|^{2\sigma}u_{\sigma}\).
\end{equation*}
We have
\begin{equation*}
  A\(|u|^{2\nu}u\) = (\nu+1)|u|^{2\nu}Au - \nu u^{\nu+1}\bar
  u^{\nu-1}\overline{Au},
\end{equation*}
and we write
\begin{align*}
 L_A(\nu,\si)&= (\nu+1)|u_\nu|^{2\nu}Au_\nu - \nu u_\nu^{\nu+1}\bar
  u_\nu^{\nu-1}\overline{Au_\nu} - (\nu+1)|u_\si|^{2\nu}Au_\si \\
  &\quad - \nu
               u_\si^{\nu+1}\bar   u_\si^{\nu-1}\overline{Au_\si} \pm \((\nu+1)|u_\nu|^{2\nu}Au_\si - \nu u_\nu^{\nu+1}\bar
  u_\nu^{\nu-1}\overline{Au_\si}\),
\end{align*}
so we have the pointwise estimate
\begin{align*}
  |L_A(\nu,\si)|&\lesssim |u_\nu|^{2\nu}
  |A(u_\nu-u_\si)|+\left| |u_\nu|^{2\nu}-|u_\si|^{2\nu}\right| |Au_\si|.
\end{align*}
Resuming the Lebesgue indices from the previous subsection,
\begin{align*}
  \|L_A(\nu,\si)\|_{L^{q_\nu'}L^{r_\nu'}}
&\lesssim 
  \|u_\nu\|^{2\nu}_{L^k L^{r_\nu}} \|A(u_\nu-u_\sigma)\|_{L^{q_\nu}L^{r_\nu}}\\
&\quad  + \left\| |u_\nu|^{2\nu} -
    |u_\si|^{2\nu}\right\|_{L^{\frac{k}{2\nu}} L^{\frac{r_\nu}{2\nu}}}  
\|A u_\sigma\|_{L^{q_\nu}L^{r_\nu}},
\end{align*}
where
\begin{equation*}
  \frac{1}{q_\nu'} = \frac{1}{q_\nu}+\frac{2\nu}{k}\Longleftrightarrow
  k = \frac{2\nu(2\nu+2)}{2-(d-2)\nu}.
\end{equation*}
The second term in the above inequality is actually considered as a
source term in view of the convergence established in the previous
subsection. 
Like in e.g. \cite{CFH11}, we distinguish two cases:
\begin{itemize}
\item If $2\nu\ge 1$, the application $z\mapsto |z|^{2\nu}$ is
  (locally) Lipschitz continuous, 
  \begin{equation*}
    \left||u_\nu|^{2\nu} -    |u_\si|^{2\nu}\right| \lesssim 
\( |u_\nu|^{2\nu-1} +    |u_\si|^{2\nu-1}\) |u_\nu-u_\si|,
\end{equation*}
hence, by H\"older inequality,
\begin{equation*}
  \left\| |u_\nu|^{2\nu} -
    |u_\si|^{2\nu}\right\|_{L^{\frac{k}{2\nu}} L^{\frac{r_\nu}{2\nu}}}
  \lesssim
  \( \|u_\nu\|_{L^kL^{r_\nu}}^{2\nu-1} +
  \|u_\si\|_{L^kL^{r_\nu}}^{2\nu-1}\) \|u_\nu-u_\si\|_{L^kL^{r_\nu}}.
\end{equation*}
\item If $0<2\nu<1$, the application $z\mapsto |z|^{2\nu}$ is H\"older
  continuous,
  \begin{equation*}
    \left||u_\nu|^{2\nu} -    |u_\si|^{2\nu}\right| \lesssim 
|u_\nu-u_\si|^{2\nu},
\end{equation*}
and
\begin{equation*}
  \left\| |u_\nu|^{2\nu} -
    |u_\si|^{2\nu}\right\|_{L^{\frac{k}{2\nu}} L^{\frac{r_\nu}{2\nu}}}
  \lesssim \|u_\nu-u_\si\|^{2\nu}_{L^k L^{r_\nu}} .
\end{equation*}
\end{itemize}
In view of Proposition~\ref{prop:cauchyNLS}, we infer
\begin{equation*}
   \left\| |u_\nu|^{2\nu} -
    |u_\si|^{2\nu}\right\|_{L^{\frac{k}{2\nu}} L^{\frac{r_\nu}{2\nu}}}
  \lesssim \|u_\nu-u_\si\|^{\min(1,2\nu)}_{L^k L^{r_\nu}}.
\end{equation*}
In view of Sobolev embedding, when $\nu<\frac{2d}{(d-2)_+}$,
\begin{equation*}
  H^{s_\nu}(\R^d)\hookrightarrow L^{2\nu+2}(\R^d),\quad s_\nu = \frac{d\nu}{2\nu+2}<1,
\end{equation*}
for $\nu\in ((1-\eta)\si,(1+\eta)\si)$ and $\eta>0$ sufficiently
small. Therefore, in view of the previous subsection,
\begin{align*}
 \|u_\nu-u_\si\|_{L^\infty L^{r_\nu}}
 &\lesssim
     \|u_\nu-u_\si\|_{L^\infty H^{s_\nu}} =  \|u_\nu-u_\si\|_{L^\infty
   L^{2}}+
\|u_\nu-u_\si\|_{L^\infty \dot H^{s_\nu}}\\
 &\lesssim |\nu-\si | + \|u_\nu-u_\si\|_{L^\infty \dot H^{s_\nu}},
\end{align*}
where $\dot H^s$ denotes the homogeneous Sobolev space. 
By interpolation,
\begin{equation*}
  \|f\|_{\dot H^{s_\nu}} \le \|f\|_{L^2}^{1-s_\nu}\|f\|_{\dot H^1}^{s_\nu},
\end{equation*}
and from the previous subsection,
\begin{equation*}
  \|u_\nu-u_\si\|_{L^\infty L^{r_\nu}}\lesssim  |\nu-\si | + 
|\nu-\si |^{1-s_\nu}\|u_\nu-u_\si\|^{s_\nu}_{L^\infty \dot H^{1}}.
\end{equation*}
Therefore, setting $\gamma :=\min(1,2\nu)$,  for any $I\subset[0,T]$,
\begin{equation*}
   \left\| |u_\nu|^{2\nu} -
    |u_\si|^{2\nu}\right\|_{L^{\frac{k}{2\nu}}(I; L^{\frac{r_\nu}{2\nu}})}
  \lesssim   |\nu-\si|^{\gamma} +
  |I|^{\gamma/k}|\nu-\si |^{\gamma(1-s_\nu)}\|u_\nu-u_\si\|^{\gamma
    s_\nu} _{L^\infty \dot H^{1}}. 
\end{equation*}
We also have
\begin{equation*}
   \|u_\nu\|^{2\nu}_{L^k (I;L^{r_\nu})}
   \|A(u_\nu-u_\sigma)\|_{L^{q_\nu}(I;L^{r_\nu})}\le
   |I|^{1/k} \|u_\nu\|_{L^\infty H^1}^{2\nu} \|A(u_\nu-u_\sigma)\|_{X(I)},
 \end{equation*}
 and so
 \begin{align*}
  \|L_A(\nu,\si)\|_{L^{q_\nu'}(I;L^{r_\nu'})}&\lesssim
  |\nu-\si|^{\gamma} + |I|^{1/k}
  \|A(u_\nu-u_\sigma)\|_{X(I)} \\ 
  	&\quad +|I|^{\gamma/k} |\nu - \si|^{\gamma (1 - s_\nu)} \|u_\nu-u_\si\|^{\gamma s_\nu}_{L^\infty \dot
    H^{1}} .   
 \end{align*}
On the other hand, we write
\begin{align*}
  S_A(\nu,\si)& = (\nu+1)|u_\si|^{2\nu}Au_\si - \nu u_\si^{\nu+1}\bar
                u_\si^{\nu-1}\overline{Au_\si} -(\si+1)|u_\si|^{2\si}Au_\si \\
  &\quad +
                \si u_\si^{\si+1}\bar  u_\si^{\si-1}\overline{Au_\si}
  \pm \((\nu+1)|u_\si|^{2\si}Au_\si - \nu u_\si^{\si+1}\bar
  u_\si^{\si-1}\overline{Au_\si}\),
\end{align*}
hence
\begin{align*}
  |S_A(\nu,\si)|\lesssim \left| |u_\si|^{2\nu} - |u_\si|^{2\si}\right|
  |Au_\si|  +|\nu-\si| |u_\si|^{2\si} |Au_\si|.
\end{align*}
Resuming the notation $I_n=[t_n,t_{n+1}]$, Strichartz estimates yield
\begin{align*}
  \|A(u_\nu-u_\si)\|_{X(I_n)}
  &\lesssim \|A(u_\nu-u_\si)(t_n)\|_{L^2} +
    \|L_A(\nu,\si)\|_{L^{q_\nu'}(I_n;L^{r_\nu'})}\\
&\quad+
  \|S_A(\nu,\si)\|_{L^{q_\si'}(I_n;L^{r_\si'})}\\
  &\lesssim \|A(u_\nu-u_\si)(t_n)\|_{L^2} +
    |\nu-\si|^{\gamma} + |I_n|^{1/k} \|A(u_\nu-u_\sigma)\|_{X(I_n)}\\
&\quad + |I_n|^{\gamma/k} |\nu-\si |^{\gamma(1-s_\nu)}\|u_\nu-u_\si\|^{\gamma
    s_\nu} _{L^\infty \dot H^{1}}\\
  &\quad + \left\| |u_\si|^{2\nu} -
    |u_\si|^{2\si}\right\|_{L^\theta(I_n;L^{\frac{r_\si}{2\si}} )}
    \|Au_\si\|_{L^{q_\si}(I_n;L^{r_\si})} \\
  &\quad +|\nu-\si|
    \|u_\si\|_{L^{2\si \theta}(I_n;L^{r_\si} )}^{2\si}
    \|Au_\si\|_{L^{q_\si}(I_n;L^{r_\si})}  \\
   &\lesssim \|A(u_\nu-u_\si)(t_n)\|_{L^2} +
    |\nu-\si|^{\gamma} + |I_n|^{1/k} \|A(u_\nu-u_\sigma)\|_{X(I_n)}\\
&\quad + |I_n|^{\gamma/k} |\nu-\si |^{\gamma(1-s_\nu)}\|u_\nu-u_\si\|^{\gamma
    s_\nu} _{L^\infty \dot H^{1}}\\
  &\quad + \left\| |u_\si|^{2\nu} -
    |u_\si|^{2\si}\right\|_{L^\theta(I_n;L^{\frac{r_\si}{2\si}} )}
    \|Au_\si\|_{X([0,T])} \\
  &\quad +|\nu-\si|
    \|u_\si\|_{L^\infty H^1}^{2\si}
    \|Au_\si\|_{X([0,T])} .
\end{align*}
We note that $\gamma s_\nu<1$, and Young
inequality implies
\begin{equation*}
  |\nu-\si |^{\gamma(1-s_\nu)} \|u_\nu-u_\si\|^{\gamma
    s_\nu} _{L^\infty \dot H^{1}}\lesssim  |\nu-\si
  |^{\gamma(1-s_\nu)/(1-\gamma s_\nu)} + \|u_\nu-u_\si\|_{L^\infty \dot H^{1}}.
\end{equation*}
Since $\gamma\le 1$, 
\begin{equation*}
  \frac{1-s_\nu}{1-\gamma s_\nu} \le 1,
\end{equation*}
for $|\nu-\si|\le 1$,
\begin{equation*}
|\nu-\si |\le  |\nu-\si |^{\gamma}\le |\nu-\si
  |^{\gamma(1-s_\nu)/(1-\gamma s_\nu)} .
\end{equation*}
The term $\left\| |u_\si|^{2\nu} -
    |u_\si|^{2\si}\right\|_{L^\theta(I_n;L^{\frac{r_\si}{2\si}} )}$
  was estimated in the previous subsection,
  \begin{equation*}
    \left\| |u_\si|^{2\nu} -
    |u_\si|^{2\si}\right\|_{L^\theta(I_n;L^{\frac{r_\si}{2\si}}
    )}\lesssim |\nu-\si|,
\end{equation*}
and we have
\begin{align*}
  \|A(u_\nu-u_\si)\|_{X(I_n)}
 &\lesssim \|A(u-u_\si)(t_n)\|_{L^2} +
    |\nu-\si|^{\gamma(1-s_\nu)/(1-\gamma s_\nu)} \\
&\quad + |I_n|^{1/k}
   \|A(u_\nu-u_\sigma)\|_{X(I_n)}
 + |I_n|^{\gamma/k} \|u_\nu-u_\si\|_{L^\infty (I_n;\dot H^{1})} \\
&\lesssim \|A(u-u_\si)(t_n)\|_{L^2} +
    |\nu-\si|^{\gamma(1-s_\nu)/(1-\gamma s_\nu)} \\
&\quad + |I_n|^{1/k}
   \|A(u_\nu-u_\sigma)\|_{X(I_n)}
 + |I_n|^{\gamma/k} \|\nabla(u_\nu-u_\si)\|_{X(I_n)}. 
\end{align*}
We first let $A=\nabla$,
and decompose $[0,T]$ into finitely many $I_n$, to get
\begin{equation*}
   \|\nabla (u_\nu-u_\si)\|_{X([0,T])} 
=\O_T\(\|\nabla (\phi_\nu-\phi_\si)\|_{L^2}+
|\nu-\si|^{\gamma(1-s_\nu)/(1-\gamma s_\nu)}\) .
\end{equation*}
Considering then $A= x+it\nabla=J(t)$, we infer
\begin{equation*}
  \|J (u_\nu-u_\si) \|_{X([0,T])}
=\O_T\(\|J (\phi_\nu-\phi_\si)\|_{L^2}+|\nu-\si|^{\gamma(1-s_\nu)/(1-\gamma s_\nu)}\).
 \end{equation*}
 In view of the expressions of $s_\nu$ and $\gamma$,
 \begin{equation*}
   \gamma\frac{1-s_\nu}{1-\gamma s_\nu}=
   \begin{cases}
     1& \text{ if }2\nu\ge 1,\\
\frac{2-(d-2)\nu}{\nu+1-d\nu^2}\nu &\text{ if }2\nu<1,
   \end{cases}
 \end{equation*}
   and

\begin{equation} \label{eq:exp_theta_1}
   \gamma\frac{1-s_\nu}{1-\gamma s_\nu}=\min\( 1,
   \frac{2-(d-2)\nu}{\nu+1-d\nu^2}\nu \).
\end{equation}
The final rate is 
\begin{equation} \label{eq:exp_theta_2}
 \theta= \inf_{|\nu-\si|\le \eta}  \gamma\frac{1-s_\nu}{1-\gamma
    s_\nu}.
\end{equation}

\section{Global dynamics in the (very) short range case}
\label{sec:global-scatt}

\subsection{From local in time to global in time}

To make the estimate from Section~\ref{sec:bounded} global in time, we
essentially proceed the same way, except that to estimate the
``nonlinear potential'' terms,
$\|u_\nu\|_{L^k(I;L^{r_\nu})}$ and  $\|u_\si\|_{L^k(I;L^{r_\nu})}$, we
want to let $I=[T,\infty)$ for $T$ sufficiently large, by using an
explicit time decay, provided in Proposition~\ref{prop:J},
\begin{equation}\label{eq:integrability-short}
  \|u_\nu(t)\|_{L^{r_\nu}} \lesssim \frac{1}{|t|^{\delta}}
  \|u_\nu(t)\|_{L^2}^{1-\delta}\|J(t)u_\nu\|_{L^2}^\delta,\quad \delta
  = \frac{d\nu}{2\nu+2}. 
\end{equation}
Resuming the same notation as in the previous section, we note the equivalence
\begin{equation*}
 k \times \delta =
  \frac{2d\nu^2}{2-(d-2)\nu}>1\Longleftrightarrow
  \nu>\frac{2-d+\sqrt{d^2+12d+4}}{4d}=\si_0(d).
\end{equation*}
We note that $\si>\si_0(d)$, so the condition is fulfilled when 
$\nu=\si$. As the condition is 
open, for $\eta>0$
sufficiently small,  the above inequality is satisfied for
$|\nu-\si|\le \eta$.
\smallbreak

For the $L^2$-estimate, the previous computations yield, for $I=(t_i,t_f)\subset \R$,
\begin{align*}
  \|u_\nu-u_\si\|_{X(I)}
&\lesssim \|u_\nu(t_i)-u_\si(t_i)\|_{L^2} \\
&\quad +  \left\| |u_{\sigma}|^{2\nu}
  - |u_{\sigma}|^{2\sigma}\right\|_{L^{\theta}(I;L^{\frac{r_\si}{2\si}})}
   \left\| u_{\sigma}\right\|_{L^{q_\si}(I;L^{r_\si})}\\
&\quad+\(\|u_\nu\|^{2\nu}_{L^k(I;L^{r_\nu})} +  \|u_\si\|^{2\nu}_{L^k(I;L^{r_\nu})}
  \) \|u_\nu-u_\sigma\|_{L^{q_\nu}(I;L^{r_\nu})}\\
&\lesssim \|u_\nu(t_i)-u_\si(t_i)\|_{L^2} 
\\
 +  |\nu-\si| &
\(
  \left\||u_{\sigma}|^{2\si(1-\eta)}\right\|_{L^\theta(I;L^{\frac{r_\si}{2\si}})} 
+  \left\||u_{\sigma}|^{2\si(1+\eta)}\right\|_{L^\theta(I;L^{\frac{r_\si}{2\si}})} 
\) \left\| u_{\sigma}\right\|_{L^{q_\si}(I;L^{r_\si})}\\
+&\(\|u_\nu\|^{2\nu}_{L^k(I;L^{r_\nu})} +  \|u_\si\|^{2\nu}_{L^k(I;L^{r_\nu})}
  \) \|u_\nu-u_\sigma\|_{L^{q_\nu}L^{r_\nu}}\\
&\lesssim \|u_\nu(t_i)-u_\si(t_i)\|_{L^2} 
\\
 +  |\nu-\si| &
\(
  \left\||u_{\sigma}|^{1-\eta}\right\|_{L^{2\si \theta}(I;L^{r_\si})} 
+  \left\||u_{\sigma}|^{1+\eta}\right\|_{L^{2\si\theta}(I;L^{r_\si})} 
\) \left\| u_{\sigma}\right\|_{L^{q_\si}(I;L^{r_\si})}\\
+&\(\|u_\nu\|^{2\nu}_{L^k(I;L^{r_\nu})} +  \|u_\si\|^{2\nu}_{L^k(I;L^{r_\nu})}
  \) \|u_\nu-u_\sigma\|_{L^{q_\nu}L^{r_\nu}}.
\end{align*}
In view of the above preliminary, for $\eta\in (0,1)$ sufficiently
small,
\begin{equation*}
  |u_{\sigma}|^{1-\eta}, |u_{\sigma}|^{1+\eta}\in
  L^{2\si\theta}(\R;L^{r_\si}),\quad u_\nu, u_\si\in
  L^k(\R;L^{r_\nu}),\quad u_\si\in L^{q_\si}(\R;L^{r_\si}), 
\end{equation*}
so we can decompose $\R$ into finitely many intervals and repeat the
previous argument. 
The approach also makes it possible to readily adapt the proof  of the
convergence in $\Sigma$, as the ``nonlinear potential'' terms that
need to be absorbed by the left hand side are the same.

Theorem~\ref{theo:global} follows, by noticing the standard identity
\begin{equation*}
  \|x e^{-i\frac{t}{2}\Delta} w\|_{L^2} = \|J(t)w\|_{L^2},
\end{equation*}
due to the identity $J(t) = U_0(t)xU_0(-t)$ and the fact that $U_0(t)$ is
unitary on $L^2$. 

\subsection{Continuity of the scattering operator}

Corollary~\ref{cor:scattering} is a rather direct consequence of the
previous estimates. We first note that since $u_{+,\si}=\lim
e^{-i\frac{t}{2}\Delta}u(t)$ as $t\to +\infty$,
Theorem~\ref{theo:global} implies that if
$\|\phi_{\si}-\phi_{\nu}\|_\Sigma\to 0$ as $\nu\to \si$, then
$\|u_{+,\si}-u_{+,\nu}\|_\Sigma\to 0$. So all we have to explain is
why if $ \|u_{-,\si}-u_{-,\nu}\|_\Sigma\to 0$, then $\|u_{\si\mid
  t=0}-u_{\nu\mid t=0}\|_\Sigma\to 0$. We recall that wave operators
are constructed via a fixed point argument on Duhamel's formula, like
for the initial value problem,
\begin{equation*}
  u_\si(t) = e^{i\frac{t}{2}\Delta} u_{-,\si}-i\int_{-\infty}^t
  e^{i\frac{t-s}{2}\Delta }\( |u_\si|^{2\si}u_\si\)(s)\dd s. 
\end{equation*}
The convergence of the integral near $t=-\infty$ is granted by the
same estimate as before, \eqref{eq:integrability-short}. Therefore, no
real modification is needed in order to prove the continuity of
$S_\si$ in $\Sigma$ at $\si>\si_0(d)$.

\section{Global dynamics for general powers}
\label{sec:global-gen}

\subsection{Convergence in Wasserstein distance: proof of
  Theorem~\ref{theo:W1}}

Again, as the equation
\eqref{eq:NLS} is reversible, we focus on positive time only.
The general principle is the following.  It
follows from 
Theorem~\ref{theo:local} that for any given $T>0$,
\begin{equation*}
\limsup_{\nu\to \si}  \sup_{t\in [0, T]}\|\varrho_\nu (t) -
\varrho_\si(t)\|_{L^1}\lesssim \limsup_{\nu\to \si}
\|\phi_\nu-\phi_\si\|_{L^2}.
\end{equation*}
Details are given below. For $T$ sufficiently large, the variations of
$\varrho_\nu$ are negligible when $t\ge T$. If we worked in a normed
space $X$, we would write, for $t>T$,
\begin{align*}
  \|\varrho_\nu(t)-\varrho_\si(t)\|_X
  & =  \left\|\varrho_\nu(T)-\varrho_\si(T)+\int_T^t
    \d_t\( \varrho_\nu(s)-\varrho_\si(s)\)\dd s\right\|_X\\
  &\le \|\varrho_\nu(T)-\varrho_\si(T)\|_X + \int_T^t \|\d_t
    \varrho_\nu(s)\|_X\dd s+ \int_T^t \|\d_t \varrho_\si(s)\|_X\dd s.
\end{align*}
If $X$ is such that we can invoke Theorem~\ref{theo:local}, 
\begin{equation*}
  \limsup_{\nu\to \si}  \sup_{t\in\R}\|\varrho_\nu(t)-\varrho_\si(t)\|_X\le \sup_\nu
  \int_T^\infty \|\d_t 
    \varrho_\nu(s)\|_X\dd s+ \int_T^\infty\|\d_t \varrho_\si(s)\|_X\dd s, 
  \end{equation*}
and it suffices to show that the right hand side is an  $o(1)$ as
$T\to \infty$. This in turn follows from the continuity equation
\begin{equation*}
  \d_t \varrho_\si +\frac{1}{\<t\>^2}\diver j_\si=0,\quad\text{where
  }j_\si(t,y) = \frac{1}{\|\phi_\si\|_{L^2}^2 }\IM \(\bar v_\si\nabla
  v_\si\) , 
\end{equation*}
where $v_\si=v_\si(t,y)$ is defined by the  transform
\begin{equation*}
  u_\si(t,x) =
  \frac{1}{\<t\>^{d/2}}v_\si\(t,\frac{x}{\<t\>}\)\exp\(i\frac{t}{1+t^2}
    \frac{|x|^2}{2}\), 
\end{equation*}
and solves
\begin{equation*}
  i\d_t v_\si +\frac{1}{2\<t\>^2}\Delta v_\si  =
  \frac{|y|^2}{2\<t\>^{2}}v_\si+ \frac{1}{\<t\>^{d\si}}|v_\si|^{2\si}v_\si.
\end{equation*}
We obviously have $ \|v_\si(t)\|_{L^2}=\|\phi_\si\|_{L^2}$.  In view
of Lemma~\ref{lem:apriori-pseudoconf}, we also have the a priori 
estimates
\begin{align}
  \label{eq:v-apriori-pseudo1}
  &\|yv_\si(t)\|_{L^2} =
   \frac{1}{\<t\>}\|xu_\si(t)\|_{L^2} \le
   \frac{1}{\<t\>}\|J(t)u_\si\|_{L^2}+\|\nabla
   u_\si(t)\|_{L^2}\le  C\(\|\phi_\si\|_\Sigma\)  ,\\
 \label{eq:v-apriori-pseudo2} &\|\nabla
  v_\si(t)\|_{L^2}= \|(x+i\<t\>\nabla)u_\si\|_{L^2}\le
    \<t\>^{\max(0,1-d\si/2)}C\(\|\phi_\si\|_\Sigma\). 
\end{align}
First, we show that Theorem~\ref{theo:local} implies that for any
$T>0$,
\begin{equation*}
  \limsup_{\nu\to\si}\sup_{t\in [0,T]}W_1\(\varrho_\nu(t),\varrho_\si(t)\) =0.
\end{equation*}
In view of \eqref{eq:v-apriori-pseudo1}, for
$\eta>0$ sufficiently small, $0<(1-\eta)\si<(1+\eta)\si<\frac{2}{(d-2)_+}$, and there exists $M_2$ such that
\begin{equation*}
  \sup_{|\nu-\si|\le \eta}\sup_{t\in
    \R}\int_{\R^d}|y|^2\varrho_\nu(t,y)\dd y\le M_2. 
\end{equation*}
In the integral from \eqref{eq:W1} for the characterization of the Wasserstein
distance $W_1$, since $\mu_1$ and $\mu_2$ have the same mass, we may
replace $\psi$ 
with $\psi-\psi(0)$, and thus suppose that
$\psi(0)=0$. We have in particular
\begin{equation*}
  |\psi(x)|\le |x|,\quad \forall x\in \R^d. 
\end{equation*}
We thus write, for $t\ge 0$,
\begin{align*}
  W_1\(\varrho_\nu(t),\varrho_\si(t)\)
  &=\sup_{\|\psi\|_{\rm Lip}\le 1}\int_{\R^d}
    \psi(y)\(\varrho_\nu(t,y)-\varrho_\si(t,y)\)\dd y\\
 & \le \int_{\R^d} |y|\left\lvert
   \varrho_\nu(t,y)-\varrho_\si(t,y)\right\rvert \dd y\\
&\le R \int_{|y|<R}  \left\lvert
  \varrho_\nu(t,y)-\varrho_\si(t,y)\right\rvert \dd y \\
  &\quad + \frac{1}{R}
  \int_{|y|\ge R} |y|^2 \left\lvert 
    \varrho_\nu(t,y)-\varrho_\si(t,y)\right\rvert \dd y  \\
 &\le R\|\varrho_\nu(t)-\varrho_\si(t)\|_{L^1}+2\frac{M_2}{R}.
\end{align*}
Optimizing in $R$ by equating the last two terms, we infer
\begin{equation*}
  W_1\(\varrho_\nu(t),\varrho_\si(t)\)\le 2\sqrt{2M_2
    \|\varrho_\nu(t)-\varrho_\si(t)\|_{L^1}}. 
\end{equation*}
Now we show that Theorem~\ref{theo:local} and the assumptions that
$\phi_\nu$ converges to $\phi_\si$ in $L^2$ imply that for any $T>0$,
\begin{equation*}
  \limsup_{\nu\to \si}\sup_{t\in
    [0,T]}\|\varrho_\nu(t)-\varrho_\si(t)\|_{L^1}=0. 
\end{equation*}
Indeed,
\begin{align*}
  \|\varrho_\nu(t)-\varrho_\si(t)\|_{L^1}
  & = \left\| \<t\>^d \frac{|u_\nu(t,y\<t\>)|^2}{\|\phi_\nu\|_{L^2}^{2}}
    -\<t\>^d
    \frac{|u_\si(t,y\<t\>)|^2}{\|\phi_\si\|_{L^2}^{2}}\right\|_{L^1}\\ 
&  = \left\|  \frac{|u_\nu(t)|^2}{\|\phi_\nu\|_{L^2}^{2}}
    -\frac{|u_\si(t)|^2}{\|\phi_\si\|_{L^2}^{2}}\pm
   \frac{|u_\nu(t)|^2}{\|\phi_\si\|_{L^2}^{2}}\right\|_{L^1} \\
  &\le \frac{1}{\|\phi_\si\|_{L^2}^{2}}\left\|  |u_\nu(t)|^2-
    |u_\si(t)|^2\right\|_{L^1} + \|\phi_\nu\|_{L^2}\left|
    \frac{1}{\|\phi_\nu\|_{L^2}^{2}}-
    \frac{1}{\|\phi_\si\|_{L^2}^{2}}\right|\\
  & \le
    \frac{1}{\|\phi_\si\|_{L^2}^{2}}\(\|\phi_\nu\|_{L^2}+\|\phi_\si\|_{L^2}\)
    \left\|  u_\nu(t)-
    u_\si(t)\right\|_{L^2} + o_{\nu\to \si}(1),
\end{align*}
hence the claimed convergence.

Now for $t>T$, we write
\begin{align*}
    W_1\(\varrho_\nu(t),\varrho_\nu(T)\)
  &=\sup_{\|\psi\|_{\rm Lip}\le 1}\int_{\R^d}
    \psi(y)\(\varrho_\nu(t,y)-\varrho_\nu(T,y)\)\dd y\\
&    =\sup_{\|\psi\|_{\rm Lip}\le 1}\int_T^t \int_{\R^d}
  \psi(y)\d_t\varrho_\nu(s,y)\dd s \dd y\\
&=\sup_{\|\psi\|_{\rm Lip}\le 1}\int_T^t \int_{\R^d}
  \psi(y)\frac{1}{\<s\>^2}\diver j_\nu (s,y)\dd s \dd y\\
  &=\sup_{\|\psi\|_{\rm Lip}\le 1}\int_T^t \int_{\R^d}
    \nabla \psi(y)\frac{1}{\<s\>^2}j_\nu (s,y)\dd s \dd y\\
& \le \int_T^\infty \frac{1}{\<s\>^2}\|j_\nu(s)\|_{L^1}\dd s\lesssim \int_T^\infty
\frac{1}{\<s\>^2}\times \<s\>^{1-\min\(1,\frac{d\nu}{2}\)}  \dd s, 
\end{align*}
where we have used Cauchy-Schwarz inequality and
\eqref{eq:v-apriori-pseudo2}. We infer 
\begin{equation*}
  \sup_{t\ge T} W_1\(\varrho_\nu(t),\varrho_\nu(T)\)=o(1) \quad\text{as
  }T\to \infty,
\end{equation*}
uniformly in $\nu\in [(1-\eta)\si,(1+\eta)\si]$, hence
Theorem~\ref{theo:W1}. 

\subsection{Interpolation}
\label{sec:interpolation}
We now prove Corollary~\ref{cor:interpolation}. 
We know from \cite[Lemma~2.1]{HaurayMischler2014} that for any
$s>\frac{d+1}{2}$, there exists $C=C(s,d)$ such that any probability
measures $f$ and $g$,
\begin{equation*}
  \|f-g\|_{H^{-s}(\R^d)}\le C W_1(f,g)^{1/2}. 
\end{equation*}
Therefore, Theorem~\ref{theo:W1} implies that 
\begin{equation}\label{eq:CV-Hs}
 \forall s>\frac{d+1}{2},\quad  \sup_{t\in
   \R}\|\varrho_\nu(t)-\varrho_\si(t)\|_{H^{-s}(\R^d)}\Tend \nu 
  \si 0. 
\end{equation}
According to the value of $\si$ relative to $2/d$, we have an extra
boundedness property for $(\varrho_\nu)_\nu$ in some Sobolev space
$W^{k,p}$ (with $k=0$ or $1$), and by interpolation, we infer
convergence to zero in intermediate Sobolev spaces.

\subsubsection*{Case $\si>2/d$}

If $\si> 2/d$, we know in addition that for $\nu$ in a
neighborhood of $\si$, $\nu\ge 2/d$, hence $v_\nu\in
L^\infty(\R;H^1)$. 

$\bullet$ In the case $d=1$, $H^1$ is an algebra, so $\varrho_\nu\in
L^\infty(\R;H^1)$. Combined with the convergence \eqref{eq:CV-Hs},
this property implies that
\begin{equation*}
 \forall s<1,\quad   \sup_{t\in \R}\|\varrho_\nu(t)-\varrho_\si(t)\|_{H^{s}(\R^d)}\Tend \nu
  \si 0. 
\end{equation*}
In particular, the convergence holds in $L^\infty(\R;L^2(\R))$, and
since $\varrho_\nu\in L^\infty(\R;L^1)$,
\begin{equation*}
 \forall p\in (1,\infty],\quad   \sup_{t\in
   \R}\|\varrho_\nu(t)-\varrho_\si(t)\|_{L^p(\R^d)}\Tend \nu 
  \si 0. 
\end{equation*}
$\bullet$ In the case $d=2$, $H^1(\R^2)$ barely fails to be an
algebra, and we check that $\varrho_\nu\in
L^\infty(\R;W^{1,p}\cap L^q)$ for all $p\in [1,2)$ and all $q\in
[1,\infty)$. Indeed, $v_\nu\in 
L^\infty(\R;L^p)$ for all $p\in [2,\infty)$, and $\nabla v_\nu\in
L^\infty(\R;L^2)$, hence the announced property by H\"older
inequality. Therefore, $\varrho_\nu$ converge toward $\varrho_\si$ in
$L^\infty(\R;X)$ for any space $X$ obtained by (real) interpolation
between $H^{-s}$ and $W^{1,p}$, and different from $W^{1,p}$. We find
in \cite{Triebel}, Theorem~1 in Section~2.4.2 (pp.~184-185) that this
space is $W^{s_\theta,p_\theta}$, where
\begin{equation*}
  s_\theta = -s\theta + (1-\theta),\quad \frac{1}{p_\theta}
  =\frac{\theta}{2}+\frac{1-\theta}{p},\quad \theta\in [0,1).  
\end{equation*}
We do not pursue the computation in full detail, and simply note that
for fixed $s$, we can find $\theta_0\in (0,1)$ such that
$s_{\theta_0}=0$. By interpolating with the boundedness in
$L^\infty L^q$ for all $q\in [1,\infty)$, we obtain the convergence in
$L^\infty L^q$ for all $q\in (1,\infty)$. 

$\bullet$ In the case $d\ge 3$, $H^1(\R^d)\hookrightarrow
L^{\frac{2d}{d-2}(\R^d)}$, and H\"older inequality implies that $\varrho_\nu\in
L^\infty(\R;W^{1,\frac{d}{d-1}})$, since
\begin{equation*}
  \frac{1}{2}+\frac{d-2}{2d} = \frac{d-1}{d}. 
\end{equation*}
Like in the case $d=2$, we simply note that there exists $\theta\in
(0,1)$ such that
\begin{equation*}
  \left[ H^{-s},W^{1,\frac{d}{d-1}}\right]_\theta = L^{p(d)},
\end{equation*}
for some $p(d)\in (1,2)$ whose exact value we do not compute. Combining
this with the boundedness in $L^\infty(\R;L^1)$, we obtain the
convergence in $L^p$ for all $p\in (1,p(d)]$. 

\subsubsection*{Case $\si<2/d$}
In this case, we only know that $v_\nu\in
L^\infty(\R;L^{2\si+2}\cap L^2)$, hence $\varrho_\nu\in
L^\infty(\R;L^{\si+1}\cap L^1)$, that is $\varrho_\nu\in
L^\infty(\R;W^{0,p})$ for all $p\in [1,1+\si]$. The interpolation
between $H^{-s}$ and $W^{0,p}$ is now given by
\begin{equation*}
   \left[ H^{-s},W^{0,p}\right]_\theta =W^{s_\theta,p_\theta},\quad
   s_\theta = -\theta s
   ,\quad\frac{1}{p_\theta}=\frac{\theta}{2}+\frac{1-\theta}{p},\quad
   \theta\in [0,1].
\end{equation*}
The result stated in Corollary~\ref{cor:interpolation} follows, by
considering the limit $\theta\to 0$.
\smallbreak

In the case $\si=2/d$, considering the case $\nu>\si$ makes it
possible to use the results from the case $\si>2/d$, while if
$\nu<2/d$, we have to use the results from the case $\si<2/d$. As the
latter are weaker than the former, the case $\si=2/d$ is included in
the statement corresponding to $\si<2/d$ in
Corollary~\ref{cor:interpolation}.

\section{Analysis of the dispersion equations}
\label{sec:ODE}
The proof of the convergence of the solution $\u_\si$ of the power
nonlinear Schr\"odinger equation \eqref{eq:rescaledNLS} to the
solution $\u_0$ of the logarithmic Schr\"odinger equation
\eqref{eq:logNLS}  involves ordinary differential equations, which
were used to describe the dispersive rate for $\u_0$ in \cite{CaGa18},
and were generalized in \cite{CCH22} in order to consider the power
case.  In this section, we recall some properties of the solutions to
these ordinary differential equations, and show convergence results at
this level when $\si\to 0$. 
\subsection{Power nonlinearity}

Like in \cite{CCH22}, for $\alpha>0$, consider
\begin{equation}\label{eq:ODEgen}
  \ddot r_\alpha = \frac{\alpha}{2r_\alpha^{\alpha+1}}\quad ;\quad r_\alpha(0)=1,\
  \dot r_\alpha(0)=0.
\end{equation}
Multiply by $\dot r_\alpha$ and integrate between $0$ and $t$:
\begin{equation}\label{eq:ODEgen-int}
  \(\dot r_\alpha\)^2 = 1-\frac{1}{r_\alpha^\alpha}.
\end{equation}
So long as $r_\alpha\ge 0$, $\ddot r_\alpha\ge 0$: $\dot r_\alpha$ is
nondecreasing, hence nonnegative, so $r_\alpha$ is nondecreasing (and
remains positive). If $r_\alpha$ is bounded, then $\ddot r_\alpha\gtrsim 1$,
hence a contradiction by integrating. Therefore, since $r_\alpha$ is
nondecreasing,
\begin{equation*}
  r_\alpha(t)\Tend t {+\infty} +\infty. 
\end{equation*}
Therefore, \eqref{eq:ODEgen-int} implies
\begin{equation*}
  \dot r_\alpha(t)\Tend t {+\infty} 1,
\end{equation*}
and so
\begin{equation*}
  r_\alpha(t)\Eq t {+\infty}t. 
\end{equation*}
We go one step further compared to \cite{CCH22}, and examine the next
term in the asymptotic expansion of $r_\alpha$ as $t\to +\infty$. Write
\begin{equation*}
  r_\alpha(t) = t + w_\alpha(t),
\end{equation*}
so \eqref{eq:ODEgen-int} yields
\begin{equation*}
  \(1+\dot w_\alpha\)^2 = 1-\frac{1}{(t+w_\alpha)^\alpha}.
\end{equation*}
As $w_\alpha(t)=o(t)$,
\begin{equation*}
  \(t+w_\alpha\)^{-\alpha} = t^{-\alpha}\( 1 +\frac{w_\alpha}{t}\)^{-\alpha} =
  t^{-\alpha}\( 1 -\alpha\frac{w_\alpha}{t} +\O\(\(\frac{w_\alpha}{t}\)^2\)\).
\end{equation*}
At leading order, we have
\begin{equation*}
  2\dot w _\alpha= -\frac{1}{t^\alpha},\quad\text{hence } w_\alpha(t) =
  \frac{1}{2(1-\alpha)} t^{1-\alpha}. 
\end{equation*}
We see that for $\alpha\approx\frac{1}{\ln t}$, $w_\alpha$ becomes
comparable with 
the leading order term, since $t^{-\alpha}$ becomes $\O(1)$.
\begin{remark}
  We note in passing that the function $r_{\rm lin}(t) = \<t\>$
  considered in Theorem~\ref{theo:W1} (and Section~\ref{sec:global-gen})
  is equal to $r_2$, since it solves 
  \begin{equation*}
    \ddot r_{\rm lin}=\frac{1}{r_{\rm lin}^3},\quad r_{\rm
      lin}(0)=1, \ \dot r_{\rm lin}(0)=0.
  \end{equation*}
  This function appears for instance when
  considering the evolution of Gaussian initial data under the linear
  Schr\"odinger flow $e^{i\frac{t}{2}\Delta}$ (see
  e.g. \cite[Section~3]{CaGa18}).   
\end{remark}
\subsection{Logarithmic case and transition}
The equation considered in \cite{CaGa18} reads
\begin{equation}
  \label{eq:ODElog}
  \ddot \tau_0 = \frac{1}{2\tau_0}\quad ;\quad \tau_0(0)=1,\
  \dot \tau_0(0)=0,
\end{equation}
and we recall the main features of the solution (see \cite{CaGa18}
for a complete justification). 
Multiply by $\dot \tau_0$ and integrate between $0$ and $t$:
\begin{equation}\label{eq:ODElog-int}
  \(\dot \tau_0\)^2 = \ln\tau_0,
\end{equation}
which readily implies $\tau_0\ge 1$. Like above, $\tau_0$ is convex, and 
\begin{equation*}
  \tau_0(t)\Tend t {+\infty} +\infty. 
\end{equation*}
This is the beginning of the analysis presented in \cite{CaGa18},
which also yields:
\begin{equation}\label{eq:disp-log}
  \tau_0(t)\Eq t {+\infty} t\sqrt{\ln t},\quad
  \dot\tau_0(t)\Eq t {+\infty} \sqrt{\ln t}.
\end{equation}

We set
\begin{equation*}
  \tau_\si(t) = r_{d\si}\(\frac{t}{\sqrt\si}\).
\end{equation*}
It solves
\begin{equation}\label{eq:tau-alpha}
  \ddot{\tau}_\si=\frac{1}{2\tau_\si^{d\si+1}}\quad
  ;\quad \tau_\si(0)=1,\quad \dot{\tau}_\si(0)=0.
\end{equation}
We have
\begin{equation*}
  \tau_\si(t)\Eq t {+\infty}\frac{t}{\sqrt \si}. 
\end{equation*}
Taking the boundary layer $\si\approx \frac{1}{\ln t}$ identified
above, we have
\begin{equation*}
  \tau_{\frac{1}{\ln t}}(t)\approx t\sqrt{\ln t},
\end{equation*}
which is the same order of magnitude as $\tau_0$. 
\smallbreak

The choice $\alpha=d\si$ is motivated by the forthcoming analysis of
the pseudo-energy $\En_\si$ defined in \eqref{eq:En-sigma} in the
proof of Lemma~\ref{lem:apriori-unif-si-t} (see also \cite{CCH22}). 

\subsection{Convergence}

We want to prove that $\tau_\si \to\tau_0$ as $\si\to
0$. The  convergence cannot 
hold uniformly in time, since the large time behavior of $\tau_0$ is
different from that of $\tau_\si$ for $\si>0$, and the above approach suggests that a
transition occurs for $\si$ of order $1/\ln t$. Subtracting the analogue of
\eqref{eq:ODEgen-int} for $\tau_\si$, from
\eqref{eq:ODElog-int}, we find
\begin{align*}
  \(\dot \tau_0\)^2- \(\dot \tau_\si\)^2
&= \ln\tau_0-\frac{1}{d\si}\(1- \frac{1}{\tau^{d\si}_\si}\)\\
&=
  \frac{1}{d\si}\( \frac{1}{\tau_\si^{d\si}}- 
\frac{1}{\tau_0^{d\si}}\) +
\underbrace{\frac{1}{d\si}\( \frac{1}{\tau_0^{d\si}} 
-1+d\si \ln \tau_0\)}_{=:S_\si}.
\end{align*}
As $\tau_0(t)\ge 1$ for all $t\ge 0$, and since
$e^{-z}-1+z\ge 0$ for all $z\ge 0$,  the source term $S_\si$ is
nonnegative for all $t\ge 0$. This implies $\tau_0(t)\ge
\tau_\si(t)$ and $\dot \tau_0(t)\ge
\dot \tau_\si(t) $ for all $t\ge 0$. Let %$w= \tau_0-\tau_\si$,
\begin{equation*}
  w= \tau_0-\tau_\si,\quad y = w^2
  +(\dot w)^2. 
\end{equation*}
We have
\begin{equation*}
  \ddot w =\frac{1}{2 \tau_0} - \frac{1}{2
    \tau_\si^{1+d\si}}= \frac{1}{2
    \tau_0^{1+d\si}}\(\tau_0^{d\si}-1\)
  +\frac{\tau_\si^{1+d\si}-
    \tau_0^{1+d\si}}{2\tau_\si^{1+d\si}\tau_0^{1+d\si}}. 
\end{equation*}
Writing
\begin{equation*}
  \tau_\si^{1+d\si} = \tau_0^{1+d\si}\(1-\frac{w}{\tau_0}\)^{1+d\si} ,
\end{equation*}
Taylor's formula yields
\begin{equation*}
  \tau_\si^{1+d\si}-    \tau_0^{1+d\si} =-(1+d\si)
  \tau_0^{1+d\si}\frac{w}{\tau_0} \underbrace{\int_0^1
  \(1-\theta\frac{w}{\tau_0}\)^{d\si}\dd \theta}_{\le 1,\text{ since
  }w\ge 0.}.  
\end{equation*}
On the other hand, for any $T>0$ (independent of $\si$), we have,
uniformly in $t\in [0,T]$,
\begin{equation*}
  0\le \frac{1}{2
    \tau_0^{1+d\si}}\(\tau_0^{d\si}-1\)\le C(T)
  \frac{\si}{\tau_0^{1+d\si}}\le C(T)\si. 
\end{equation*}
Therefore,
\begin{equation*}
  \dot y 
= 2w\dot w +
   2\dot w\ddot w\le 2
 w\dot w + C(T)\si
 |\dot w| +  \tau_0^{d\si} |w\dot w|
\end{equation*}
Using again the fact that $\tau_0$ is bounded on $[0,T]$, Young
inequality yields
\begin{equation*}
   \dot y \le y + C(T)^2\si^2 + (\dot w)^2+ C(T) y\le
   K(T)y + C(T)^2\si^2,
\end{equation*}
and we conclude by Gr\"onwall lemma:
\begin{lemma}\label{lem:cvODE-bounded}
  Let $\si>0$. For all $t\ge 0$,
  \begin{equation}
    \label{eq:compar-tau}
    \tau_0(t)\ge \tau_\si(t),\quad \dot\tau_0(t)\ge \dot \tau_\si(t).
  \end{equation}
  Let $T>0$. There exists $C=C(T)>0$ such that for all $\si\in (0,1)$,
  \begin{equation*}
    \sup_{t\in [0,T]} |\tau_0(t)-\tau_\si(t)| +  \sup_{t\in [0,T]}
    |\dot \tau_0(t)-\dot \tau_\si(t)|\le C\si. 
  \end{equation*}
\end{lemma}

Using this first result, we now address the convergence on larger time
intervals:
\begin{proposition}\label{prop:cvODE}
  There exists $C>0$ such that for all $\si\in (0,1]$,
  \begin{equation*}
    |\tau_0(t)- \tau_\si(t)|\le C \si t\(\ln t\)^{3/2},\quad
    |\dot\tau_0(t)- \dot\tau_\si(t)|\le C \si \(\ln t\)^{3/2}, \quad \forall
    t\ge 2.
  \end{equation*}
\end{proposition}
Recall that $w=\tau_0 - \tau_{\si} \ge 0$. As we know that 
\[  \dot{\tau}_{\si}^2 = \frac{1}{d\si} \left( 1 -
    \frac{1}{\tau_{\si}^{d\si}} \right) \quad \text{and} \quad
  \dot{\tau}_0^2=\ln \tau_0,\] 
we can write
\[  \dot{w}= \sqrt{\ln \tau_0} - \frac{1}{\sqrt{d\si}} \sqrt{1 -
    \frac{1}{\tau_{\si}^{d\si}}} = \underbrace{f(0,\tau_0) -
    f(0,\tau_{\si})}_{= A(\tau_0,\tau_{\si})} +
  \underbrace{f(0,\tau_{\si})-f(\si,\tau_{\si})}_{= B}, \] 
where $f(\cdot,z)$ is continuous at every $z>0$,
\[ f(\si,z) := \frac{1}{\sqrt{d\si}}\sqrt{1-\frac{1}{z^{d\si}}} \underset{\si \rightarrow 0}{\longrightarrow} \sqrt{\ln z} =: f(0,z).  \]

To estimate $A$, we write that
\[ \sqrt{\ln \tau_0 } - \sqrt{\ln \tau_{\si}} = \sqrt{\ln \tau_0 } \left( 1 - \sqrt{\frac{\ln \tau_{\si}}{\ln \tau_0}} \right) = \sqrt{\ln \tau_0 }  \left(1 - \sqrt{1- \frac{\ln(\tau_0/\tau_{\si})}{\ln \tau_0} }\right). \]
Denote $X=\frac{\ln(\tau_0/\tau_{\si})}{\ln \tau_0}$ and remark that
as $1 \le \tau_{\si} \le \tau_0$, we have $0 \le X \le 1$. Then
from the estimate $\frac12 X \le 1-\sqrt{1-X} \le X$ for all $X \in
\left[0,1 \right]$, we get 
\[  0 \le A(\tau_0,\tau_{\si})  \le \frac{\ln \left( \frac{\tau_0}{\tau_{\si}}  \right)}{\sqrt{\ln \tau_0}}= \frac{\ln \left( 1+ \frac{\tau_0-\tau_{\si}}{\tau_{\si}}  \right)}{\sqrt{\ln \tau_0}}.  \]
We then denote $Y=\frac{\tau_0-\tau_{\si}}{\tau_{\si}}$, so that using
the inequality
\begin{equation*}
 \ln(1+Y) \le \frac{Y}{1+Y}\(1+\ln(1+Y/2)\), 
\end{equation*}
which is easily proved by studying the map $Y\mapsto (1+Y)\ln(1+Y)
-Y\(1+\ln(1+Y/2)\)$ (its second order derivative is negative for $Y>0$), 
we infer
\[ A(\tau_0,\tau_{\si}) \le \frac{\tau_0-\tau_{\si}}{\tau_0 \sqrt{\ln \tau_0}} \left(1+ \ln \left( \frac{\tau_0+\tau_{\si}}{2\tau_{\si}} \right) \right) .  \]
We now need to bound $\tau_{\si}$ from below. By definition,
$\tau_{\si}(0)=1$ and $\dot{\tau}_\si(0)=0$, so using the previous
estimate on compact time intervals, we know that for all $\epsilon>0$,
there exists $T_0>0$ independent of $\si>0$ such that for $t\ge T_0$,
\[  \dot{\tau}_{\si}(t)^2=\frac{1}{d\si}
  \left(1-\frac{1}{\tau_{\si}(t)^{d\si}} \right) \ge \frac{1}{d\si}
  \left(1-\frac{1}{(1+\epsilon)^{d\si}} \right) \ge C_0 , \] 
with $C_0>0$ independent of $\si\in (0,1]$. Taking the square root and
integrating over time, we then get that there exists $c_0>0$
independent of $\si$ such that 
\[ \tau_{\si}(t) \ge c_0 t   \]
for all $t\ge T_0$. We then infer from the behavior of $\tau_0$ that 
\[ \ln \left(\frac12 + \frac{\tau_0}{\tau_{\si}} \right) \le \ln
  \left( \frac12 +c_0^{-1} \sqrt{\ln t} \right) \le C \ln \ln t ,  \] 
for a constant $C>0$ independent of $\si$. This entails that there
exists a constant $C_A>0$ uniform in $\si\in (0,1]$ such that for all
$t \ge T_0$, 
\[ A(\tau_0,\tau_{\si}) \le C_A \frac{\tau_0-\tau_{\si}}{\tau_0}
  \frac{\ln \ln t}{ \sqrt{\ln \tau_0}}. \]

We now turn to the analysis of $B$: for $z>0$ fixed, Taylor expansion
in $\si$ yields
\[  \frac{1}{z^{d\si}} = 1 - d\si \ln z+ \int_0^{d\si}( d\si - \beta)
  \frac{(\ln z )^2}{z^{\beta}} \dd \beta , \] 
and we write
\[ f(\si,z)= \sqrt{ \ln z - \frac{1}{d\si} \int_0^{d\si} ( d\si -
    \beta) \frac{(\ln z )^2}{z^{\beta}} \dd \beta },  \] 
so
\[ f(0,z)-f(\si,z)=\sqrt{\ln z} \left( 1 - \sqrt{1 -
      \underbrace{\frac{\ln z}{d\si} \int_0^{d\si} \frac{d\si -
          \beta}{z^{\beta}} \dd \beta}_{=: \zeta_{\si}(z)} }
  \right).   \] 
Note that we have $0 \le \zeta_{\si}(z) \le 1$ for all $z \ge
1$. Thanks to the estimate
\begin{equation*}
  1-\sqrt{1-X} \le X,\quad  0 \le X \le 1,
\end{equation*}
we infer that
\[  f(0,\tau_{\si})-f(\si,\tau_{\si}) \le  \frac{(\ln
    \tau_{\si})^{3/2}}{d\si} \int_0^{d\si} \frac{d\si -
    \beta}{\tau_{\si}^{\beta}} \dd \beta \le \frac{d\si}{2} (\ln
  \tau_{0})^{3/2},\] 
using that $\tau_{\si} \le \tau_0$ and  $\tau_{\si} \ge 1$. To
summarize, we have proved that
\begin{equation}
  \label{eq:w-avant-gronwall}
  0 \le \dot{w} \le  \underbrace{\frac{ \ln \ln t}{2\tau_0 \sqrt{\ln
        \tau_0}}}_{=:\psi(t)} w + \underbrace{\frac{d\si}{2} (\ln
    \tau_{0})^{3/2}}_{=:K(t)},\quad \forall t\ge T_0. 
\end{equation}
Gr\"onwall lemma yields
\[ w(t) \le \exp\left(\int_{T_0}^t \psi(s) \dd s\right) w(T_0) + \int_{T_0}^t K(t) \exp\left(\int_s^t \psi(\gamma) \dd \gamma \right) \dd s.  \]
In view of \cite[Lemma~1.6]{CaGa18},
\[ \tau_0(s)= s \sqrt{\ln s} \left( 1 + \O \left(
      \frac{\ln\ln s}{\ln s} \right) \right),  \]
and so
\[ \ln \tau_0(s)=\ln s + \ln \sqrt{\ln s} + \O \left( \frac{\ln\ln
      s}{\ln s} \right).  \] 
We then compute 
\begin{align*}
  \psi(s) & = \frac{\ln \ln s}{2 \tau_0(s) \sqrt{\ln{\tau_0}(s)}} \\
  & = \frac{\ln \ln s}{2 s \sqrt{\ln s} \left(1 + \O \left(
    \frac{\ln\ln s}{\ln s} \right) \right) \sqrt{ \ln s + \ln \sqrt{\ln s} +
    \O \left( \frac{\ln\ln s}{\ln s} \right)} }  \\ 
  & = \frac{\ln \ln s}{2 s \ln s \left(1 + \O \left( \frac{\ln\ln s}{\ln s}
    \right) \right) \sqrt{1 + \frac{\ln \ln s }{2\ln s} +
    \left(1 + \O \left( \frac{\ln\ln s}{(\ln s)^2} \right) \right) }} \\ 
  & = \frac{\ln \ln s}{2 s \ln s } \frac{1}{ 1 +\O\( \frac{\ln\ln
    s}{ \ln s}\) }  = \frac{\ln \ln s}{2 s \ln s } \( 1 +\O\( \frac{\ln\ln
    s}{ \ln s}\)\), 
  \end{align*}
hence, since $s\mapsto \frac{(\ln\ln s)^2}{s(\ln s)^2}$ is integrable
as $s\to \infty$ (it is controlled by $\frac{1}{s(\ln s)^{3/2}}$ for instance),
\[ \int_{T_0}^t \psi(\gamma) \dd \gamma = \frac{1}{4}  (\ln \ln
  t)^2 +\O(1)\quad \text{as }t\to \infty.  \]
We can then estimate the first term in the Gr\"onwall inequality,
using Lemma~\ref{lem:cvODE-bounded} on $[0,T_0]$, so that 
\[ \exp\left(\int_{T_0}^t \psi(s) \dd s\right) w(T_0) \le C(T_0) \si e^{\frac{1}{4}(\ln \ln t)^2} , \]
for a constant $C(T_0)>0$ independent of $\si$ and $t$. For the second
term, we once again use the asymptotic behavior of $\tau_0$ together
with the previous bound on $\psi$ so that 
\begin{align*}
 \int_{T_0}^t K(t) \exp\left(\int_s^t \psi(\gamma) \dd \gamma \right) \dd s & \lesssim \si e^{\frac{1}{4}(\ln \ln t)^2} \int_{T_0}^t (\ln s )^{\frac{3}{2}} e^{-\frac{1}{4}(\ln \ln s )^2} \dd s  \\
 & \lesssim  \si e^{\frac{1}{4}(\ln \ln t)^2}  (\ln t )^{\frac{3}{2}} \int_{T_0}^t e^{-\frac{1}{4}(\ln \ln s )^2} \dd s.
 \end{align*}
The last integral diverges when $t \to +\infty$, as can be seen by
integrating by parts: 
 \[ \int_{T_0}^t e^{-\frac{1}{4}(\ln \ln s )^2} \dd s = \left[ s e^{-\frac{1}{4}(\ln \ln s )^2}  \right]^{t}_{T_0} +  \frac{1}{2} \int_{T_0}^t  \frac{\ln \ln s }{ \ln s} e^{-\frac{1}{4}(\ln \ln s )^2} \dd s. \]
Since $ \frac{\ln \ln s }{ \ln s} \to 0$ as $s\to \infty$, the
boundary term (at $s=t$) provides the leading order behavior in the 
asymptotic expansion as  $t \to 
\infty$, so finally we have   
 \[  \int_{T_0}^t K(t) \exp\left(\int_s^t \psi(\gamma) \dd \gamma \right) \dd s \lesssim  \si t (\ln t)^{3/2}.    \]
 Going back to the Gr\"onwall inequality, we finally get by comparing behaviors that
\[ w(t) \lesssim \si t (\ln t)^{3/2}.  \]
In the regime $\si = \frac{1}{\ln t}$, we recover the time growth of
$w$, induced by $\tau_0$. Going back to \eqref{eq:w-avant-gronwall},
we infer the announced estimate for $\dot w= \dot
\tau_0-\dot\tau_\si$.

\section{Uniform estimates}
\label{sec:apriori}

In this section, we prove large time a priori estimates for the
solutions to \eqref{eq:rescaledNLS} and \eqref{eq:logNLS}. To do so, we
modulate the corresponding solutions $\u_\si$ and $\u_0$ by using the
dispersion rates $\tau_\si$ and $\tau_0$ considered in
Section~\ref{sec:ODE}.

\subsection{Dispersion rescaling: logarithmic case}
Let $\tau_0$ denote the solution to \eqref{eq:ODElog}. As established
in \cite{CaGa18} (with a different normalization for $\tau_0$), the
function $v_0$, related to the solution $\u_0$ to 
\eqref{eq:logNLS} via
\begin{equation}\label{eq:u-v-log}
   \u_0(t,x) =
  \frac{1}{\tau_0(t)^{d/2}}v_0\(t,\frac{x}{\tau_0(t)}\)
e^{i\frac{\dot\tau_0(t)}{\tau_0(t)}\frac{|x|^2}{2}},
\end{equation}
which is nothing but \eqref{eq:u-v} with indices $0$, solves:
\begin{equation*}
  i{\partial}_t v_0 +\frac{1}{2\tau_0(t)^2}\Delta  v_0 =
 \frac{|y|^2}{4}v_0+ v_0\ln(|v_0|^2)- d v_0\ln \tau_0. 
\end{equation*}
The last term is removed by an explicit time-dependent gauge
transform. 
The pseudo-energy
\begin{equation*}
  \En_0(t)= \frac{1}{2\tau_0(t)^2}\|\nabla v_0(t)\|_{L^2}^2 
+ \frac{1}{4}\|yv_0(t)\|_{L^2}^2+ \int_{\R^d} |v_0(t,y)|^2\ln
(|v_0(t,y)|^2)\dd y
\end{equation*}
satisfies
\begin{equation*}
  \dot\En_0= -2\frac{\dot
    \tau_0}{\tau_0}\times\frac{1}{2\tau_0(t)^2}\|\nabla_y v_0(t)\|_{L^2}^2 .
\end{equation*}
We infer the a priori estimates (see \cite{CaGa18} for details)
\begin{equation}\label{eq:apriori-log-v0}
 \sup_{t\ge 0}\left(\int_{{\mathbb R}^d}\left(1+|y|^2+\left|\ln
    |v_0(t,y)|^2\right|\right)|v_0(t,y)|^2\dd y +\frac{1}{\tau_0(t)^2}\|\nabla
  v_0(t)\|^2_{L^2}\right)<\infty
\end{equation} 
and
\begin{equation}\label{eq:integralkin}
 \int_0^\infty \frac {\dot \tau_0(t)}{\tau_0^3(t)}\|\nabla
  v_0(t)\|^2_{L^2({\mathbb R}^d)} \dd t<\infty.
\end{equation}
Back to $\u_0$, \eqref{eq:u-v-log}, \eqref{eq:disp-log}
and \eqref{eq:apriori-log-v0} yield the bound
\begin{equation}
  \label{eq:u0-H1}
  \|\nabla \u_0(t)\|_{L^2(\R^d)}\lesssim \<\ln t\>^{1/2}.
\end{equation}
If is shown in \cite{CaGa18} that this bound is sharp, a property that
we do not use here. 
\subsection{Toward the logarithmic nonlinearity}

If we consider \eqref{eq:rescaledNLS}, the
change of unknown 
\begin{equation}
  \label{eq:u-v}
  \u_\si(t,x) =
  \frac{1}{\tau_\si(t)^{d/2}}v_\si\(t,\frac{x}{\tau_\si(t)}\)
e^{i\frac{\dot\tau_\si(t)}{\tau_\si(t)}\frac{|x|^2}{2}}, 
\end{equation}
leads to
\begin{equation*}
  i\d_t v_\si +\frac{1}{2\tau_\si^2}\Delta v_\si
  =\frac{|y|^2}{4\tau_\si^{d\si}}v_\si
  +\frac{1}{\si\tau_\si^{d\si}}|v_\si|^{2\si}v_\si -
    \frac{1}{\si}v_\si. 
\end{equation*}
Up to another gauge transform, we get
\begin{equation}\label{eq:v-sigma}
  i\d_t v_\si +\frac{1}{2\tau_\si^2}\Delta v_\si
  =\frac{|y|^2}{4\tau_\si^{d\si}} v_\si
  +\frac{1}{\si\tau_\si^{d\si}}\(|v_\si|^{2\si}-1\)v_\si . 
\end{equation}
This means that we have replaced the initial $v_\si$ with 
\begin{equation}\label{eq:vtilde-v}
  \tilde v_\si (t,y) = v_\si(t,y) \exp \( -i\frac{t}{\si} +
  i\int_0^t\frac{\dd s}{\si\tau_\si(s)^{d\si}}\). 
\end{equation}
We have dropped the tildas to lighten notations.
We note that if we forget that $\tau_\si$ solves \eqref{eq:tau-alpha}, then
the coefficient of the harmonic potential is given more generally by
$\tau_\si\ddot \tau_\si/2$. The choice $\alpha=d\si$ in
Section~\ref{sec:ODE} is made to ensure that the power of $\tau_\si$
in front of the two terms on the right hand side of \eqref{eq:v-sigma}
is the same, a property which is crucial in 
the proof of the following lemma:
\begin{lemma}\label{lem:apriori-unif-si-t}
Let $\eps\in (0,1/d)$, and $(\phi_\si)_{0\le \si\le \eps}$ bounded in
$\Sigma$. For $\u_\si$ solution to 
  \eqref{eq:rescaledNLS}, define $v_\si$ by
  \begin{equation*}
    \u_\si(t,x) =
    \frac{1}{\tau_\si(t)^{d/2}}v_\si\(t,\frac{x}{\tau_\si(t)}\)
    \exp\(i\frac{\dot\tau_\si(t)}{\tau_\si(t)}\frac{|x|^2}{2}\),
  \end{equation*}
  where $\tau_\si$ solves \eqref{eq:tau-alpha}.
  There exists $C$ independent of $t\ge 0$ and $\si\in (0,1/d)$ such
  that
\begin{align*}
  \frac{1}{\tau_\si(t)^{2-d\si}}\|\nabla v_\si(t)\|_{L^2}^2 
  + \|yv_\si(t)\|_{L^2}^2
  +
\int_{\R^d}
\left|\frac{|v_\si(t,y)|^{2\si}-1}{\si}\right| \lvert v_\si(t,y)\rvert^2\dd y\le C,
\end{align*}
and 
\begin{equation*}
  \int_0^\infty \frac{\dot \tau_\si(t)}{\tau_\si(t)^{3-d\si}}\|\nabla
  v_\si(t)\|_{L^2}^2\dd t\le C. 
\end{equation*}
\end{lemma}
\begin{proof}
  Introduce the
pseudo-energy 
\begin{equation}\label{eq:En-sigma}
  \begin{aligned}
    \En_\si(t)&= \frac{1}{2\tau_\si(t)^2}\|\nabla v_\si(t)\|_{L^2}^2 
+ \frac{1}{4\tau_\si(t)^{d\si}}\|yv_\si(t)\|_{L^2}^2\\
&\quad +
\frac{1}{(\si+1)\tau_\si(t)^{d\si}}\int_{\R^d}
  \(\frac{|v_\si(t,y)|^{2\si}-1}{\si}\)|v_\si(t,y)|^2\dd y.
  \end{aligned}
\end{equation}
  Since the purely time dependent phase function in \eqref{eq:vtilde-v} does not affect
  $ \mathcal E_\si$, we may consider that $v_\si$ solves
  \eqref{eq:v-sigma}.
  We compute
  \begin{align*}
  \dot\En_\si(t)
  &=-\frac{\dot \tau_\si(t)}{\tau_\si(t)}\Bigg( \frac{1}{\tau_\si(t)^2}\|\nabla v_\si(t)\|_{L^2}^2
   +\frac{d\si}{4\tau_\si(t)^{d\si}} \|y v_\si(t)\|_{L^2}^2 \\
 &\quad
   + \frac{d\si}{(\si+1)\tau_\si(t)^{d\si}}\int_{\R^d}
\(\frac{|v_\si(t,y)|^{2\si}-1}{\si}\)|v_\si(t,y)|^2\dd y           
\Bigg) .
\end{align*}
We note that, like $\En_0$, $\En_\si$ is not the sum of positive terms, and decompose
it as $\En_\si=\En_\si^+-\En_\si^-$, where 
\begin{align*}
  \En_\si^+(t)&= \frac{1}{2\tau_\si(t)^2}\|\nabla v_\si(t)\|_{L^2}^2 
+ \frac{1}{4\tau_\si(t)^{d\si}}\|yv_\si(t)\|_{L^2}^2\\
&\quad +
\frac{1}{(\si+1)\tau_\si(t)^{d\si}}\int_{|v_\si|\ge 1}
  \(\frac{|v_\si(t,y)|^{2\si}-1}{\si}\)|v_\si(t,y)|^2\dd y,\\
   \En_\si^-(t)&=\frac{1}{(\si+1)\tau_\si(t)^{d\si}}\int_{|v_\si|< 1}
  \(\frac{1-|v_\si(t,y)|^{2\si}}{\si}\)|v_\si(t,y)|^2\dd y.
\end{align*}
Now $\En_\si^+$ is the sum of three nonnegative terms, and $\En_\si^-$
is nonnegative. Taylor formula for the function
$f(\si)=y^\si$ yields 
\begin{equation*}
  \frac{1-y^\si}{\si} = \ln \frac{1}{y}\int_0^1 y^{\theta \si}\dd\theta.
\end{equation*}
We infer
\begin{align*}
  \En_\si^-(t)&\le \frac{1}{(\si+1)\tau_\si(t)^{d\si}}\int_{|v_\si|< 1}
             |v_\si(t,y)|^2\ln \frac{1}{|v_\si(t,y)|^2}\dd y\\
  &\lesssim \frac{C(\eta)}{(\si+1)\tau_\si(t)^{d\si}}\int_{\R^d}
             |v_\si(t,y)|^{2-2\eta}\dd y,
\end{align*}
where $\eta>0$ is arbitrarily small. In view of the conservation of
the mass and \eqref{eq:GNdual}, this implies
\begin{equation}\label{eq:est-En-}
  \En_\si^- (t)\lesssim \frac{1}{\tau_\si(t)^{d\si}}\|y
  v_\si(t)\|_{L^2}^{\frac{d\eta}{2-2\eta}}\lesssim
  \frac{1}{\tau_\si(t)^{d\si}}\(
  \tau_\si^{d\si}\En_\si^+\)^{\frac{d\eta}{1-1\eta}} .
\end{equation}
This implies in particular that $\En_\si(0)$ is bounded uniformly in
$\si\in (0,1/d)$. 
In view of the derivative of $\En_\si$, we also have
\begin{equation*}
  \frac{\dd}{\dd t}\(\tau_\si^{d\si}\En_\si\) = -\dot\tau_\si
  \tau^{d\si-1}\(1-\frac{d\si}{2}\) \frac{1}{\tau_\si^2}\|\nabla
  v_\si\|_{L^2}^2\le -\dot\tau_\si
  \tau_\si^{d\si-1} \frac{1}{2\tau_\si^2}\|\nabla
  v_\si\|_{L^2}^2,
\end{equation*}
where we have used the facts that $\tau_\si,\dot\tau_\si\ge 0$ and
$\si<1/d$. We infer the uniform bound
\begin{equation*}
  \En_\si(t)\le \frac{\En_\si(0)}{\tau_\si(t)^{d\si}}\le
  \frac{C}{\tau_\si(t)^{d\si}}. 
\end{equation*}
Invoking \eqref{eq:est-En-}, we have
\begin{equation*}
  \tau_\si^{d\si}\En_\si^+\le C + \tau_\si^{d\si}\En_\si^- \lesssim 1
  + \(  \tau_\si^{d\si}\En_\si^+\)^{\frac{d\eta}{1-\eta}}.
\end{equation*}
Taking $\si>0$ such that $\frac{d\eta}{1-\eta}<1$ shows that
$\tau_\si^{d\si}\En_\si^+$ is bounded uniformly in $t\ge 0$ and
$\si\in (0,1/d)$. Again from \eqref{eq:est-En-}, this implies that so
is $\tau_\si^{d\si}\En_\si^-$, hence
$\tau_\si^{d\si}\(\En_\si^++\En_\si^-\)\le C$, which is the first
claim of the lemma. We infer that $\tau_\si^{d\si}\En_\si$ is uniformly bounded from
below, so its (nonpositive) derivative is integrable,
\begin{equation*}
   \int_0^\infty
   \frac{\dot\tau_\si(t)}{\tau_\si(t)^{3-d\si}}\|\nabla
  v_\si(t)\|_{L^2}^2\dd t\le C,
\end{equation*}
which completes the proof. 
\end{proof}

\section{From power to logarithmic nonlinearity:
  Ehrenfest time}
\label{sec:ehrenfest}

Since the rest of the paper is dedicated to the limit $\si\to 0$, we
shall assume from now on that $0<\si<1/(2d)$, and
$\eps\le 1/(2d)$ in the assumption of Theorem~\ref{theo:log-loc-temps}.

\subsection{Proof of Theorem~\ref{theo:log-loc-temps}}
In view of Lemma~\ref{lem:apriori-unif-si-t}, since $\tau_\si(t),\dot
\tau_\si(t)\ge 0$ for all $t\ge 0$,
we have
\begin{equation}\label{eq:borne-unif-vers-log}
  \begin{aligned}
  &\|\nabla \u_\si(t)\|_{L^2} \le \frac{1}{\tau_\si(t)}\|\nabla
  v_\si(t)\|_{L^2} + \dot \tau_\si(t) \|yv_\si(t)\|_{L^2}\le C \(
    \tau_\si(t)^{d\si/2} + \dot \tau_\si(t)\),\\
  &\|x\u_\si(t)\|_{L^2} =\tau_\si(t)\|yv_\si(t)\|_{L^2}\le C\tau_\si(t),
\end{aligned}
\end{equation}
where $C$ is independent of $\si\in (0,\eps)$ and $t\ge 0$. 
The difference $w=\u_\si-\u_0$ solves
\begin{equation*}
    i\d_t w +\frac{1}{2}\Delta w =
    \frac{1}{\si}\(|\u_\si|^{2\si}-1\)\u_\si - \u_\si\ln(|\u_\si|^2) +  \u_\si\ln(|\u_\si|^2) -\u_0\ln(|\u_0|^2),
\end{equation*}
with initial value $w_{\mid t=0}=\phi_\si-\phi_0$.  The source term is
\begin{equation}\label{eq:S_si}
  S_\si=  \frac{1}{\si}\(|\u_\si|^{2\si}-1\)\u_\si - \u_\si\ln(|\u_\si|^2) .
\end{equation}
Taylor formula applied to $\si\mapsto
(|z|^{2\si}-1)z=(e^{\si \ln|z|^2}-1)z$  yields
\begin{equation}\label{eq:S_si-Taylor}
  S_\si 
= \si \u_\si \(\ln(|\u_\si|^2)\)^2 \int_0^1 (1-\theta)
    |\u_\si|^{2\theta\si}\dd\theta, 
\end{equation}
hence the pointwise bound
\begin{equation*}
  |S_\si |\lesssim \si \(\ln(|\u_\si|^2)\)^2 \(
  |\u_\si|+|\u_\si|^{2\si+1}\). 
\end{equation*}
Let $\eta>0$: for $0<\si\le\eta/4$,  we have the uniform  pointwise estimate
\begin{equation*}
  |S_\si|\lesssim \si \(|\u_\si|^{1-\eta} + |\u_\si|^{1+\eta} \),
\end{equation*}
where the implicit constant depends on $\eta>0$. 
Gagliardo-Nirenberg inequality \eqref{eq:GN} yields, in view of
\eqref{eq:borne-unif-vers-log}, 
\begin{equation*}
 \left\||\u_\si|^{1+\eta}\right\|_{L^2}=
 \|\u_\si\|^{1+\eta}_{L^{2+2\eta}}\lesssim \|\u_\si\|_{L^2}^{1+\eta-d\eta/2}\|\nabla
 \u_\si\|_{L^2}^{d\eta/2} \lesssim \(\tau_\si(t)^{d\si/2}+\dot \tau_\si(t)\)^{d\eta/2}.
\end{equation*}
Similarly, \eqref{eq:GNdual} yields, in view of \eqref{eq:borne-unif-vers-log},
\begin{equation*}
 \left\||\u_\si|^{1-\eta}\right\|_{L^2}=
 \|\u_\si\|^{1-\eta}_{L^{2-2\eta}}\lesssim \|\u_\si\|_{L^2}^{1-\eta-d\eta/2}\|x
 \u_\si\|_{L^2}^{d\eta/2} \lesssim \tau_\si(t)^{d\eta/2}.
\end{equation*}
We infer
\begin{equation}\label{eq:source-log}
  \|S_\si(t)\|_{L^2}\lesssim \si\( \tau_\si(t)^{d\si/2}
  +\tau_\si(t)+\dot \tau_\si(t)\)^{d\eta/2}. 
\end{equation}
To complete the energy estimate, recall an identity discovered in
\cite{CaHa80}: 
\begin{lemma}[From Lemma~1.1.1 in \cite{CaHa80}]\label{lem:CH}
  There holds
  \begin{equation*}
    \left|\operatorname{Im}\left(\left(z_2 \log
 \left|z_2\right|^2-z_1 \log
 \left|z_1\right|^2\right)\left(\overline{z_2}-
 \overline{z_1}\right)\right)\right| 
    \le 2\left|z_2-z_1\right|^2, \quad \forall z_1, z_2 \in
    \mathbb{C}.
  \end{equation*}
\end{lemma}
The energy estimate then yields, together with Cauchy-Schwarz inequality,
\begin{equation*}
  \frac{\dd}{\dd t}\|w\|_{L^2}^2 \le 4 \|w\|_{L^2}^2 +
  2\|w\|_{L^2}\|S_\si\|_{L^2}\le 5\|w\|_{L^2}^2 +\|S_\si\|_{L^2}^2,
\end{equation*}
where we have used Young inequality. Gr\"onwall lemma and
\eqref{eq:source-log} imply: 
\begin{equation*}
  \|w(t)\|_{L^2}^2 \lesssim \|w(0)\|^2_{L^2} e^{5t}+\si^2 \int_0^t
  e^{5(t-s)}\( \tau_\si(s)^{d\si/2}
  +\tau_\si(s)+\dot \tau_\si(s)\)^{d\eta}\dd s.
\end{equation*}
In view of \eqref{eq:disp-log} and \eqref{eq:compar-tau}, we
infer the bound, uniform in $\si\in (0,\eps)$,
\begin{equation*}
   \|w(t)\|^2_{L^2} \lesssim \|w(0)\|^2_{L^2} e^{5t}+\si^2 \int_0^t
  e^{5(t-s)}\< s\sqrt{\ln   s }\>^{d\eta}\dd s.
\end{equation*}
This yields Theorem~\ref{theo:log-loc-temps}.

\subsection{Convergence in Wasserstein distance}
\label{sec:ehrenfest-wasserstein}

Recall that  the probability densities $\varrho_0$ and
$\varrho_\si$ are given by
\begin{equation*}
  \varrho_0(t,y) = \tau_0(t)^d\left\lvert
    \u_0\(t,y\tau_0(t)\)\right\rvert^2 \|\phi_0\|_{L^2}^{-2},\quad
   \varrho_\si(t,y) = \tau_\si(t)^d\left\lvert
    \u_\si\(t,y\tau_\si(t)\)\right\rvert^2 \|\phi_\si\|_{L^2}^{-2},
\end{equation*}
and, to lighten notations, introduce $\rho_0$ and
$\rho_\si$ are given by
\begin{equation*}
  \rho_0(t,y) = \tau_0(t)^d\left\lvert
    \u_0\(t,y\tau_0(t)\)\right\rvert^2,\quad
   \rho_\si(t,y) = \tau_\si(t)^d\left\lvert
    \u_\si\(t,y\tau_\si(t)\)\right\rvert^2.
\end{equation*}
We have established in Section~\ref{sec:apriori} the uniform bounds
for the second momenta of these densities: there exists
$C\(\|\phi\|_\Sigma\)$ independent of $t\in \R$ and $\si<1/(2d)$ such that
\begin{equation}\label{eq:bound-rho-2-mom}
  \int_{\R^d}\(1+|y|^2\) \(\rho_0(t,y) +\rho_\si(t,y)\)\dd y\le
  C\(\|\phi\|_\Sigma\) ,\quad \forall t\ge 0,\  \forall \si\in
  \(0,\frac{1}{2d}\). 
\end{equation}
Using the triangle inequality and changing variables in the integrals,
\begin{align*}
  \int_{\R^d} \left\lvert \rho_0(t,y)-\rho_\si(t,y)\right\rvert \dd
  y
  &\le \int_{\R^d} \left\lvert
    \tau_0(t)^d\left|\u_0\(t,y\tau_0(t)\)\right|^2
    -\tau_\si(t)^d\left|\u_0\(t,y\tau_\si(t)\)\right|^2 \right\rvert \dd 
    y\\
  &\quad + \int_{\R^d} \left\lvert |\u_0(t,x)|^2-|\u_\si(t,x)|^2\right\rvert \dd
    x\\
 &\lesssim \int_{\R^d} \left\lvert
    \left|\u_0\(t,y\)\right|^2
   -\(\frac{\tau_\si(t)}{\tau_0(t)}\)^d\left|\u_0\(t,y
   \frac{\tau_\si(t)}{\tau_0(t)}\)\right|^2\right\rvert \dd   y\\
  &\quad + \|\u_0(t)-\u_\si(t)\|_{L^2}\\
 &\lesssim \left| 1- \(\frac{\tau_\si(t)}{\tau_0(t)}\)^d\right|
   \|\u_0(t)\|_{L^2}^2 \\
  &\quad+ \int_{\R^d} \left\lvert
    \left|\u_0\(t,y\)\right|^2
   -\left|\u_0\(t,y
   \frac{\tau_\si(t)}{\tau_0(t)}\)\right|^2\right\rvert \dd   y
  + C_1\si e^{C_0t},
\end{align*}
where we have used \eqref{eq:compar-tau} and
Theorem~\ref{theo:log-loc-temps} for the last 
inequality. Proposition~\ref{prop:cvODE} implies, in view of the behavior of
$\tau_0$,
\begin{equation*}
  \left| 1- \frac{\tau_\si(t)}{\tau_0(t)}\right|\lesssim
  \si  \ln (t+4),\quad\forall t\ge 0,
\end{equation*}
where we add $4$ to
simplify notations for small $t$.
In view of \eqref{eq:compar-tau}, 
the first term of the previous right hand 
side is controlled by $\si \ln (t+4)$ for $t\ge 0$. For the second
term, in view of  Cauchy-Schwarz inequality,
 \begin{align*}
    \int_{\R^d} \left\lvert
    \left|\u_0\(t,y\)\right|^2
   -\left|\u_0\(t,y
  \frac{\tau_\si(t)}{\tau_0(t)}\)\right|^2\right\rvert \dd y
 &\le
   \(1+\(\frac{\tau_0(t)}{\tau_\si(t)}\)^{d/2}\)\|\u_0(t)\|_{L^2}\\
  &\quad \times 
   \left\| \u_0\(t,y\)-\u_0\(t,y
   \frac{\tau_\si(t)}{\tau_0(t)}\)\right\|_{L^2}.
 \end{align*}
Taylor
formula yields
\begin{equation*}
  \u_0\(t,y\)-\u_0\(t,y
   \frac{\tau_\si(t)}{\tau_0(t)}\)=
   \(1-\frac{\tau_\si(t)}{\tau_0(t)}\)\int_0^1 \nabla \u_0\(t,y\(
   1+\kappa \(\frac{\tau_\si(t)}{\tau_0(t)}-1\)\)\)\dd \kappa.
 \end{equation*}
 In view of \eqref{eq:compar-tau}, 
 \begin{equation*}
   1+\kappa \(\frac{\tau_\si(t)}{\tau_0(t)}-1\)
   \ge\frac{\tau_\si(t)}{\tau_0(t)},\quad\forall \kappa \in
   [0,1],\ \forall t\ge 0. 
 \end{equation*}
Using \eqref{eq:u0-H1}, we
infer, for $0\le t\lesssim \ln\frac{1}{\si}$ (which implies
$\frac{\tau_0(t)}{\tau_\si(t)}\lesssim 1$),
\begin{equation*}
  \int_{\R^d} \left\lvert
    \left|\u_0\(t,y\)\right|^2
   -\left|\u_0\(t,y
  \frac{\tau_\si(t)}{\tau_0(t)}\)\right|^2\right\rvert \dd y
  \lesssim  \si \ln (t+4) \|\nabla \u_0(t)\|_{L^2}\lesssim \si \(\ln
(t+4)\)^{3/2}.
\end{equation*}
This proves strong convergence in $L^1$ for $0\le t\lesssim
\ln\frac{1}{\si}$. Along with the uniform 
boundedness of second momenta \eqref{eq:bound-rho-2-mom},
we have in particular, for all $\theta\in (0,1)$,
\begin{equation*}
 \sup_{|t|\le \frac{\theta}{C_0}\ln\frac{1}{\si}}
 \int_{\R^d}|y|\left\lvert \rho_0(t,y)-\rho_\si(t,y)\right\rvert\dd 
  y\Tend \si 0 0.
\end{equation*}
As the same remains true for $\varrho_0$ and $\varrho_\si$, this
implies that for all $\theta\in (0,1)$, 
\begin{equation*}
 \sup_{|t|\le \frac{\theta}{C_0}\ln\frac{1}{\si}} W_1\(
 \varrho_0(t),\varrho_\si(t)\)\Tend \si 0 0.
\end{equation*}

\section{From power to logarithmic nonlinearity:   global dynamics}
\label{sec:log}

We keep the notations from the previous section. Recall that changing the unknown function like in \eqref{eq:u-v}, up
to a purely time dependent phase function \eqref{eq:vtilde-v}, the new
wave function solves \eqref{eq:v-sigma}. 
From Madelung transform $\rho_{\sigma}=|v_{\sigma}|^2$ and $J_{\sigma}=\IM(\overline{v_{\sigma}} \nabla v_{\sigma})$, we infer the system
\[ \left\{
\begin{aligned}
& \partial_t \rho_{\sigma} + \frac{1}{\tau_{\sigma}^2} \diver J_{\sigma} = 0,\\
& \partial_t J_{\sigma} + \frac{1}{(\sigma +1)
  \tau_{\sigma}^{d\sigma}} \nabla \rho_{\sigma}^{ \sigma +1} +
  \frac{2}{\tau_{\sigma}^{d\sigma}} y \rho_{\sigma} = \frac{1}{4
  \tau_{\sigma}^2} \Delta \nabla \rho_{\sigma} -
  \frac{1}{\tau_{\sigma}^2} \diver \RE ( \nabla v_{\sigma} \otimes
  \nabla \overline{v_{\sigma}} ) . 
\end{aligned} \right.  \]
Combining these equations and writing $\nu_{\sigma}=\RE ( \nabla v_{\sigma} \otimes \nabla \overline{v_{\sigma}} )$ leads to the following non-autonomous fractional Fokker-Planck equation with source terms
\begin{equation} \label{eq:rho_t}
\partial_t ( \tau_{\sigma}^2  \partial_t \rho_{\sigma} ) = \frac{1}{(\sigma +1) \tau_{\sigma}^{d\sigma}} \Delta \rho_{\sigma}^{ \sigma +1} + \frac{2}{\tau_{\sigma}^{d\sigma}} \diver (y \rho_{\sigma}) - \frac{1}{4 \tau_{\sigma}^2} \Delta^2 \rho_{\sigma} + \frac{1}{\tau_{\sigma}^2} \diver \diver \nu_{\sigma}.
\end{equation} 
Note that we have
\[   \partial_t ( \tau_{\sigma}^2  \partial_t \rho_{\sigma} ) = \tau_{\sigma}^2 \partial_t^2 \rho_{\sigma} + 2 \dot{\tau}_{\sigma}\tau_{\sigma} \partial_t \rho_{\sigma}.  \]
Here we recall that, as $t\to \infty$,
\[ \tau_{\sigma}(t) \sim \frac{t}{\sqrt{\sigma}}  \quad \text{and}
  \quad \dot{\tau}_{\sigma}(t) \rightarrow
  \frac{1}{\sqrt{\sigma}},  \] 
so $\tau_{\sigma}^2 \sim (\dot{\tau}_{\sigma}\tau_{\sigma} )^2$ and
both terms are of the same order in time, but not in $\si$. 

\subsection{Change of time variable}
We now make the change of variables $s=s_{\sigma}(t)$ such that
\[ \partial_s = (1-d \sigma)\dot{\tau}_{\sigma}(t)
  \tau_{\sigma}(t)^{d\sigma+1} \partial_t,    \] 
hence
\begin{align*}
s_{\sigma}(t) & =\int \frac{1}{(1-d \sigma)\dot{\tau}_{\sigma}(\gamma) \tau_{\sigma}(\gamma)^{d\sigma+1}} \dd \gamma 
 = 	\int \frac{\ddot{\tau}_{\sigma}(\gamma)}{2(1-d \sigma)\dot{\tau}_{\sigma}(\gamma)} \dd \gamma \\ 
& = \frac{1}{2(1-d \sigma)} \ln \dot{\tau}_{\sigma}(t) + C_0, 
\end{align*}  
so taking $C_0 = 0$
and using that
\[ \dot{\tau}_{\sigma}(t) = \sqrt{\frac{1}{d \sigma} \left( 1
      -\frac{1}{\tau_{\sigma}(t)^{d\sigma}}\right)} , \] 
 we infer
\[ s_{\sigma}(t)= \frac{1}{4(1-d\sigma)} \ln \left(  \frac{1}{d \sigma} \left( 1 -\frac{1}{\tau_{\sigma}(t)^{d\sigma}} \right) \right).  \]
In particular we see that as $t \rightarrow + \infty$, 
\[  s_{\sigma}(t) \longrightarrow -\frac{1}{4(1-d\sigma)} \ln \left(
    d \sigma  \right) =: s_{\sigma}^{\max},\] 
so this new change of variable compactifies time. On
the other hand, defining
$\widetilde{\tau}_{\sigma}(s_{\sigma}(t))=\tau_{\sigma}(t)$, we have  
\begin{equation}
  \label{eq:tau-tilde}
 \widetilde{\tau}_{\sigma}(s_{\sigma})= \left( 1 -d\sigma e^{4 s_{\sigma}(1-d\sigma)}\right)^{-\frac{1}{d\sigma}} ,
\end{equation}
and
\[  \dot{\widetilde{\tau}}_{\sigma}(s_{\sigma})=\frac{\dd}{\dd t}
  \widetilde{\tau}_{\sigma}(s_{\sigma}(t)) =\dot \tau_\si (t)= e^{2
    s_{\sigma}(1-d\sigma)}.  \] 
From this change of variable we infer a new integrable bound from
Lemma~\ref{lem:apriori-unif-si-t} on the quantity 
\[ V_{\sigma}(t) := \frac{1}{\tau_{\sigma}(t)^{2-d\sigma}} \| \nabla v_{\sigma}(t)\|_{L^2}^2,  \]
writing that
\begin{align*}
\int_0^{\infty} \frac{\dot{\tau}_{\sigma}(t)}{\tau_{\sigma}(t)} V_{\sigma}(t) \dd t & = (1-d\sigma) \int_{-\infty}^{s_{\sigma}^{\max}} \dot{\tilde{\tau}}_{\sigma}(s)^2 \tilde{\tau}_{\sigma}(s)^{d\sigma}  \widetilde{V}_{\sigma}(s) \dd s \\
& = (1-d\sigma) \int_{-\infty}^{s_{\sigma}^{\max}}  e^{4s(1-d\sigma)} \left(1-d\si e^{4s(1-d\si)}\right)^{-1} \widetilde{V}_{\sigma}(s) \dd s,
\end{align*}
with the notation $\widetilde{V}_{\sigma}(s_{\sigma})
=V_{\sigma}(t(s_{\sigma})) $, which implies the uniform bound
\begin{equation} \label{eq:integrable_estimate_s}
\int_{-\infty}^{s_{\sigma}^{\max}}  e^{4s(1-d\sigma)} \widetilde{V}_{\sigma}(s) \dd s \le \frac{C}{1-d\sigma}.
\end{equation}
Now making this change of time variable in equation \eqref{eq:rho_t} we then get the following equation on $\widetilde{\rho}_{\sigma}(s_\si(t))=\rho_{\sigma}(t)$:
\begin{equation} \label{eq:rho_s}
\partial_{s_{\sigma}} \widetilde{\rho}_{\sigma}  = \frac{1}{\sigma+1} \Delta \widetilde{\rho}_{\sigma}^{\sigma+1} + 2 \diver( y \widetilde{\rho}_{\sigma}) + \mathcal{R},
\end{equation}
where
\begin{align*}
\mathcal{R} & =   \frac{2}{1-d\sigma} e^{-4s_{\sigma}(1-d \sigma)} \left(1 - d \sigma e^{4 s_{\sigma}(1-d\sigma)} \right) \partial_{s_{\sigma}} \widetilde{\rho}_{\sigma}\\
& \quad -\frac{e^{-4 s_{\sigma}(1-d\sigma)}}{(1-d\sigma)^2} \left(1 - d \sigma e^{4 s_{\sigma}(1-d\sigma)} \right) \partial_{s_{\sigma}}^2 \widetilde{\rho}_{\sigma} \\
& \quad - \frac{1}{4} \left(1 - d \sigma e^{4 s_{\sigma}(1-d\sigma)} \right)^{\frac{2-d\sigma}{d\sigma}}  \Delta^2 \widetilde{\rho}_{\sigma}   \\
& \quad + \left(1 - d \sigma e^{4 s_{\sigma}(1-d\sigma)} \right)^{\frac{2-d\sigma}{d\sigma}} \diver \diver \tilde{\nu}_{\sigma} \\
& =:R_1+R_2+R_3+R_4,
\end{align*}
with the  notation $\tilde{\nu}_{\sigma}(s_\si(t))=\nu_{\sigma}(t)$. Note that at leading order in time, discarding the term $\mathcal{R}$, we formally get the porous medium equation with drift
\[ \partial_t f = \frac{1}{\sigma+1} \Delta f^{\sigma+1} + 2 \diver( y f) , \]
whose solutions are known to converge exponentially in time towards a
Barenblatt profile in Wasserstein distance~\cite{Otto2001}. Of course
as our time range in $s_{\sigma}$ is compact, we do not expect such
property for our system. However it is known that such Barenblatt
functions converge towards a limit Gaussian profile as $\sigma \to 0$,
see for instance \cite[Theorem 2.10]{Chauleur2022}. This motivates our
upcoming convergence result of $\widetilde{\rho}_{\sigma}(s_{\sigma})$
towards a limit Gaussian profile in Wasserstein distance as both $t
\to +\infty$ and $\sigma \to 0$. Our strategy of proof will be to look
at the nonlinear porous medium equation with drift as a perturbation
of a linear harmonic Fokker-Planck operator in the limit $\sigma \to
0$. 
 
\subsection{Harmonic Fokker-Planck operator}
We recall that the Wasserstein distance between two probability measures $\nu_1$ and $\nu_2$ is given by
\[ W_p(\nu_1,\nu_2) = \inf \enstq{\left( \int_{\R^d \times \R^d} |x-y|^p \dd \mu(x,y) \right)^{\frac{1}{p}}}{(\pi_j)_{\#} \mu = \nu_j}  \]
for any $1 \le p < \infty$, where $\mu$ varies among all probability measures on $\R^d \times \R^d$, and where $\pi_j : \R^d \times \R^d \rightarrow \R^d$ denotes the canonical projection onto the $j$-th factor (see for instance \cite{Vi03}). In particular, for all $1 \le p \le q < \infty$, we directly get that
\[ W_p(\nu_1,\nu_2) \le  W_q(\nu_1,\nu_2). \]

We now introduce the well-known Fokker-Planck operator
\[L := \Delta + 2 \diver(y \ \cdot),\] 
and recall some of its useful properties, starting from the convergence of its associated linear flow towards  the universal Gaussian profile 
\[  \Gamma(x) :=\frac{1}{\pi^{d/2}} e^{-|x|^2}.  \]
\begin{lemma}[see e.g. Theorem~11.2.1 in \cite{AmbrosioGigliSavare} or  Theorem~24.7 in \cite{Vi09}]
  \label{lem:fokker_planck_to_gaussian}
Let $\phi $ be a probability density such that 
\[ \int_{\R^d}|x|^2 \phi(x) \dd x <\infty \quad \text{and} \quad \int_{\R^d} \phi(x) \, |\ln \phi(x)| \dd x <\infty.  \]
Then for all $s \ge 0$, we have
\[ W_1(e^{sL}\phi,\Gamma) \le W_2(e^{sL}\phi,\Gamma) \le e^{-2s} W_2(\phi,\Gamma).  \]
\end{lemma}

\begin{lemma} \label{lem:fokker_planck_with_weights}
Let $\phi \in L^1(\R^d)$, then
\[ \| |x|^2 e^{sL} \phi \|_{L^1} \le C  ( \| \phi \|_{L^1} + \| |x|^2 \phi \|_{L^1}) . \]
\end{lemma}
\begin{proof}
The kernel of the harmonic Fokker-Planck operator writes
\[ K(t,x,y) := \pi^{-\frac{d}{2}} (1-e^{-4s})^{-\frac{d}{2}} \Gamma \left( (x-e^{-2s})(1-e^{-4s})^{-\frac12}  \right),   \]
where $\Gamma(z)=\exp ( -|z|^2)$, see for instance \cite[Lemma~4.7]{Ferriere2021}, so that
\[ (e^{sL} \phi)(x) = \int_{\R^d} K(s,x,y) \phi(y) \dd y  \]
for all $x \in \R^d$. We then compute
\begin{align*}
\int_{\R^d} |x|^2 |(e^{sL} \phi) (x)| \dd x  & \lesssim \int_{\R^d} \int_{\R^d} |x-e^{-2s}y|^2 K(s,x,y) |\phi(y)| \dd y \dd x \\
& \quad +e^{-4s} \int_{\R^d} \int_{\R^d}  K(s,x,y) |y|^2 |\phi(y)| \dd y \dd x \\
& \lesssim  (1- e^{-4s}) \int_{\R^d} |\phi(y)| \int_{\R^d} \frac{|z|^2}{\pi^{d/2}} \Gamma(z)  \dd z  \dd y \\
& \quad +e^{-4s} \int_{\R^d} |e^{sL} (|x|^2 \phi(x)) | \dd x ,
\end{align*}
where we have performed the change of variable $z=(x-e^{-2s}y)(1-e^{-4s})^{-\frac12}$ in the first integral on the right hand side. The result follows, as
\[ \| |x|^2 e^{sL} \phi \|_{L^1} \le \| \phi \|_{L^1} \| |x|^2 \Gamma \|_{L^1} + \|  e^{sL} (|x|^2\phi) \|_{L^1}. \qedhere \]
\end{proof}

We also provide the following property, which allows to trade differentiation in space with exponential decay in time.
\begin{lemma}[Lemma~4.8 in \cite{Ferriere2021}] \label{lem:fokker_planck_trade}
Let $\phi \in L^1(\R^d)$, $n\in \mathbb{N}$ and $K \in \left\{1,\ldots,d\right\}^n$. Then $f(s)=e^{sL}\left(\partial_K \phi  \right)$ is a $W^{\infty,1}$ function for all $s > 0$, and for all $m \in \N$ we have
\[ f(s)=e^{-2ns} \partial_K(e^{sL}\phi) \quad \text{and} \quad \| f(s)\|_{\dot{W}^{m-n,1}} \le \frac{e^{-2ns}}{(1-e^{-4s})^{\frac{m}{2}}} \| \Gamma\|_{\dot{W}^{m,1}}  \| \phi \|_{L^1}.\]
\end{lemma}

\subsection{Duhamel formula and perturbation of the porous medium equation} \label{sec:perturbation_porous_medium}
We can express Equation~\eqref{eq:rho_s} as a perturbation of the
Fokker-Planck evolution equation with source term 
\[ \partial_{s_{\sigma}} \widetilde{\rho}_{\sigma} =\frac{1}{1+\sigma}
  L \widetilde{\rho}_{\sigma} + \left(1-\frac{1}{1+\sigma} \right)
  \underbrace{\diver(2y \widetilde{\rho}_{\sigma})}_{=:R_0^{(1)}}
  +\frac{1}{1+\sigma}  \underbrace{\Delta \(
    \widetilde{\rho}_{\sigma}^{1+\sigma} -\widetilde{\rho}_{\sigma}
    \)}_{=:R_0^{(2)}} + \mathcal{R} . \] 
Let $s_{\sigma} \in \left[0,s_{\sigma}^{\max} \right]$. Denoting $R_0:=R_0^{(1)}+\si^{-1} R_0^{(2)}$, this equation writes in Duhamel form as
\begin{equation} \label{eq:rho_s_Duhamel}
\widetilde{\rho}_{\sigma}(s_{\sigma})=e^{s_{\sigma}\frac{L}{1+\sigma}} \widetilde{\rho}_{\sigma}(0) + \frac{\sigma}{1+\sigma} \int_0^{s_{\sigma}} e^{(s_{\sigma}-r)\frac{L}{1+\sigma}} R_0 \dd r + \int_0^{s_{\sigma}} e^{(s_{\sigma}-r)\frac{L}{1+\sigma}} \mathcal{R} \dd r.
\end{equation}  
Note that even if
$s_{\sigma}=0$ does not corresponds to $t=0$ back to the original time
variable, it still leads to an initial time which is bounded in
$\sigma$, hence it is harmless in the following long time analysis of
our equation.

Our strategy of proof is the following: even if $s_{\sigma}$ is
limited to a compact set $\left[0,s_{\sigma}^{\max} \right]$, we
remark that as $\sigma \rightarrow 0$, we have $s_{\sigma}^{\max}
\rightarrow +\infty$ hence $e^{s_{\sigma}L_{\sigma}}
\widetilde{\rho}_{\sigma}(0)$ should be arbitrary near the limiting
Gaussian profile at large $s_{\sigma}$ when $\sigma \rightarrow
0$. Our aim is then to prove that remainder terms are indeed
negligible, thanks to vanishing time coefficients for $\mathcal{R}$,
or in the limit $\sigma \rightarrow 0$ for $R_0$.  

The rest of this section consists essentially in showing the following
result: 
\begin{theorem}\label{theo:cv-log-W1}
  Under the assumptions of Theorem~\ref{theo:log-temps-long}, there
  exists $C>0$ independent of $\si$ such that for all $0\le s_\si\le
  s_\si^{\max}$, 
  \begin{multline*}
    \sup
    \enstq{\int_{\R^d} \varphi
      \(
      \widetilde{\rho}_{\sigma}(s_{\sigma})-e^{s_{\sigma}\frac{L}{1+\sigma}}
      \widetilde{\rho}_{\sigma}(0) \)}{ \varphi \in \cont(\R^d)
      \text{ and }  \|\varphi\|_{\rm 
      Lip}\le 1} \\
      \le C(\sigma + e^{-2 s_{\sigma}(1-d\sigma)}).
  \end{multline*}
\end{theorem}
Even though the supremum is reminiscent of the definition of the distance $W_1$,
we do not use this notation in the above statement, because the
densities $\tilde \rho_\si$ (or, equivalently, $\rho_\si$) are not
normalized to be probability densities, contrary to $\varrho_\si$.

\subsection{Duality and regularization} \label{sec:duality_and_reg}
In order to prove Theorem \ref{theo:cv-log-W1}, we will rely on the dual
representation of the Wasserstein distance $W_1$ recalled in the
introduction,  \eqref{eq:W1}.
%\[ W_1(\mu_1,\mu_2)= \sup \enstq{\int_{\R^d} \varphi \dd (\mu_1 - \mu_2)}{ \varphi \in \cont(\R^d) \ \text{and}\ 1-\Lip}  \]
%for any suitable probability measures $\mu_1$ and $\mu_2$.
Now  the test functions $\varphi$ must be regularized in
order to absorb space derivatives by integration by parts, and also
need to be truncated as they may grow  linearly at infinity. The
aim of this paragraph is to provide precise bounds on such
regularization and cut-off processes. 

Let $\varphi$ be any $1$-Lipschitz function. Let $\zeta \in \mathcal{S}(\R^d)$ such that $\zeta \ge 0$ and $\int_{\R^d}\zeta =1$, and consider the mollifier $\zeta_{\epsilon}(x)=\epsilon^{-d} \zeta(x/\epsilon)$. We can then define the regularized function 
\[ \widetilde{\varphi}_{\epsilon} := \varphi \ast \zeta_{\epsilon},\]
 which is still $1$-Lipschitz and satisfies
\[  \| \widetilde{\varphi}_{\epsilon} - \varphi \|_{L^{\infty}} \le \epsilon \| |\cdot|\zeta \|_{L^1}. \]
We now introduce the cut-off function $\chi \in \cont_c^\infty(\R^d)$ such that $\chi \equiv 1$ on $\mathcal{B}(0,1)$ and $0 \le \chi \le 1$. We then define $\chi_{\epsilon}(x)=\chi(\epsilon x)$ and 
\[ \varphi_{\epsilon} := (\widetilde{\varphi}_{\epsilon} - \widetilde{\varphi}_{\epsilon}(0)) \chi_{\epsilon} \in \cont_c(\R^d). \]
We similarly get that $\varphi_{\epsilon}$ is Lipschitz continuous,
uniformly in $\epsilon$,  and that
\[ \left\|  \frac{\widetilde{\varphi}_{\epsilon} - \widetilde{\varphi}_{\epsilon}(0) - \varphi_{\epsilon} }{|\cdot|^2} \right\|_{L^{\infty}} \le \epsilon.  \]
We now state the main property which allows to trade space derivatives into negative powers of $\epsilon$.
\begin{lemma}[Equation~4.21 in \cite{Ferriere2021}] \label{lem:regularization_trade}
Let $n\in \N$. There exists $C_n>0$ such that for all $\epsilon>0$,
\[ \| \varphi_{\epsilon} \|_{\dot{W}^{1+n,\infty}} \le C_n \epsilon^{-n}.  \]
\end{lemma}

\subsection{Proof of Theorem~\ref{theo:cv-log-W1}}
%We now have all the tools at hand to prove our convergence
%result. 

 Let $\varphi$ be a  $1$-Lipschitz function. Following the previous discusssion, we
first introduce the decomposition
\begin{equation} \label{eq:rho_reg_cut_off}
\begin{aligned}
\int_{\R^d} \varphi (\tilde{\rho}_{\sigma}(s_{\sigma})-e^{s_{\sigma}\frac{L}{1+\sigma}} \widetilde{\rho}_{\sigma}(0)) = & \int_{\R^d} (\varphi-\widetilde{\varphi}_{\epsilon}) (\tilde{\rho}_{\sigma}(s_{\sigma})-e^{s_{\sigma}\frac{L}{1+\sigma}} \widetilde{\rho}_{\sigma}(0)) \\
& + \int_{\R^d} (\widetilde{\varphi}_{\epsilon} -  \varphi_{\epsilon})  (\tilde{\rho}_{\sigma}(s_{\sigma})-e^{s_{\sigma}\frac{L}{1+\sigma}} \widetilde{\rho}_{\sigma}(0)) \\
& + \int_{\R^d} \varphi_{\epsilon} (\tilde{\rho}_{\sigma}(s_{\sigma})-e^{s_{\sigma}\frac{L}{1+\sigma}} \widetilde{\rho}_{\sigma}(0)).
\end{aligned}
\end{equation}   
To bound the first term, we make use of the previous estimates from Section \ref{sec:duality_and_reg}, so that
\begin{multline*}   \left| \int_{\R^d} (\varphi-\widetilde{\varphi}_{\epsilon}) (\tilde{\rho}_{\sigma}(s_{\sigma})-e^{s_{\sigma}\frac{L}{1+\sigma}} \widetilde{\rho}_{\sigma}(0)) \right| \\
\le \| \widetilde{\varphi}_{\epsilon} - \varphi \|_{L^{\infty}} \left(  \| \tilde{\rho}_{\sigma} \|_{L^1} + \| e^{s_{\sigma}\frac{L}{1+\sigma}} \widetilde{\rho}_{\sigma}(0) \|_{L^1} \right) \lesssim \epsilon. 
\end{multline*}  
Similarly we estimate the second term, recalling that
$\int_{\R^d}\tilde{\rho}_{\sigma}(s_{\sigma}) = \int_{\R^d}
e^{s_{\sigma}\frac{L}{1+\sigma}} \widetilde{\rho}_{\sigma}(0)$ by
the conservation of mass, and writing
\begin{multline*}
 \left|  \int_{\R^d} ( \widetilde{\varphi}_{\epsilon}-\varphi_{\epsilon}) (\tilde{\rho}_{\sigma}(s_{\sigma})-e^{s_{\sigma}\frac{L}{1+\sigma}} \widetilde{\rho}_{\sigma}(0)) \right| \\ 
 \le \left|  \int_{\R^d} ( \widetilde{\varphi}_{\epsilon}- \widetilde{\varphi}_{\epsilon}(0)-\varphi_{\epsilon}) (\tilde{\rho}_{\sigma}(s_{\sigma})-e^{s_{\sigma}\frac{L}{1+\sigma}} \widetilde{\rho}_{\sigma}(0)) \right| \\
  + \left| \widetilde{\varphi}_{\epsilon}(0) \int_{\R^d} \left(\tilde{\rho}_{\sigma}(s_{\sigma}) - e^{s_{\sigma}\frac{L}{1+\sigma}} \widetilde{\rho}_{\sigma}(0)  \right)  \right|    \\
 \le \left\| \frac{\widetilde{\varphi}_{\epsilon}-\widetilde{\varphi}_{\epsilon}(0)-\varphi_{\epsilon}}{1+|y|^2} \right\|_{L^{\infty}} \left( \| |y|^2 \tilde{\rho}_{\sigma} \|_{L^1} + \| |y|^2 e^{s_{\sigma}\frac{L}{1+\sigma}} \widetilde{\rho}_{\sigma}(0) \|_{L^1}  \right) \lesssim \epsilon.
 \end{multline*}
In the previous estimate, we used that, from
Lemma~\ref{lem:apriori-unif-si-t}, $\| |y|^2\tilde{\rho}_{\sigma}
\|_{L^1}$ is 
uniformly bounded with respect to $\sigma$ and $s_{\sigma}$, as well
as Lemma~\ref{lem:fokker_planck_with_weights} in order to estimate $\| |y|^2
e^{s_{\sigma}\frac{L}{1+\sigma}} \widetilde{\rho}_{\sigma}(0)
\|_{L^1}$. We now 
turn our attention to the second term of 
\eqref{eq:rho_reg_cut_off}, using the Duhamel formulation from
equation \eqref{eq:rho_s_Duhamel} to express
$\tilde{\rho}_{\sigma}(s_{\sigma})$, and we control each error term.

 The first error term corresponds to
 \begin{align*}
  E_0  & = E_0^{(1)}+ E_0^{(2)} \\
  &=\frac{\sigma}{1+\sigma}\int_{\R^d} \varphi_{\epsilon}  \int_0^{s_{\sigma}} e^{\frac{s_{\sigma}-r}{1+\sigma} L}R_0^{(1)} \dd r  + \frac{1}{1+\sigma}\int_{\R^d} \varphi_{\epsilon}  \int_0^{s_{\sigma}} e^{\frac{s_{\sigma}-r}{1+\sigma} L} R_0^{(2)} \dd r . 
  \end{align*}
  We easily estimate 
  \begin{align*}
  E_0^{(1)} & = \frac{\sigma}{1+\sigma} \int_{\R^d} \varphi_{\epsilon}  \int_0^{s_{\sigma}} e^{\frac{s_{\sigma}-r}{1+\sigma} L} \diver(2y \widetilde{\rho}_{\sigma} ) \dd r  \\
  & = \frac{2 \sigma}{1+\sigma} \int_0^{s_{\sigma}}
    e^{-2\frac{s_{\sigma}-r}{1+\sigma}} \int_{\R^d} 
   \nabla\varphi_{\epsilon} \cdot e^{\frac{s_{\sigma}-r}{1+\sigma} L} (y \widetilde{\rho}_{\sigma})  \dd r, 
  \end{align*}
  where we have used Lemma~\ref{lem:fokker_planck_trade} and an
  integration by parts. From Lemma~\ref{lem:apriori-unif-si-t}, we
  then infer the uniform bound 
  \[ \left\lvert  E_0^{(1)} \right\rvert \lesssim \sigma \| \nabla
    \varphi_{\epsilon} \|_{L^{\infty}} \int_0^{s_{\sigma}}
    e^{-2\frac{s_{\sigma}-r}{1+\sigma}} \| y
    \widetilde{\rho}_{\sigma}(r) \|_{L^1} \dd r \lesssim  \sigma.\]
 We similarly write, using Lemma~\ref{lem:fokker_planck_trade} again,
 \begin{align*}
E_0^{(2)} & = \frac{1}{1+\sigma}\int_{\R^d} \varphi_{\epsilon}  \int_0^{s_{\sigma}} e^{\frac{s_{\sigma}-r}{1+\sigma} L} R_0^{(2)} \dd r  \\
 & = \frac{1}{1+\sigma} \int_{\R^d} \varphi_{\epsilon}
   \int_0^{s_{\sigma}} e^{\frac{s_{\sigma}-r}{1+\sigma} L}\Delta \(
 \widetilde{\rho}_{\sigma}^{1+\sigma} -\widetilde{\rho}_{\sigma}  \)\dd r
   \\ 
& = \frac{1}{1+\sigma} \int_{\R^d} \varphi_{\epsilon}
  \int_0^{s_{\sigma}} e^{-\frac{4}{1+\sigma}(s_{\sigma}-r)} \Delta
  e^{\frac{s_{\sigma}-r}{1+\sigma} L} \(
  \widetilde{\rho}_{\sigma}^{1+\sigma} -\widetilde{\rho}_{\sigma}  \)\dd r  \\
 &  = \frac{1}{1+\sigma} \int_0^{s_{\sigma}}
   e^{-\frac{4}{1+\sigma}(s_{\sigma}-r)} \int_{\R^d} \nabla
   \varphi_{\epsilon} \cdot \nabla e^{\frac{s_{\sigma}-r}{1+\sigma} L} \(
   \widetilde{\rho}_{\sigma}^{1+\sigma} -\widetilde{\rho}_{\sigma}   \)\dd r.
\end{align*}   
In view of Lemma~\ref{lem:apriori-unif-si-t}, there exists $C>0$
independent of $\si\in (0,\epsilon)$ and $t\ge 0$ such that
\[ \sup_{0\le s_\si\le s_{\sigma}^{\max}  } \left\|
    \widetilde{\rho}_{\sigma}(s_\si)^{1+\sigma}
    -\widetilde{\rho}_{\sigma}(s_\si)
  \right\|_{L^1} \le C \si . \]
  From H\"older inequality, Lemma~\ref{lem:fokker_planck_trade} with
 $m=1$, we infer 
\[ \left\lvert E_0^{(2)}\right\rvert \le \frac{1}{1+\sigma}
  \int_0^{s_{\sigma}} \frac{e^{-\frac{4}{1+\sigma}(s_{\sigma}-r)}}{\(1
    - e^{-\frac{4}{1+\sigma}(s_{\sigma}-r)}\)^{1/2}}  \| \nabla
  \varphi_{\epsilon} \|_{L^{\infty}} \|\widetilde{\rho}_{\sigma}^{1+\sigma}
    -\widetilde{\rho}_{\sigma}
  \|_{L^1}  \dd r \le C
  \sigma, \]  
thanks to Lemma~\ref{lem:regularization_trade}. 
\bigbreak

The second error term, associated to $R_1$ defined in equation
\eqref{eq:rho_s}, writes 
\begin{align*}
 E_1 & = \frac{2}{1-d\sigma} \int_{\R^d} \varphi_{\epsilon}  \int_0^{s_{\sigma}}   e^{\frac{s_{\sigma}-r}{1+\sigma} L} e^{-4r(1-d\sigma)} \left( 1-d\sigma e^{4r(1-d\sigma)}\right) \partial_{s_{\sigma}} \tilde{\rho}_{\sigma}(r) \dd r  \\
 & = \frac{2}{1-d\sigma} \int_0^{s_{\sigma}}  e^{-4r(1-d\sigma)} \left( 1-d\sigma e^{4r(1-d\sigma)}\right) \int_{\R^d} \varphi_{\epsilon} e^{\frac{s_{\sigma}-r}{1+\sigma} L} \left( \partial_{s_{\sigma}}\tilde{\rho}_{\sigma}(r) \right) \dd r.
 \end{align*}
 To get further estimates on $\partial_{s_{\sigma}} \tilde{\rho}_{\sigma}$, we recall that in the time variable $t$ we have
\[  \partial_t \rho_{\sigma}(t)+\frac{1}{\tau_{\sigma}(t)^2} \diver
  J_{\sigma}(t)= \partial_t
  \rho_{\sigma}(t)+\frac{1}{\tau_{\sigma}(t)^{1+\frac{d\sigma}{2}}}
  \diver \mu_{\sigma}(t)= 0, \]
with
\[ \mu_{\sigma}(t):= \frac{1}{\tau_{\sigma}(t)^{1-\frac{d\sigma}{2}}} \IM(\overline{v_{\sigma}}(t) \nabla v_{\sigma}(t) ),  \]
which satisfies a uniform $L^1$-estimate from Cauchy-Schwarz
inequality, as 
\begin{equation} \label{eq:bound_mu_sigma}
\| \mu_{\sigma}(t) \|_{L^1}\lesssim
\frac{1}{\tau_{\sigma}(t)^{1-\frac{d\sigma}{2}}} \| \nabla
v_{\sigma}(t)\|_{L^2} \| v_{\sigma} (t) \|_{L^2} = \| \phi
\|_{L^2} \sqrt{V_{\sigma}(t)}. 
\end{equation}
 Moreover, we recall that $V_\sigma$ satisfies a uniform (both in
 $\sigma$ and $t$) bound and a uniform integrability property, from
 Lemma~\ref{lem:apriori-unif-si-t}. 
 Hence we end up, in the $s_{\sigma}$ time variable, with
\begin{align*}
\partial_{s_{\sigma}}\tilde{\rho}_{\sigma} & =(1-d \sigma)\dot{\tau}_{\sigma}(t(s_{\sigma})) \tau_{\sigma}(t(s_{\sigma}))^{1+d\sigma} \partial_t\rho_{\sigma}(t(s_{\sigma})) \\
& =-(1-d \sigma) \dot{\tau}_{\sigma}(t(s_{\sigma})) \tau_{\sigma}(t(s_{\sigma}))^{\frac{d\sigma}{2}} \diver \mu_{\sigma}(t(s_{\sigma})) \\
& = -(1-d \sigma) e^{2s_{\sigma}(1-d\sigma)} \left(1-d\sigma e^{4s_{\sigma}(1-d \sigma)}  \right)^{-1/2} \diver \mu_{\sigma}(t(s_{\sigma})),
\end{align*}
so using properties of the Fokker-Planck flow from
Lemma~\ref{lem:fokker_planck_trade}, and integrating by parts, we get 
\begin{align*}
  E_1  & = -2\int_0^{s_{\sigma}}  e^{-2r(1-d\sigma)} \left(
         1-d\sigma e^{4r(1-d\sigma)}\right)^{1/2}
         \int_{\R^d} \varphi_{\epsilon} e^{\frac{s_{\sigma}-r}{1+\sigma}
         L} \diver \mu_{\sigma}(r) \, \dd r \\ 
  & =  2\int_0^{s_{\sigma}}  e^{-2r(1-d\sigma)}
    e^{-\frac{2}{1+\sigma}(s_{\sigma}-r)} \left( 1-d\sigma
    e^{4r(1-d\sigma)}\right)^{1/2} \int_{\R^d} \nabla
    \varphi_{\epsilon} \cdot\(e^{\frac{s_{\sigma}-r}{1+\sigma} L}
    \mu_{\sigma}(r)\) \dd r.   
 \end{align*}
 By taking the absolute value and using \eqref{eq:bound_mu_sigma} and
 Lemma~\ref{lem:regularization_trade}, we infer 
 \begin{align*} 
 |E_1| & \le 2 \| \nabla \varphi_{\epsilon} \|_{L^{\infty}} \|
         \phi \|_{L^2}  \int_0^{s_{\sigma}}
         e^{-2r(1-d\sigma)} e^{-\frac{2}{1+\sigma}(s_{\sigma}-r)}
         \left(1-\frac{d\sigma}{4} e^{4s(1-d\sigma)}
         \right)^{1/2} \sqrt{\widetilde{V}_{\sigma} (r)} \,
         \dd r   \\ 
 & \lesssim \left( \int_0^{s_{\sigma}} e^{-8r(1-d\sigma)}
   e^{-\frac{4}{1+\sigma}(s_{\sigma}-r)} \dd r  \right)^{1/2}
   \left( \int_0^{s_{\sigma}} e^{4r(1-d\sigma)} \widetilde{V}_{\sigma}
   (r) \dd r  \right)^{1/2} \\ 
 & \lesssim \frac{\mathcal{E} (0)}{1 - d \sigma} \left(
   \int_0^{s_{\sigma}} e^{-8r(1-d\sigma)}
   e^{-\frac{4}{1+\sigma}(s_{\sigma}-r)} \dd r  \right)^{1/2} ,
 \end{align*}
 from Cauchy-Schwarz inequality and using the uniform bound from
 equation \eqref{eq:integrable_estimate_s}. Since we assume
 $\sigma<1/(2d)$,  $2-2d\si-\frac{1}{1+\si}\ge  \frac{1}{2d+1}>0$, we
 then get the decreasing in time estimate 
 \begin{equation} \label{eq:E_1_estimate}
 	|E_1|  \lesssim \left( \int_0^{s_{\sigma}} e^{-4r
            \left(2-2d\sigma-\frac{1}{1+\sigma} \right)}
          e^{-\frac{4}{1+\sigma}s_{\sigma}} \dd r  \right)^{1/2}
        \lesssim  e^{-\frac{2}{1+\sigma}s_{\sigma}}.  
 \end{equation}
We now turn to the fourth error term. We directly write  
 \begin{align*}
 E_3 & = -\frac{1}{4} \int_{\R^d} \varphi_{\epsilon}
       \int_0^{s_{\sigma}} e^{-\frac{s_{\sigma}-r}{1+\sigma}
       L} \left(1 - d\sigma e^{4r(1-d\sigma)}
       \right)^{\frac{2-d\sigma}{d\sigma}} \Delta^2
       \tilde{\rho}_{\sigma}(r) \dd r \\ 
 & = -\frac{1}{4} \int_0^{s_{\sigma}}
   e^{-\frac{8}{1+\sigma}(s_{\sigma}-r)} \left(1 - d\sigma
   e^{4r(1-d\sigma)} \right)^{\frac{2-d\sigma}{d\sigma}} \int_{\R^d}
   \Delta^2 \varphi_{\epsilon} e^{-\frac{s_{\sigma}-r}{1+\sigma} L}
   \tilde{\rho}_{\sigma}(r) \dd r ,
 \end{align*}
 so by H\"older inequality, and from
 Lemma~\ref{lem:regularization_trade},  we infer
 \begin{align*}
 |E_3| & \le \frac{1}{4} \| \Delta^2 \varphi_{\epsilon} \|_{L^{\infty}} \| \tilde{\rho}_{\sigma}(0) \|_{L^1} \int_0^{s_{\sigma}}  e^{-\frac{8}{1+\sigma}(s_{\sigma}-r)} \exp \left(\frac{2-d\sigma}{d\sigma} \ln \left(1 - d\sigma e^{4r(1-d\sigma)} \right)\right) \dd r  \\
 & \lesssim \frac{1}{\epsilon^3} e^{-\frac{8}{1+\sigma}s_{\sigma}} \int_0^{\infty} e^{\frac{8r}{1+\sigma}} e^{-\frac{3}{2}e^{2r}} \dd r \\
 & \lesssim \frac{1}{\epsilon^3} e^{-\frac{8}{1+\sigma}s_{\sigma}}.
 \end{align*}
We similarly estimate the fifth error term. We first write 
 \begin{align*}
 E_4 & = \int_{\R^d} \varphi_{\epsilon}  \int_0^{s_{\sigma}} \left(1 - d\sigma e^{4r(1-d\sigma)} \right)^{\frac{2-d\sigma}{d\sigma}}  e^{-\frac{s_{\sigma}-r}{1+\sigma} L}  \diver \diver \tilde{\nu}_{\sigma}(r) \dd r \\
 & = \int_0^{s_{\sigma}} \left(1 - d\sigma e^{4r(1-d\sigma)} \right)^{\frac{2-d\sigma}{d\sigma}}  e^{-\frac{4}{1+\sigma}(s_{\sigma}-r)} \int_{\R^d} \nabla^2 \varphi_{\epsilon} \cdot e^{-\frac{s_{\sigma}-r}{1+\sigma} L} \tilde{\nu}_{\sigma}(r) \dd r,
 \end{align*}
 where ``$\cdot$'' denotes here the Frobenius scalar product between
 matrices. 
 In view of the inequality  $\| \tilde{\nu}_{\sigma} \|_{L^1} \le \| \nabla v_{\sigma} \|_{L^2}^2$,  applying \eqref{eq:integrable_estimate_s}, we have
 \begin{align*}
 |E_4| & \le \|\nabla^2 \varphi_{\epsilon} \|_{L^{\infty}}  \int_0^{s_{\sigma}}  e^{-\frac{4}{1+\sigma}(s_{\sigma}-r)} \widetilde{V}_{\sigma}(r) \dd r  \\
 & \lesssim \frac{1}{\epsilon} e^{-4s_{\sigma}(1-d\si)} \int_0^{s_{\sigma}} \underbrace{e^{-4\left(\frac{1}{1+\sigma}-1+d\sigma  \right)(s_{\sigma}-r)}}_{\lesssim 1} e^{4r(1-d\si)} \widetilde{V}_{\sigma}(r) \dd r \\
 & \lesssim \frac{1}{\epsilon} e^{-4s_{\sigma}(1-d\si)} \frac{C}{1-d\si}.
 \end{align*}
 We finally estimate the third error term
 \[  E_2= \int_{\R^d} \varphi_{\epsilon}  \int_0^{s_{\sigma}}  e^{-\frac{s_{\sigma}-r}{1+\sigma} L} R_2 (r) \dd r ,\]
 which is the most delicate to treat due to the presence of the double
 time derivative: 
\begin{equation*}
R_2 (s_{\sigma}) = - \frac{e^{-4 s_{\sigma}(1-d\sigma)}}{(1-d\sigma)^2} \left(1 - d \sigma e^{4 s_{\sigma}(1-d\sigma)} \right) \partial_{s_{\sigma}}^2 \widetilde{\rho}_{\sigma}.
\end{equation*} 
 Writing 
 \begin{align*}
 R_2 (s_{\sigma}) & =- \partial_{s} \left( \frac{e^{-4 s_{\sigma} (1-d\sigma)}}{(1-d\sigma)^2} \left(1 - d \sigma e^{4 s_{\sigma} (1-d\sigma)} \right) \partial_{s} \widetilde{\rho}_{\sigma} \right) -\frac{4}{1-d\si} e^{-4s_{\sigma}(1-d\si)} \partial_{s}  \widetilde{\rho}_{\sigma}(s_{\sigma}) \\
 & =: R_2^{(1)}+ R_2^{(2)},
 \end{align*} 
 we treat each term separately. We begin with 
 \[ E_2^{(1)}   =  \frac{1}{(1-d\si)^2} \int_{\R^d} \varphi_{\epsilon}  \int_0^{s_{\sigma}} e^{-\frac{s_{\sigma}-r}{1+\sigma} L} \partial_s \left(e^{-4 r(1-d\sigma)} \left(1 - d \sigma e^{4 r(1-d\sigma)} \right) \partial_{s}  \widetilde{\rho}_{\sigma}(r)  \right) \dd r.\]
We will use a result which is a direct consequence of \cite[Corollary 4.10]{Ferriere2021}.
\begin{lemma}
Let $T>0$ and $f \in L^{\infty}((0,T);W^{1,1}(\R^d)) \cap
\cont((0,T);L^1(\R^d))$. For all $s \in (0,T)$, we have
\[ \int_0^s e^{(s-r)L} \partial_s f(r) \dd r =  \int_0^s L e^{(s-r)L} f(r) \dd r + f(s) - e^{sL}f(0). \]
\end{lemma}
Such a result allows to rewrite $E_2^{(1)}$ as a sum of four error
terms, as $L$ is the sum of two operators,
\begin{align*}
E_2^{(1)} & =  -\frac{1}{(1-d\si)^2} \int_0^{s_{\sigma}} e^{-4 r(1-d\sigma)} \left(1 - d \sigma e^{4 r(1-d\sigma)} \right) \int_{\R^d} \varphi_{\epsilon}   e^{-\frac{s_{\sigma}-r}{1+\sigma} L}  \Delta \partial_{s}  \widetilde{\rho}_{\sigma}(r)   \dd r \\
& \quad - \frac{1}{(1-d\si)^2} \int_0^{s_{\sigma}} e^{-4 r(1-d\sigma)} \left(1 - d \sigma e^{4 r(1-d\sigma)} \right) \int_{\R^d} \varphi_{\epsilon}   e^{-\frac{s_{\sigma}-r}{1+\sigma} L} \diver(2y \partial_{s}  \widetilde{\rho}_{\sigma}(r) )  \dd r \\
& \quad - \frac{1}{(1-d\si)^2} e^{-4 s_{\sigma}(1-d\sigma)} \left(1 - d \sigma e^{4 s_{\sigma}(1-d\sigma)} \right) \int_{\R^d} \varphi_{\epsilon} \partial_{s}  \widetilde{\rho}_{\sigma}(s_{\sigma}) \\
& \quad - \frac{1}{(1-d\si)^2} \left(1 - d \sigma \right) \int_{\R^d} \varphi_{\epsilon} e^{\frac{s_{\sigma}}{1+\sigma} L} \partial_{s} \widetilde{\rho}_{\sigma}(0) \\
& =:  E_2^{(1,1)}+ E_2^{(1,2)}+ E_2^{(1,3)}+ E_2^{(1,4)}.
\end{align*}
To estimate $E_2^{(1,1)}$ we recall that
\begin{equation}
  \label{eq:d_s-tilde-rho}
  \partial_s   \widetilde{\rho}_{\sigma}(r)=-(1-d\sigma)
  e^{2r(1-d\si)}  \left(1 - d \sigma e^{4 r(1-d\sigma)}
  \right)^{-1/2} \diver \mu_{\si},  
\end{equation}
which implies
\begin{equation*}
  E_2^{(1,1)} =  \frac{e^{-6\frac{s_{\si}}{1+\si}}}{1-d\si} \int_0^{s_{\sigma}}  e^{r\left[\frac{6}{1+\si}-2
  (1-d\sigma)\right]} \sqrt{1 - d \sigma e^{4 r(1-d\sigma)}}
     \int_{\R^d} \nabla
  \varphi_{\epsilon}   \cdot \Delta e^{-\frac{s_{\sigma}-r}{1+\sigma} L}
  \mu_{\sigma} (r)   \dd r  ,
\end{equation*}
from Lemma~\ref{lem:fokker_planck_trade}. 
Since the square root term in $E_2^{(1,1)}$ is controlled by one,
\begin{equation} \label{eq:bound_E_2^11}
  \left\lvert E_2^{(1,1)} \right\rvert
\le \frac{e^{-6\frac{s_{\si}}{1+\si}}}{(1-d\si)^2} 
  \int_0^{s_{\sigma}} e^{r \left[ \frac{6}{1+\si} -2
    (1-d\sigma) \right]} \left| \int_{\R^d} \nabla
\varphi_{\epsilon}   \cdot \Delta
 \(e^{-\frac{s_{\sigma}-r}{1+\sigma} L} \mu_{\sigma} (r)\) \right|  \dd r. 
\end{equation}
We then remark that, for any $r \in (0, s_{\si})$, Lemma~\ref{lem:fokker_planck_trade} with $m=1$ leads to
\begin{equation} \label{eq:first_bound_E_2^11}
\begin{aligned}
	\left| \int_{\R^d} \nabla \varphi_{\epsilon}   \cdot \Delta
  \(e^{-\frac{s_{\sigma}-r}{1+\sigma} L} \mu_{\sigma} (r)\) \right|
  &\le \| \nabla \varphi_{\epsilon} \|_{L^\infty}
    \left\| \Delta \(e^{-\frac{s_{\sigma}-r}{1+\sigma} L} \mu_{\sigma}
    (r)\)\right\|_{L^1} \\ 
&\lesssim \frac{1}{1 - e^{- 4 \frac{s_{\sigma}-r}{1+\sigma}}} \|
  \mu_{\sigma} (r) \|_{L^1} \\ 
		&\lesssim \frac{1}{1 - e^{- 4
                  \frac{s_{\sigma}-r}{1+\sigma}}} \| \phi\|_{L^2}
                  \sqrt{\widetilde V_\sigma (r)},
\end{aligned}
\end{equation}
thanks to \eqref{eq:bound_mu_sigma}. This last estimate suggests adding and subtracting the factor~$e^{- 4 \frac{s_{\sigma}-r}{1+\sigma}}$ inside the estimating term for $E_2^{(1,1)}$. In fact, this leads to
\begin{multline} \label{eq:second_bound_E_2^11}
	\int_0^{s_{\sigma}} e^{r \left[ \frac{6}{1+\si} -2 (1-d\sigma)
          \right]} \left(1 - e^{- 4 \frac{s_{\sigma}-r}{1+\sigma}}
        \right) \left| \int_{\R^d} \nabla \varphi_{\epsilon}   \cdot
          \Delta \(e^{-\frac{s_{\sigma}-r}{1+\sigma} L} \mu_{\sigma} (r)\)
        \right|  \dd r \\  
\begin{aligned}
  &\lesssim \int_0^{s_{\sigma}} e^{r \left[ \frac{6}{1+\si} -2 (1-d\sigma) \right]} \sqrt{\widetilde V_\sigma (r)} \dd r \\
&\lesssim \left( \int_0^{s_{\sigma}} e^{2 r \left[ \frac{6}{1+\si} -4
  (1-d\sigma) \right]} \dd r \right)^{1/2} \left( \int_0^{s_{\sigma}}
  e^{4r (1-d\sigma)} \widetilde V_\sigma (r) \dd r \right)^{1/2} \\ 
&\lesssim e^{ 6 \frac{s_{\sigma}}{1+\sigma}} \frac{e^{ -4 s_{\sigma}  (1-d\sigma) }}{2 \sqrt{\frac{3}{1+\si} -2 (1-d\sigma)}} \sqrt{\frac{\mathcal{E} (0)}{1 - d \sigma}}.
		\end{aligned}
\end{multline}
On the other hand, take some $S \in (0, s_{\sigma})$ to be optimized later. Then we also have
\begin{multline} \label{eq:third_bound_E_2^11}
	\int_0^{S} e^{r \left[ \frac{6}{1+\si} -2 (1-d\sigma) \right]}
        e^{- 4 \frac{s_{\sigma}-r}{1+\sigma}} \left| \int_{\R^d}
          \nabla \varphi_{\epsilon}   \cdot \Delta
          \(e^{-\frac{s_{\sigma}-r}{1+\sigma} L} \mu_{\sigma} (r)\)
        \right|  \dd r \\  
	\begin{aligned}
&\lesssim \int_0^{S} e^{r \left[ \frac{6}{1+\si} -2 (1-d\sigma)
  \right]} \frac{e^{- 4 \frac{s_{\sigma}-r}{1+\sigma}}}{1 - e^{- 4
  \frac{s_{\sigma}-r}{1+\sigma}}} \sqrt{V_\sigma (r)} \dd r \\ 
			&\lesssim e^{- 2 \frac{s_{\sigma}}{1+\sigma}} e^{4S \left(\frac{2}{1+\sigma} - 1 + d\sigma \right)} \left( \int_0^{S}  \frac{e^{- 4 \frac{s_{\sigma}-r}{1+\sigma}}}{\left(1 - e^{- 4 \frac{s_{\sigma}-r}{1+\sigma}}\right)^2} \dd r \right)^{1/2} \left( \int_0^{s_{\sigma}} e^{4r (1-d\sigma)} \widetilde V_\sigma (r) \dd r \right)^{1/2} \\
			&\lesssim e^{- 2 \frac{s_{\sigma}}{1+\sigma}}  e^{4s_{\sigma} \left(\frac{2}{1+\sigma} - 1 + d\sigma \right)} \left( \int_0^{S} \frac{e^{- 4 \frac{s_{\sigma}-r}{1+\sigma}}}{\left(1 - e^{- 4 \frac{s_{\sigma}-r}{1+\sigma}}\right)^2} \dd r \right)^{1/2} \sqrt{\frac{\mathcal{E} (0)}{1 - d \sigma}} \\
			&\lesssim e^{- 2 \frac{s_{\sigma}}{1+\sigma}} e^{4s_{\sigma} \left(\frac{2}{1+\sigma} - 1 + d\sigma \right)} (1 + \sigma) \left( \frac{1}{1 - e^{- 4 \frac{s_{\sigma}-S}{1+\sigma}}} - \frac{1}{1 - e^{- 4 \frac{s_{\sigma}}{1+\sigma}}} \right)^{1/2} \\
			&\lesssim e^{- 2 \frac{s_{\sigma}}{1+\sigma}} e^{4s_{\sigma} \left(\frac{2}{1+\sigma} - 1 + d\sigma \right)} \frac{1}{\left( 1 - e^{- 4 \frac{s_{\sigma}-S}{1+\sigma}} \right)^{1/2}} \\
			&\lesssim e^{4s_{\sigma} \left(\frac{2}{1+\sigma} - 1 + d\sigma \right)} \frac{1}{\left( e^{\frac{4 s_{\sigma}}{1+\sigma}} - e^{\frac{4 S}{1+\sigma}} \right)^{1/2}}.
		\end{aligned}
\end{multline}
In order to handle the remaining integral on $\left[S,s_{\sigma}
\right]$, similarly to estimate \eqref{eq:first_bound_E_2^11} we use
Lemma~\ref{lem:fokker_planck_trade} this time with $m=2$, which
implies that, for all $r \in (S, s_{\sigma})$, 
\begin{align*}
	\left| \int_{\R^d} \nabla \varphi_{\epsilon}   \cdot \Delta e^{-\frac{s_{\sigma}-r}{1+\sigma} L} \mu_{\sigma} (r) \right| &\le \| \Delta \varphi_{\epsilon} \|_{L^\infty} \| \nabla e^{-\frac{s_{\sigma}-r}{1+\sigma} L} \mu_{\sigma} (r) \|_{L^1} \\
%		&\lesssim \frac{1}{\epsilon} \frac{1}{\left(1 - e^{- 4 \frac{s_{\sigma}-r}{1+\sigma}}\right)^{1/2}} \| \mu_{\sigma} (r) \|_{L^1} \\
		&\lesssim \frac{1}{\epsilon} \frac{1}{\left(1 - e^{- 4 \frac{s_{\sigma}-r}{1+\sigma}}\right)^{1/2}} \| \phi \|_{L^2} \sqrt{ \widetilde V_\sigma (r)}.
\end{align*}
Therefore, from the uniform bound on $V_\sigma$ and by direct integration, we also have
\begin{equation} \label{eq:fourth_bound_E_2^11}
\begin{aligned}
	\int_S^{s_{\sigma}} e^{r \left[ \frac{6}{1+\si} -2 (1-d\sigma) \right]} e^{- 4 \frac{s_{\sigma}-r}{1+\sigma}} \left| \int_{\R^d} \Delta \varphi_{\epsilon}   \cdot \nabla e^{-\frac{s_{\sigma}-r}{1+\sigma} L} \mu_{\sigma} (r) \right|  \dd r \\ 
		\begin{aligned}
			&\lesssim \frac{1}{\epsilon} \int_S^{s_{\sigma}} e^{r \left[ \frac{6}{1+\si} -2 (1-d\sigma) \right]} \frac{e^{- 4 \frac{s_{\sigma}-r}{1+\sigma}}}{\Bigl( 1 - e^{- 4 \frac{s_{\sigma}-r}{1+\sigma}} \Bigr)^{1/2}} \sqrt{\widetilde V_\sigma (r)} \dd r \\
			&\lesssim \frac{1}{\epsilon} e^{s_{\sigma} \left[ \frac{6}{1+\si} -2 (1-d\sigma) \right]} \int_S^{s_{\sigma}} \frac{e^{- 4 \frac{s_{\sigma}-r}{1+\sigma}}}{\Bigl( 1 - e^{- 4 \frac{s_{\sigma}-r}{1+\sigma}} \Bigr)^{1/2}} \dd r \\
			&\lesssim \frac{1}{\epsilon} \frac{\sigma+1}{2}  e^{s_{\sigma} \left[ \frac{6}{1+\si} -2 (1-d\sigma) \right]} \Bigl( 1 - e^{- 4 \frac{s_{\sigma}-S}{1+\sigma}} \Bigr)^{1/2} \\
& \lesssim \frac{1}{\epsilon}  e^{s_{\sigma} \left[ \frac{6}{1+\si} - 4 (1-d\sigma) \right]} \Bigl( e^{\frac{4 s_{\sigma}}{1+\sigma}} - e^{\frac{4 S}{1+\sigma}} \Bigr)^{1/2}.
		\end{aligned}
\end{aligned}
\end{equation}
Taking $S$ such that $e^{\frac{4 s_{\sigma}}{1+\sigma}} - e^{\frac{4 S}{1+\sigma}} = 1$, estimate \eqref{eq:third_bound_E_2^11} multiplied by the factor $e^{-6\frac{s_{\si}}{1+\si}}$ leads to the bound
%\begin{multline*}
%	e^{-6\frac{s_{\si}}{1+\si}} \int_0^{s_{\sigma}} e^{r \left[ \frac{6}{1+\si} -2 (1-d\sigma) \right]} \left(1 - e^{- 4 \frac{s_{\sigma}-r}{1+\sigma}} \right) \left| \int_{\R^d} \nabla \varphi_{\epsilon}   \cdot \Delta e^{-\frac{s_{\sigma}-r}{1+\sigma} L} \mu_{\sigma} (r) \right|  \dd r \\
%	\lesssim e^{- 4 s_{\sigma} (1 - d \sigma)}
%\end{multline*}
\begin{multline*}	
e^{-6\frac{s_{\si}}{1+\si}} \int_0^{S} e^{r \left[ \frac{6}{1+\si} -2 (1-d\sigma) \right]} e^{- 4 \frac{s_{\sigma}-r}{1+\sigma}} \left| \int_{\R^d} \nabla \varphi_{\epsilon}   \cdot \Delta e^{-\frac{s_{\sigma}-r}{1+\sigma} L} \mu_{\sigma} (r) \right|  \dd r \\
\begin{aligned}
&	\lesssim e^{-4 s_\si   (1-d\si)+2\frac{s_\si}{1+\si}}  = e^{-2s_\si  (1-d\si)} e^{2 \si s_\si \frac{d \si + d - 1}{1 + \si}} \\
& \lesssim e^{-2s_\si  (1-d\si)} e^{2\sigma s_\si^{\max} \frac{d\si+d-1}{1+\si}} = e^{-2s_\si  (1-d\si)} (d \sigma)^{- \sigma\frac{d\si+d-1}{4(1+\si)(1 - d \si)}} \\
& \lesssim  e^{-2s_\si  (1-d\si)},
\end{aligned}
\end{multline*}
while \eqref{eq:fourth_bound_E_2^11} leads to
\[	e^{-6\frac{s_{\si}}{1+\si}} \int_S^{s_{\sigma}} e^{r \left[ \frac{6}{1+\si} -2 (1-d\sigma) \right]} e^{- 4 \frac{s_{\sigma}-r}{1+\sigma}} \left| \int_{\R^d} \nabla \varphi_{\epsilon}   \cdot \Delta e^{-\frac{s_{\sigma}-r}{1+\sigma} L} \mu_{\sigma} (r) \right|  \dd r	\lesssim \frac{1}{\epsilon} e^{- 4 s_{\sigma} (1 - d \sigma)}.\]
Summing these last two estimates with \eqref{eq:second_bound_E_2^11} multiplied by $e^{-6\frac{s_{\si}}{1+\si}}$ leads, in view of \eqref{eq:bound_E_2^11}, to the bound
\begin{equation*}
	| E_2^{(1,1)} | \lesssim e^{- 2 s_{\sigma} (1 - d \sigma)} + \frac{1}{\epsilon} e^{- 4 s_{\sigma} (1 - d \sigma)}.
\end{equation*}
We now turn to $E_2^{(1,2)}$: invoking \eqref{eq:d_s-tilde-rho},
 integrating by parts and reordering terms, we find
 \begin{align*}
   E_2^{(1,2)}
   & = \frac{2}{1-d\si} \int_0^{s_{\sigma}} e^{-2 r(1-d\sigma)}
     \left(1 - d \sigma e^{4 r(1-d\sigma)} \right)^{1/2}
     e^{-2\frac{s_{\si}-r}{1+\si}} \\
     & \qquad \qquad \qquad \qquad \qquad \qquad \qquad \qquad \times \int_{\R^d} \nabla \varphi_{\epsilon}\cdot
     e^{-\frac{s_{\sigma}-r}{1+\sigma} L} (y \diver \mu_{\sigma}(r))
     \dd r \\ 
 & = - \frac{2}{1-d\si} \int_0^{s_{\sigma}} e^{-2 r(1-d\sigma)}
   \left(1 - d \sigma e^{4 r(1-d\sigma)} \right)^{1/2}
   e^{-4\frac{s_{\si}-r}{1+\si}} \\
   & \qquad \qquad \qquad \qquad \qquad \qquad \qquad \qquad \times \int_{\R^d} \nabla \varphi_{\epsilon}
   \cdot \diver  e^{-\frac{s_{\sigma}-r}{1+\sigma} L}
   (\mu_{\sigma}(r)\otimes y
   )  \dd r \\ 
 & \quad - \frac{2d}{1-d\si} \int_0^{s_{\sigma}} e^{-2 r(1-d\sigma)}
   \left(1 - d \sigma e^{4 r(1-d\sigma)} \right)^{1/2}
   e^{-2\frac{s_{\si}-r}{1+\si}} \\
   & \qquad \qquad \qquad \qquad \qquad \qquad \qquad \qquad \times \int_{\R^d} \nabla \varphi_{\epsilon}
   \cdot  e^{-\frac{s_{\sigma}-r}{1+\sigma} L}  \mu_{\sigma}(r)  \dd r
   \\ 
 & =: E_2^{(1,2,1)}+E_2^{(1,2,2)},
 \end{align*}
 which allows to define two new error terms $E_2^{(1,2,1)}$ and $E_2^{(1,2,2)}$ that we need to estimate. The second one is straightforward from the same techniques as above, using \eqref{eq:bound_mu_sigma} and \eqref{eq:integrable_estimate_s}, so that
 \[ \left\rvert E_2^{(1,2,2)}\right\rvert \lesssim  e^{-\frac{2 s_{\si}}{1+\si}} \left( \int_0^{\infty} e^{-4r\left(2-2d\si-\frac{1}{1+\si} \right)}  \dd r \right)^{1/2}, \]
 where the last integral is convergent since, as we have seen before,
 the assumption $\si <
 1/(2d)$ implies $2-2d\si-\frac{1}{1+\si} >\frac{1}{2d+1}$. 
 For the first one we rather use the identity
 \[ y \cdot\mu_{\sigma} = \IM \left(
     \frac{1}{\tau_{\sigma}^{1-\frac{d\si}{2}}}  \overline{v_{\si}}
     y\cdot \nabla v_{\si}   \right)  , \]
 and that $\|y v_{\sigma} \|_{L^2} \le C$ from
 Lemma~\ref{lem:apriori-unif-si-t}, along with Lemma
 \ref{lem:fokker_planck_trade} with $m=1$, so that 
 \begin{align*}
 \left\lvert E_2^{(1,2,1)}\right\rvert & \lesssim \int_0^{s_{\sigma}}  e^{-2 r(1-d\sigma)} \underbrace{\left(1 - d \sigma e^{4 r(1-d\sigma)} \right)^{1/2}}_{\le 1} \frac{e^{-4\frac{s_{\si}-r}{1+\si}}}{\(1 - e^{-4\frac{s_{\si}-r}{1+\si}}\)^{1/2}} \| y \mu_{\si}(r) \|_{L^1} \dd r \\
 & \lesssim \int_0^{s_{\sigma}}  e^{-2 r(1-d\sigma)}
   \frac{e^{-4\frac{s_{\si}-r}{1+\si}}}{\(1 -
   e^{-4\frac{s_{\si}-r}{1+\si}}\)^{1/2}} \|y v_{\sigma} \|_{L^2}
   \sqrt{\widetilde V_\si(r)} \dd r \\ 
 & \lesssim \int_0^{s_{\sigma}}  e^{-2 r(1-d\sigma)} \frac{e^{-4\frac{s_{\si}-r}{1+\si}}}{\(1 - e^{-4\frac{s_{\si}-r}{1+\si}}\)^{1/2}} \dd r \\
& \lesssim \int_0^{s_{\sigma}}  e^{-2 (s_{\sigma} - q) (1-d\sigma)} \frac{e^{-\frac{4q}{1+\si}}}{\(1 - e^{-\frac{4q}{1+\si}}\)^{1/2}} \dd q \\
 & \lesssim e^{-2 s_{\sigma} (1-d\sigma)} \int_0^{s_{\sigma}} \frac{e^{-2q \left(\frac{2}{1+\si} - (1-d\sigma) \right)}}{\(1 - e^{-\frac{4q}{1+\si}}\)^{1/2}} \dd q \\
 & \lesssim e^{-2 s_{\sigma} (1-d\sigma)}.
 \end{align*}
 We can now go back to $E_2^{(1,3)}$ and $E_2^{(1,4)}$, which are straightforward in view of the above analysis: as
 \[ E^{(1,3)}_2=\frac{1}{1-d\si}  e^{-2 s_{\si}(1-d\sigma)} \left(1 - d \sigma e^{4 s_{\si}(1-d\sigma)} \right)^{1/2} \int_{\R^d} \varphi_{\epsilon} \diver \mu_{\si} \dd r,\]
 we directly get 
 \[ \left\lvert  E^{(1,3)}_2\right\rvert \lesssim e^{-2 s_{\si} (1-d\si)}, \]
 and similarly, writing
 \[  E^{(1,4)}_2= - \frac{1}{1-d\si} e^{-2\frac{s_{\si}}{1+\si}} \left(1 - d \sigma  \right)^{1/2} \int_{\R^d} \nabla   \varphi_{\epsilon} \cdot e^{\frac{s_{\si}}{1+\si}L} \mu_{\si}(0), \]
 from Lemma~\ref{lem:fokker_planck_trade}, we infer
  \[ \left\lvert E^{(1,4)}_2\right\rvert \lesssim e^{-\frac{2 s_{\si}}{1+\si}}. \]
  Gathering the bounds on $E_0^{(1)}$, $E_0^{(2)}$, $E_1$, $E_3$, $E_4$, $E_2^{(1,1)}$, $E_2^{(1,2,1)}$, $E_2^{(1,2,2)}$, $E_2^{(1,3)}$ and $E_2^{(1,4)}$, we get the final estimate
%  \[ W_1(\widetilde{\rho}_{\si}(s_{\si}),\Gamma) \lesssim \epsilon + \si + 3 e^{-\frac{2 s_{\si}}{1+\si}} +\frac{1}{\epsilon^3} e^{-\frac{8 s_{\si}}{1+\si}}+ \frac{1}{\epsilon} e^{-4 s_{\si} (1 - d \sigma)} + \textcolor{blue}{\frac{1}{\epsilon^2} e^{-4 s_{\si} (1 - d \sigma)}} + 2 e^{-2 s_{\si}(1-d\si)}.\]
%
%  
%
%  Or, with the good estimate,
  \begin{multline*}
  	\int_{\R^d} \varphi_{\epsilon} (\tilde{\rho}_{\sigma}(s_{\sigma})-e^{s_{\sigma}\frac{L}{1+\sigma}} \widetilde{\rho}_{\sigma}(0)) \\
  	\lesssim \epsilon + \si + 3 e^{-\frac{2 s_{\si}}{1+\si}} +\frac{1}{\epsilon^3} e^{-\frac{8 s_{\si}}{1+\si}}+ \frac{2}{\epsilon} e^{-4 s_{\si} (1 - d \sigma)} + 4 e^{-2 s_{\si}(1-d\si)} .
  \end{multline*}
 We optimize this estimate by setting  $\epsilon = e^{-2 s_{\sigma} (1 - d
   \sigma)}$, which yields
  \begin{equation*}
  	\int_{\R^d} \varphi_{\epsilon} (\tilde{\rho}_{\sigma}(s_{\sigma})-e^{s_{\sigma}\frac{L}{1+\sigma}} \widetilde{\rho}_{\sigma}(0)) \lesssim \sigma +
        e^{-2s_\si  (1-d\si)}, 
  \end{equation*}
%  We finally use Lemma \ref{lem:fokker_planck_to_gaussian} to conclude that
%  \begin{align*}
%  W_1 \left(\widetilde{\rho}_{\si}(s_{\si}),e^{s_{\sigma} \frac{L}{1+\sigma}} \widetilde{\rho}_{\si}(0) \right) &  \le W_1(\widetilde{\rho}_{\si}(s_{\si}), \Gamma) + W_1(\Gamma, e^{s_{\sigma} \frac{L}{1+\sigma}} \widetilde{\rho}_{\si}(0)) \\
%  & \lesssim \sigma + e^{-2 s_{\sigma}(1-d\sigma)}+ e^{-2 \frac{s_{\sigma}}{1+\sigma}} W_2(\Gamma, \widetilde{\rho}_{\si}(0)) \\
%  & \lesssim \sigma + e^{-2 s_{\sigma}(1-d\sigma)},
%  \end{align*}
%invoking that $W_2(\Gamma, \widetilde{\rho}_{\si}(0))$ is uniformly bounded with respect to $\sigma$ and using that $e^{-2 s_{\sigma}(\frac{1}{1+\sigma} - 1 + d\sigma)}\le 1$, 
which ends the proof of Theorem~\ref{theo:cv-log-W1}.

\subsection{Conclusion: proof of Theorem~\ref{theo:log-temps-long}}
\subsubsection{Convergence in Wasserstein distance \eqref{eq:convergence_thm}}

Recall that $\Gamma(x)=\frac{e^{-|x|^2}}{(2\pi)^{d/2}}$, so  
\[ \int_{\R^d} \widetilde \varrho_{\sigma}(s_{\sigma})=\frac{1}{\|
    \phi_{\sigma}\|_{L^2}^2}\int_{\R^d}\tilde{\rho}_{\sigma}(s_{\sigma})
  = 1 =\int_{\R^d} \Gamma , \] 
with the obvious notation $\widetilde
\varrho_{\sigma}(s_{\sigma})=\varrho_{\sigma}(s_{\sigma}(t))$. Using
Theorem \ref{theo:cv-log-W1} (as the same remains valid for
$\widetilde \varrho_{\sigma}$) alongside Lemma
\ref{lem:fokker_planck_to_gaussian}, we can then write that 
\begin{align*}
 W_1(\widetilde \varrho_{\sigma}(s_{\sigma}), \Gamma) & \le W_1(\widetilde \varrho_{\sigma}(s_{\sigma}),e^{\frac{s_{\sigma}}{1+\sigma} L}
  \tilde{\rho}_{\sigma}(0) )  + W_1(e^{\frac{s_{\sigma}}{1+\sigma} L}  \tilde{\rho}_{\sigma}(0), \Gamma)  \\
  & \lesssim  \sigma +
        e^{-2s_\si  (1-d\si)}  + e^{-\frac{2 s_{\sigma}}{1+\sigma}}  W_2(\tilde{\rho}_{\sigma}(0),\Gamma).
  \end{align*}

%The linear term can be gathered
%with the expected Gaussian limit term, so that
%Lemma~\ref{lem:fokker_planck_to_gaussian} yields 
%\begin{align*}
%  \int_{\R^d} \varphi_{\epsilon} (e^{\frac{s_{\sigma}}{1+\sigma} L}
%  \tilde{\rho}_{\sigma}(0)-\Gamma ) \le
%  W_1(e^{\frac{s_{\sigma}}{1+\sigma} L}
%  \tilde{\rho}_{\sigma}(0),\Gamma) & \le W_2
%  (e^{\frac{s_{\sigma}}{1+\sigma} L} \tilde{\rho}_{\sigma}(0),\Gamma)\\
%&  \le e^{-\frac{2 s_{\sigma}}{1+\sigma}}
%  W_2(\tilde{\rho}_{\sigma}(0),\Gamma), 
%\end{align*}
%which ensures the convergence.
Going back to the time variable $t$, we then get that for all $t\ge 2$,
\[ W_1(\varrho_{\sigma}(t),\Gamma) \lesssim \sigma + \frac{1}{\dot{\tau}_{\sigma}(t)}.   \]

On the other hand, we recall that from \cite[Theorem~1.5]{Ferriere2021}, there exists
$C>0$ such that for all $t\ge 2$,
\begin{equation*}
  W_1\(\varrho_0(t),\Gamma\)\le \frac{C}{\sqrt{\ln t}}. 
\end{equation*}

We finally need a last property on the dispersive rate to conclude:
\begin{lemma} \label{lem:decreasing_bound_W1}
For all $\epsilon >0$, there exists $T>0$ and $\sigma_0>0$ such that for all $0 < \sigma \le \sigma_0$ and for all $t \ge T$, $\dot{\tau}_{\sigma}(t)^{-1} \le \epsilon$.
\end{lemma}
\begin{proof}
We know from Proposition \ref{prop:cvODE} that there exists a constant
$C>0$ such that 
\[ \dot{\tau}_0(t) -C \sigma (\ln t)^{3/2} \le \dot{\tau}_{\sigma}(t)   \]
for all $t \ge 2$. As $\dot{\tau}_0(t) \sim \sqrt{\ln t}$ as $t \to +\infty$, there exists $T=T(\epsilon)>0$ large enough such that 
\[ \dot{\tau}_0(T) \ge 2 \epsilon^{-1}.  \]
On the other hand, for $\sigma_0=\sigma_0(T,\epsilon)>0$ small enough we have
\[  C \sigma (\ln T)^{3/2} \le \epsilon^{-1} ,\quad \forall
  \si\le \si_0.\]
We then deduce that $\dot{\tau}_{\sigma}(T) \le \epsilon^{-1}$, and as $\dot{\tau}$ is a non-decreasing function since $\ddot{\tau}_{\sigma} \ge 0$, we get the result.
\end{proof}

We now fix $\epsilon>0$. We take $T>0$ and $\sigma_0>0$ given by
Lemma~\ref{lem:decreasing_bound_W1} so that  
\[ W_1( \varrho_{\sigma}(t), \varrho_{0}(t) )  \lesssim \sigma +\frac{1}{\dot{\tau}_{\sigma}(t)} + \frac{1}{\sqrt{\ln t}} \le \epsilon , \]
for all $ t \ge T$ and $\sigma\le\sigma_0$. On the other hand, using Theorem \ref{theo:log-loc-temps}, there exists $\sigma_1>0$ such that
\[ W_1( \varrho_{\sigma}(t), \varrho_{0}(t) ) \lesssim \sqrt{C_1 \sigma + \| \phi_{\sigma}-\phi_0 \|_{L^2} } e^{\frac{C_0 T}{2}} \le \epsilon, \]
for all $0 \le t \le T$ taking $\sigma \le \min(\sigma_0,\sigma_1)$
small enough. This ends the proof of
Theorem~\ref{theo:log-temps-long}, up to the final claim concerning
the convergence rate.

\subsubsection{Towards the convergence rate \eqref{eq:convergence_rate_thm}}

In this section we perform a refinement of the previous analysis in
order to obtain a convergence rate for the quantity of interest
$\sup_{t \ge 0} W_1(\varrho_{\sigma}(t),W_1(\varrho_0(t))$, assuming
that $\| \phi_{\sigma} - \phi_0 \|_{L^2}=\O(\si)$. First we give a
more precise asymptotic expansion for $\dot{\tau}_{\sigma}^{-1}$. 

\begin{lemma} \label{lem:rate_dot_tau}
There exists a constant $C>0$ such that for all $t \ge 2$, we have
\[ \left| \frac{1}{\dot{\tau}_{\sigma}(t)} - \sqrt{d \sigma} \right| \le C \frac{1}{\sqrt{\ln t}}. \]
\end{lemma}
\begin{proof}
We first recall that $\tau_0(t) \ge \tau_{\sigma}(t)$ for all $t \ge 0$ hence
\[ \dot{\tau}_{\sigma}^2 = \frac{1}{d \sigma} \left(1 - \frac{1}{\tau_{\sigma}^{d\sigma}} \right) \le  \frac{1}{d \sigma} \left(1 - \frac{1}{\tau_{0}^{d\sigma}} \right) \le \ln \tau_0 = \dot{\tau}_{0}^2. \]
Denote $X_{\sigma}(t)=\dot{\tau}_{\sigma}(t)^{-1} - \sqrt{d \sigma}\ge
0$, so we have
\begin{align*}
\dot{X}_{\sigma} & = - \frac{\ddot{\tau}_{\sigma}}{\dot{\tau}_{\sigma}^2} = - \frac{1}{2 \dot{\tau}_{\sigma}^2 \tau_{\sigma}^{1+d\sigma}} = - \frac{1}{2 \dot{\tau}_{\sigma}^2 \tau_{\sigma}}(1-d \sigma \dot{\tau}_{\sigma}^2) = - \frac{\dot{\tau}_{\sigma}^{-1} + \sqrt{d \sigma}}{2 \tau_{\sigma}} X_{\sigma} \\
& \le - \frac{\dot{\tau}_{0}^{-1} + \sqrt{d \sigma}}{2 \tau_{0}} X_{\sigma} = - \ddot{\tau}_0 (\dot{\tau}_{0}^{-1} + \sqrt{d \sigma}) X_\si,
\end{align*}
as $\dot{\tau}_0 \ge \dot{\tau}_{\sigma}$. We then compute, for any $t_0 >1$,
\begin{align*}
X_{\sigma}(t) & \le X_{\sigma}(t_0) \exp \left( -\int_{t_0}^t  \ddot{\tau}_0(s) (\dot{\tau}_{0}(s)^{-1} + \sqrt{d \sigma} )\dd s \right) \\
& =  X_{\sigma}(t_0) \exp \left( - \left[ \ln (\dot{\tau}_0(s)) + \sqrt{d \sigma} \dot{\tau}_0(s) \right]_{t_0}^t \right) \\
& \le X_{\sigma}(t_0) \dot{\tau}_0(t_0) e^{\sqrt{d \sigma} \dot{\tau}_0(t_0)} \frac{1}{\dot{\tau}_0(t)} e^{-\sqrt{d \sigma} \dot{\tau}_0(t)},
\end{align*}
which provides the announced bound recalling that $\dot{\tau}_0(t) \sim \sqrt{\ln t}$.
\end{proof}
The convergence rate from Lemma~\ref{lem:rate_dot_tau} then ensures that there exists $C_1>0$ such that
\begin{equation} \label{eq:estimate_W1_varrho_1}
 W_1( \varrho_{\sigma}(t), \varrho_{0}(t) ) \le C_1\left( \frac{1}{\sqrt{\ln t}} + \sqrt{\sigma} \right) ,
 \end{equation}
for all $t \ge 2$ from the previous analysis. On the other hand, for any 1-Lipschitz function $\varphi$,
\begin{multline*}
\int_{\R^d} \varphi(y) (\varrho_{\sigma}(t,y)-  \varrho_{0}(t,y)) \dd y  \\
\begin{aligned}
& =  \int_{\R^d} \varphi(y) \left( \|\phi_{\sigma} \|_{L^2}^{-2} \tau_{\sigma}(t)^d | \u_{\sigma}(t,y\tau_{\sigma}(t))|^2 -  \|\phi_{0} \|_{L^2}^{-2} \tau_{0}(t)^d | \u_{0}(t,y\tau_{0}(t))|^2  \right) \dd y \\
& = \left( \| \phi_{\sigma} \|_{L^2}^{-2} - \| \phi_{0} \|_{L^2}^{-2} \right)   \int_{\R^d} \varphi(y) \tau_{\sigma}(t)^d | \u_{\sigma}(t,y\tau_{\sigma}(t))|^2 \dd y \\
& \quad + \| \phi_{0} \|_{L^2}^{-2} \int_{\R^d} \varphi(y) \left( \tau_{\sigma}(t)^d | \u_{\sigma}(t,y\tau_{\sigma}(t))|^2 -
\tau_{0}(t)^d | \u_{0}(t,y\tau_{0}(t))|^2  \right) \dd y.
\end{aligned} 
\end{multline*}   
Since $\varphi$ is $1$-Lipschitz, $\varphi(y)\le |\varphi(0)|+|y|$. It
is clear that the first term on the right hand side is 
$\O(\sigma)$ from the assumption $\|\phi_\si-\phi_0\|_{L^2}=\O(\si)$, and using the
boundedness of the second momenta \eqref{eq:bound-rho-2-mom}.
Resuming
the computations from Section~\ref{sec:ehrenfest-wasserstein},
\begin{align*}
  \int_{\R^d}
  &\left|  \tau_{\sigma}(t)^d | \u_{\sigma}(t,y\tau_{\sigma}(t))|^2 -
\tau_{0}(t)^d | \u_{0}(t,y\tau_{0}(t))|^2  \right| \dd y \\
  &\quad
    \lesssim \(1+\(\frac{\tau_0(t)}{\tau_\si(t)}\)^{d/2}\)\( \sigma
    \(\ln (t+4)\)^{3/2}  + \sigma e^{C_0 t}\).
\end{align*}
We have a similar estimate for
\begin{equation*}
 \int_{\R^d}
  |y| \left|  \tau_{\sigma}(t)^d | \u_{\sigma}(t,y\tau_{\sigma}(t))|^2 -
\tau_{0}(t)^d | \u_{0}(t,y\tau_{0}(t))|^2  \right| \dd y ,
\end{equation*}
by distributing the weight $|y|$ as follows when invoking Cauchy-Schwarz
inequality,
\begin{equation*}
  \int_{\R^d} |y|\times\left| |f(y)|^2-|g(y)|^2\right|\dd y\le \( \|y
    f\|_{L^2}+\|yg\|_{L^2}\)\|f-g\|_{L^2}, 
  \end{equation*}
  and using the boundedness of momenta in $L^2$, \eqref{eq:bound-rho-2-mom}. 
Up to changing the constants, we infer that there exist $C_0$ and $C_2>0$ such
that for all $t \ge 0$, 
\begin{equation} \label{eq:estimate_W1_varrho_2}
W_1( \varrho_{\sigma}(t), \varrho_{0}(t) ) \le C_2 \sigma  e^{C_0 t}.
\end{equation} 
We interpolate between estimates \eqref{eq:estimate_W1_varrho_1} and \eqref{eq:estimate_W1_varrho_2}. Denote by $t_{\sigma} \ge 2$ the unique time such that we have the equality
\[  C_1\left( \frac{1}{\sqrt{\ln t_{\sigma}}} + \sqrt{\sigma} \right) = C_2 \sigma  e^{C_0 t_{\sigma}}. \]
Writing $Y=\sqrt{\sigma}$, we consider the quadratic equation
\[ C_2 e^{C_0 t_{\sigma}} Y^2 - C_1 Y - \frac{C_1}{\sqrt{\ln t_{\sigma}}} = 0 ,  \]
whose only positive root is
\[ Y_+=  \frac{C}{2} e^{-C_0 t_{\sigma}} + \frac12 \left( C^2e^{-2C_0
      t_{\sigma}}  + 4 C \frac{e^{-C_0 t_{\sigma}}}{\sqrt{\ln
        t_{\sigma}}}  \right)^{1/2},\quad\text{where }C=C_1/C_2. \]
With rough upper and lower estimates, we infer that $Y_+^2=\si$ satisfies
\[ c_1 \frac{e^{-C_0 t_{\sigma}}}{\sqrt{\ln t_{\sigma}}} \le \sigma \le c_2  \frac{e^{-C_0 t_{\sigma}}}{\sqrt{\ln t_{\sigma}}}, \]
for some constants $c_1$, $c_2>0$. The lower bound provides that $\ln \ln (\sigma^{-1}) \le c_3 \ln t_{\sigma}$ for a constant $c_3>0$, which combined with the upper bound allows to write that 
\[  e^{C_0 t_{\sigma}} \le \frac{c_2 \sqrt{c_3}}{\sigma \sqrt{\ln \ln \frac{1}{\sigma}} } .\]
Plugging such estimate into \eqref{eq:estimate_W1_varrho_2} with $t=t_{\sigma}$ gives the result.

\subsection*{Acknowledgments} The authors are grateful to Fran\c cois
Bolley for helpful discussions on Wasserstein distance.  
Q.C. and G.F. acknowledge the support of the CDP C2EMPI, together with
the French State under the France-2030 programme, the University of
Lille, the Initiative of Excellence of the University of Lille, the
European Metropolis of Lille for their funding and support of the
R-CDP-24-004-C2EMPI project.

\bibliographystyle{abbrv}
\bibliography{continuity}

\end{document}